\documentclass[a4paper,twoside,10pt]{amsart}
\usepackage{amsmath}
\usepackage{a4wide}
\usepackage{amsthm}
\usepackage{wasysym}
\usepackage{mathrsfs}
\usepackage[hidelinks]{hyperref}
\usepackage{amssymb}
\usepackage{arydshln}
\usepackage{enumitem}
\usepackage{mathtools}
\usepackage{etoolbox}
\usepackage{mathtools}
\def\ebinom#1#2{\binom{#1}{{#2}/\mathcal E}}
\def\lexleq{\mathop{\leq_{\mathrm{lex}}}}

\def\fat#1{\mathbf{#1}}
\def\Emil{\mathcal E^{\mathrm{Mil}}}
\def\Ecs{\mathcal E^{\mathrm{CS}}}
\def\Eqn{\mathscr {E}}
\def\Egr{\mathcal E^{\mathrm{GR}}}
\def\Mmil{\mathcal M_{\mathcal E^{\mathrm{Mil}}}}
\def\Mcs{\mathcal M_{\mathcal E^{\mathrm{CS}}}}
\def\Mgr{\mathcal M_{\mathcal E^{\mathrm{GR}}}}
\def\fatnn#1#2{T_\fat{#1}({#2})}
\def\Str{\mathrm{Str}}
\def\fatn#1{T_\fat{#1}}
\def\comp{\mathbin{\text{\bf\smiley{}}}}
\def\fatl#1{\tilde{\fat{#1}}}
\def\fatc#1{\mathbf{\dddot{\fat{#1}}}}
\def\fatf#1{F_\fat{#1}}
\def\fatfn#1#2{F_\fat{#1}^{#2}}

\def\oleq{\leq^\circ}
\def\finleq{\leq_{\mathrm{fin}}}
\def\lift#1#2#3{\mathrm{Lift}_{\fat{#1}}(#2,\allowbreak #3)}
\def\liftp#1#2#3{\mathrm{Lift}_{\fat{#1}'}(#2,\allowbreak #3)}

\def\depth{\mathrm{depth}}
\def\conc{^\frown}

\DeclareMathSymbol{\ORDrestriction}{\mathord}{AMSa}{"16}
\renewcommand{\restriction}{\ORDrestriction\nolinebreak}
\def\meet{\wedge}
\def\typev{\mathrm{type}}
\def\S{\mathcal{S}}
\def\M{\mathcal{M}}
\def\AM{\mathcal{AM}}
\def\Dp{\mathrm{Dp}}
\def\Cl{\mathrm{Cl}}
\def\ind#1#2{\iota(#1,#2)}
\def\OAge{\mathrm{OAge}}
\def\OEmb{\mathrm{OEmb}}
\def\Emb{\mathrm{Emb}}
\def\type#1{{\str{#1}^+}}
\def\freef#1{\mathrm{fl}_\type{#1}}
\def\Dc{\mathrm{Dc}}
\def\Stree{(T,\allowbreak {\preceq},\allowbreak \Sigma,\allowbreak \S)}
\def\SNtree{(T,\allowbreak {\preceq},\allowbreak \Sigma,\allowbreak \S,\allowbreak \M)}
\def\str#1{\mathbf {#1}}
\def\pstr#1{{\mathbf {#1}^+}}

\def\Gp{{\mathbf G_\omega}}
\def\nbrel#1#2{R\ifstrempty{#1}{}{_{#1}}\ifstrempty{#2}{}{^{#2}}}
\def\rel#1#2{\nbrel{\ifstrempty{#1}{}{\str{#1}}}{#2}}
\def\KF{\mathcal K^\mathcal F}
\def\KFpt{\mathcal{K}^\mathcal{F}_{\mathrm{pt}}}
\def\KFt{\mathcal K^\mathcal F_\mathrm{t}}

\def\embGG{\Phi_{\str{G},\Gp}}
\def\embKG{\Phi_{\str{K},\str{G}}}
\def\embKH{\Phi_{\str{K},\str{H}}}
\def\embHG{\Phi_{\pstr{H},\pstr{G}}}
\def\embF{\Phi_{F}}
\def\embGpK{\Phi_{\Gp,\str{K}}}
\newif\ifsimpletheorem
\simpletheoremtrue
\newif\ifcomplextheorem
\complextheoremtrue

\theoremstyle{plain}
\newtheorem{theorem}{Theorem}[section]
\newtheorem{corollary}[theorem]{Corollary}
\newtheorem{prop}[theorem]{Proposition}
\newtheorem{observation}[theorem]{Observation}
\newtheorem{lemma}[theorem]{Lemma}

\newtheorem{claim}[theorem]{Claim}

\theoremstyle{definition}
\newtheorem{definition}[theorem]{Definition}
\newtheorem{algorithm}[theorem]{Algorithm}
\newtheorem{example}[theorem]{Example}
\newtheorem*{remark*}{Remark}
\newtheorem*{claim*}{Claim}

\theoremstyle{remark}
\newtheorem{remark}[theorem]{Remark}
\begin{document}
\title[Trees with successor operation]{Ramsey theorem for trees with successor operation}
\author[M.B.]{Martin Balko}
\author[D.Ch.]{David Chodounsk\'y}
\address{Department of Applied Mathematics (KAM), Charles University, Ma\-lo\-stranské~nám\v estí 25, Praha 1, Czech Republic}
\email{\{balko,chodounsky,hubicka\}@kam.mff.cuni.cz}
\author[N.D.]{Natasha Dobrinen}
\address{Department of Mathematics, University of Notre Dame, 255 Hurley Bldg, Notre Dame, IN 46556 USA}
\email{ndobrine@nd.edu}
\author[J.H.]{Jan Hubi\v cka}
\author[M.K.]{Mat\v ej Kone\v cn\'y}
\address{Institute of Algebra\\TU Dresden\\ Dresden, Germany}
\email{matej.konecny@tu-dresden.de}
\author[J.N.]{Jaroslav Ne\v set\v ril}
\address{Computer Science Institute of Charles University (IUUK), Charles University, Ma\-lo\-stranské~nám\v estí 25, Praha 1, Czech Republic}
\email{nesetril@kam.mff.cuni.cz}
\author[A.Z.]{Andy Zucker}
\address{Department of Pure Mathematics, University of Waterloo, 200 University Ave W, Waterloo, ON N2L 3G1, Canada}
\email{a3zucker@uwaterloo.ca}
\begin{abstract}
	We prove a general Ramsey theorem for trees with a successor operation. This theorem is a common generalization of the Carlson--Simpson Theorem and the Milliken Tree Theorem for regularly branching trees.

	Our theorem has a number of applications both in finite and infinite combinatorics. For example, we give a short proof of the unrestricted Nešetřil--Rödl theorem, and we recover the Graham--Rothschild theorem. Our original motivation came from the study of big Ramsey degrees -- various trees used in the study can be viewed as trees with a successor operation. To illustrate this, we give a non-forcing proof of a theorem of Zucker on big Ramsey degrees.
\end{abstract}
\maketitle             
\tableofcontents
\section{Introduction}
Many well known Ramsey-type results speak about trees. Obviously, the Hal\-pern--L\"auchli theorem~\cite{Halpern1966} and the Milliken tree theorems~\cite{Milliken1979} are examples of this, as well as Ramsey's theorem itself is~\cite{Ramsey1930} (a chain is just a particularly simple tree). Various parameter space theorems (for example the Hales--Jewett theorem~\cite{Hales1963}, the Graham--Rothschild theorem~\cite{Graham1971}, or the Carlson--Simpson Theorem~\cite{carlson1984}) also speak about trees: Given a finite alphabet $\Sigma$, the set $\Sigma^{<\omega}$ of all finite words over $\Sigma$ forms a tree with the the order given by end-extension,  parameter words then describe special regular subtrees. Ellentuck's theorem~\cite{Ellentuck1974}, or Hindman's theorem~\cite{hindman1974}, Gowers' $\mathrm{FIN}_k$ theorem~\cite{gowers1992lipschitz}, and other block-sequence theorems are also examples of this phenomenon.

In this paper we prove a new abstract Ramsey theorem for finitely branching trees equipped with a successor operation, which, for regularly branching trees, is a common generalization of the Carlson--Simpson Theorem~\cite{carlson1984} and
Milliken's tree theorem~\cite{Milliken1979}, as well as a whole spectrum of theorems ``in-between.'' For trees with unbounded branching it provides a new form of Ramsey theorem for colouring isomorphic copies of a tree in itself. Even better, these trees form a topological Ramsey space, which we show by proving that certain closely related objects satisfy Todorcevic's Ramsey space axioms~\cite{todorcevic2010introduction}.

The original motivation for our work was the recent progress in the area of big Ramsey degrees, as each big Ramsey result turns out to give a Ramsey theorem for special finite branching trees. In the late 1960's, Laver produced upper bounds on big Ramsey degrees of the rationals which now can be seen as using the Halpern--L\"auchli
\cite{Halpern1966} theorem and Milliken's tree theorem~\cite{Milliken1979} (see e.g.~\cite{devlin1979}). Laver's technique has been
adapted to other cases and, until recently, it was the main tool used in the area (see
\cite{devlin1979,Laflamme2006,todorcevic2010introduction,Hubicka2020uniform}).  To prove bounds on
big Ramsey degrees of the homogeneous universal triangle-free
graph,~Dobrinen introduced a concept of coding trees~\cite{dobrinen2017universal}. The proof of the Ramsey theorem for coding
trees develops set-theoretic forcing on such trees,
building on Harrington's forcing proof of the Halpern--L\"auchli theorem.
Coding trees have been generalized to free
amalgamation classes in binary languages in a series of
papers~\cite{Balko2021exact,dobrinen2019ramsey,dobrinen2023infinite,zucker2020}.
Recently, Hubi\v cka~\cite{Hubicka2020CS} introduced a new connection between the Carlson--Simpson theorem~\cite{carlson1984} and big
Ramsey degrees, which has significantly simplified some of earlier results and made it possible to bound
big Ramsey degrees of partial orders and metric spaces~\cite{balko2021big,Balko2023}.

Our theorem aims to give an abstract and practical framework for phrasing and proving Ramsey tree theorems in this area, thereby avoiding the need for complex ad hoc arguments and proofs. We intend for it to help separate the structural part of the arguments from the Ramsey-theoretic part. To support our claims, we give a new proof of a theorem of Zucker on big Ramsey degrees of free amalgamation classes in finite binary languages~\cite{zucker2020} which is entirely elementary and does not require the arguments from set-theoretic forcing. A vast generalisation of this is in preparation.

We stress that our theorem is of independent interest and is not limited to the particular setup of big Ramsey degrees. To illustrate this, we give a simple proof of the unrestricted version of the Ne\v set\v ril--R\"odl~\cite{Nevsetvril1977} and Abramson--Harrington theorem~\cite{Abramson1978}, and prove a spectrum of Ramsey theorems for regularly branching trees whose extreme points are the Carlson--Simpson theorem~\cite{carlson1984} and Milliken's tree theorem~\cite{Milliken1979} for regularly branching trees. We also obtain the Dual Ramsey theorem, and the Graham--Rothschild theorem~\cite{Graham1971} as well as its $*$-version.

\subsection{Trees with a successor operation}

A \emph{tree} is a (possibly empty) partially ordered set $(T, \preceq)$ such
that, for every $a \in T$, the set $\{b \in T : b \prec a\}$ is finite and linearly ordered by $\preceq$.
Given a node $a\in T$, we denote by $\ell(a)=|\{b\in T:b\prec a\}|$ the \emph{level} of $a$. We put $T(n)=\{a\in T: \ell(a)=n\}$, $T({\leq}n)=\{a\in T: \ell(a)\leq n\}$ and analogously for $T({<}n)$, $T({\geq}n)$ and $T({>}n)$.
We call $b\in T$ an \emph{immediate successor} of $a\in T$ if $a\prec b$ and there is no $c\in T$ satisfying $a\prec c\prec b$.  A tree $T$ is \emph{finitely branching} if for every $a\in T$ the set of all immediate successors of $a$ is finite.

We equip trees with an additional operation that can be used to build all immediate successors of a given node of the tree.
\begin{definition}[$\S$-tree]
	\label{def:sucessor}
	An \emph{$\S$-tree} is a quadruple $(T,{\preceq},\Sigma,\S)$ where $(T,{\preceq})$ is a countable finitely branching tree with finitely many nodes of level 0, $\Sigma$ is a set called the \emph{alphabet} and $\S$ is a partial function $\S\colon T\times T^{<\omega}\times \Sigma\to T$ called the \emph{successor operation} satisfying the following three axioms:
	\begin{enumerate}[label=(S\arabic*)]
		\item\label{S1} If $\S(a,\bar p,c)$ is defined for some \emph{base} $a\in T$, \emph{parameter} $\bar p\in T^{<\omega}$ and \emph{character} $c\in \Sigma$, then $\S(a,\bar p,c)$ is an immediate successor of $a$ and all nodes in $\bar p$ have levels  at most $\ell(a)-1$. 
		\item\label{S2} \emph{Injectivity:} If $\S(a,\bar p,c)=\S(b,\bar q,d)$, then $a=b$, $\bar p=\bar q$ and $c=d$.
		\item\label{S3} \emph{Constructivity:} For every node $a\in T$ and its immediate successor $b$, there exist $\bar p\in T^{<\omega}$ and $c\in \Sigma$ such that $b=\S(a,\bar p,c)$.
	\end{enumerate}
\end{definition}
Note that $(T, \preceq)$ is uniquely determined from the operation $\S$ and the set of nodes of level 0.
If $\bar p$ is the empty sequence, then we will also write $\S(a,c)$ for brevity. 
\begin{example}
	\label{ex:1}
	Consider tree $(2^{{<}\omega},\sqsubseteq)$ of all finite $\{0,1\}$-sequences (or words) where  $a\sqsubseteq b$ if and only if $a$ is an initial segment of $b$. One can naturally view this as an $\S$-tree $(2^{{<}\omega},\sqsubseteq,\{0,1\},\S)$  where $\S$
    is defined only for empty parameters $\bar{p}$.
	For every $a\in B$ and  $c\in \{0,1\}$, we let $\S(a,c)=a\conc c$ where $a\conc c$ denotes the sequence $a$ extended by $c$.

	Observe that every finite sequence can be constructed from the empty sequence $()$
	using the successor operation.  For example, the sequence $a=01011$ can be
	constructed as
	\[\S(\S(\S(\S(\S((),0),1),0),1),1).\]
\end{example}
Given $S\subseteq T$, a function $f\colon S\to T$ is \emph{level-preserving} if for every $a,b\in S$ such that $\ell(a)=\ell(b)$ we also have $\ell(f(a))=\ell(f(b))$.
For a level-preserving function $f\colon S\to T$, we denote by $\tilde{f}$ the function $\tilde{f}\colon \ell(S)\to \omega$ defined
by $\tilde{f}(n)=\ell(f(a))$ for some $a\in S$ with $\ell(a)=n$.

We consider the following embeddings of trees.
\begin{definition}[Shape-preserving functions]
	\label{def:shape-pres}
	Let $(T,{\preceq},\Sigma,\S)$ be an $\S$-tree.
	We call an injection $F\colon T\to T$ a \emph{shape-preserving function} if 
	\begin{enumerate}[label=(\roman*)]
		\item\label{def:shape-pres:item1} $F$ is level preserving, and
		\item\label{def:shape-pres:item2} $F$ is \emph{weakly $\S$-preserving}: 
			For every $a\in T$, $\bar{p}\in T^{< \omega}$ and $c\in \Sigma$ such that $\S(a,\bar{p},c)$ is defined, we have
			\[\S(F(a),F(\bar{p}),c)\preceq F(\S(a,\bar{p},c))\]
			(in particular, $\S(F(a),F(\bar{p}),c)$ is defined). Here, $F(\bar{p})$ is the tuple consisting of images of entries of $\bar{p}$ under $F$.
		\item\label{def:shape-pres:item3} For every $a\in T(0)$ and $b$ such that $a\preceq b$ it also holds that $a\preceq F(b)$.
	\end{enumerate}
	Given $S\subseteq T$, we also call a function $f\colon S\to T$ \emph{shape-preserving} if it extends to a shape-preserving function $F\colon T\to T$.
\end{definition}
It is not hard to verify (see Proposition~\ref{prop:shape-pres}) that shape-preserving functions form a closed monoid, and every shape-preserving function $F$ is an embedding $F\colon (T,\preceq)\to(T,\preceq)$ where $(F[T],\preceq)$ is a meet- and level-preserving subtree of $(T,\preceq)$. 
\subsection{Ramsey theorem for shape-preserving functions}

\begin{definition}
	Let $\Stree$ be an $\S$-tree, $F\colon T\to T$ be a shape-preserving function and $m\geq 0$ a level.  We say that $F$ is \emph{skipping level $m$} if $m\notin \tilde{F}[\omega]$ and
	that $F$ is \emph{skipping only level $m$} if $\tilde{F}[\omega]=\omega\setminus \{m\}$.
\end{definition}

It follows from the definition of shape-preserving function that
$F$ skipping only level $m$ is uniquely defined by set $F[T(m)]$, as $F$ is the identity on $T({<}m)$, and $F(\S(a,\bar{p},c)) = \S(F(a),F(\bar{p}),c)$ whenever $\ell(\S(a,\bar{p},c)) > m$.

\begin{definition}[$(\S,\M)$-tree]
	\label{def:sntree}
	Given an $\S$-tree $\Stree$ and a monoid $\M$ of some shape-preserving functions $T\to T$, we call $\SNtree$ an \emph{$(\S,\M)$-tree} if the following three conditions are satisfied:
	\begin{enumerate}[label=(M\arabic*)]
		\item\label{item:Nmonoid}
		\emph{The set $\M$ forms a closed monoid}: $\M$ contains the identity function $\mathrm{Id}\colon T\to T$,
		for every $F,F'\in \M$ it holds that $F\circ F'\in \M$, and every sequence $(F_i)_{i\in \omega}$ of functions from $\M$
		satisfying
		$(\forall_{i\in \omega})F_i\restriction_{T({\leq}i)}=F_{i+1}\restriction_{T({\leq}i)}$ has a limit in $\M$, that is, a function $F_\infty \in \M$ such that $(\forall_{i\in \omega})F_\infty\restriction_{T({\leq}i)} = F_i\restriction_{T({\leq}i)}$. Equivalently, one can define the \emph{pointwise convergence topology} on the set of shape-preserving functions in the typical way, and then we demand that $\M$ is closed in this topology.
		\item\label{item:Ndecomposition} \emph{The set $\M$ admits decompositions}:
		For every $n\in \omega$ and $F\in \M$ skipping level $\tilde{F}(n)-1$ such that $\tilde{F}(n)>0$,
		there exist $F_1, F_2\in \M$ such that $F_2$ skips only level $\tilde{F}(n)-1$ and $$F_2\circ F_1\restriction_{T({\leq} n)}=F\restriction_{T({\leq} n)}.$$
		\item\label{item:Nduplication} \emph{The set $\M$ is closed for duplication}: For all $n$ and $m$ with $n<m\in \omega$, there exists a function $F_m^n\in \M$
		skipping only level $m$ such that
		for every $a\in T(n)$, $b\in T(m)$, $\bar{p}\in T^{<\omega}$ and $c\in \Sigma$, where $\S(a,\bar{p},c)$ is defined and $\S(a,\bar{p},c)\preceq b$,
		we have $F^n_m(b)=\S(b,\bar{p},c)$.
	\end{enumerate}
\end{definition}
\begin{figure}
	\centering
	\includegraphics[scale=0.833]{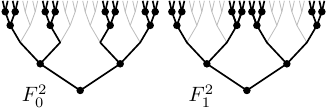}
	\caption{Duplication functions requried by \ref{item:Nduplication} on a binary tree.}
	\label{fig:duplication}
\end{figure}
\begin{figure}
	\centering
	\includegraphics[scale=0.833]{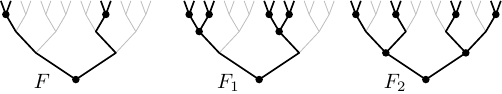}
	\caption{Decomposition of function $F$ given by \ref{item:Ndecomposition}.}
	\label{fig:decomposition}
\end{figure}
See Figures \ref{fig:duplication} and \ref{fig:decomposition} for examples of shape-preserving functions in \ref{item:Ndecomposition} and \ref{item:Nduplication}.

Assuming that~\ref{item:Nmonoid} holds, \ref{item:Ndecomposition} is equivalent with the property that every
function in $\M$ is the limit of functions constructed from the identity by composing functions from $\M$ skipping only one level.
Duplication is often an essential step in
proving the pigeonhole version of various dual Ramsey theorems.
Put:

\[\begin{aligned}
		\M^n             & =\{F\in \M:F\restriction_{T({<} n)}\hbox{ is identity}\}, \\
		\AM              & =\{F\restriction_{T({<}k)}:F\in \M, k> 0\},               \\
		\AM^n_k          & =\{F\restriction_{T({<}n+k)}:F\in \M^n\},                 \\
		\AM^n_k({\leq}m) & =\{f:f\in \AM^n_k\hbox{ and }\tilde{f}(n+k-1)\leq m\},    \\
		\AM^n_k(m)       & =\{f:f\in \AM^n_k\hbox{ and }\tilde{f}(n+k-1)=m\}.
	\end{aligned}
\]
$\AM^n_k$ thus denotes finite approximations to members of $\M$ where $n$ levels are fixed and $k$ levels move.
$\AM^n_k({\leq}m)$ denotes such extensions ending on level $m$ or earlier
and $\AM^n_k(m)$ denotes such extensions ending precisely on level $m$.
We now state the finite-dimensional Ramsey theorem for $(\S,\M)$-trees.
\begin{theorem}
	\label{thm:shaperamseyN}
	Let $\SNtree$ be an $(\S,\M)$-tree.
	Then, for every pair $n,k \in \omega$ and every finite colouring
	$\chi$ of $\AM^n_k$, there exists $F\in \M^n$ such that $\chi$ is constant when restricted to
	$\{F\circ g\colon g\in \AM^n_k\}$.
\end{theorem}
By compactness we also get the following two corollaries:
\begin{corollary}
	\label{cor:shaperamseyN1}
	Let $\SNtree$ be an $(\S,\M)$-tree.
	Then, for every quadruple $n,k,m,r \in \omega$ there exists $N\in \omega$ such that for every $r$-colouring
	$\chi\colon \AM^n_k({\leq}N)\to r$, there exists $f\in \AM^n_m({\leq}N)$ such that $\chi$ is constant when restricted to
	$\{F\circ g\colon g\in \AM^n_k(\leq m)\}$.
\end{corollary}
\begin{corollary}
	\label{cor:shaperamseyN2}
	Let $\SNtree$ be an $(\S,\M)$-tree.
	Then, for every quadruple $n,k,m,r \in \omega$ there exists $N\in \omega$ such that for every $r$-colouring
	$\chi\colon \AM^n_k(N)\to r$, there exists $F\in \AM^n_m(N)$ such that $\chi$ is constant when restricted to
	$\{F\circ g\colon g\in \AM^n_k(m)\}$.
\end{corollary}
To state an infinite-dimensional strengthening of Theorem~\ref{thm:shaperamseyN} we first review the basic notions of Ramsey spaces~(see Chapters~4 and~5 of~\cite{todorcevic2010introduction} for details).
Recall that a subset $\mathcal{X}$ of a topological space is
\begin{enumerate}
	\item \emph{nowhere dense} if every non-empty open set contains a non-empty open subset that avoids $\mathcal{X}$.
	\item \emph{meager} if is the union of countably many nowhere dense sets,
	\item has the \emph{Baire property} if it can be written as the symmetric difference of an open set
	      and a meager set.
\end{enumerate}
\begin{definition}[Ellentuck topological space $\M$]
	\label{defn:EllentuckN}
	Given an $(\S,\M)$-tree $\SNtree$ we
	equip $\M$ with the \emph{Ellentuck topology} given by the following basic open sets:
	$$[f,F]=\{F\circ F':F'\in \M\hbox{ and $F\circ F'$ extends $f$}\}$$
	for every $f\in \AM$ and $F\in \M$.
\end{definition}
Equivalently, the Ellentuck topology on $\M$ is the coarsest left-invariant topology on the monoid $\M$ which refines the pointwise convergence topology.

Given a shape-preserving function $F\in \M$ and $f\colon T({\leq}n)\to T$ such that $f\in \AM$ we define
$\depth_F(f)=\tilde{g}(n)$ for $g\in \AM$ satisfying $F\circ g=f$.  We set
$\depth_F(f)=\omega$ if there is no such $g$.
Given $n\in \omega$ we also write $[n,F]$ for $[F\restriction_{T({\leq}n)},F]$.
\begin{definition}
	\label{def:ramsey}
	Let $\mathcal{X}$ be a subset of $\M$.
	\begin{enumerate}
		\item We call $\mathcal{X}$ \emph{Ramsey} if for every non-empty basic
		      set $[f,F]$ there is $F'\in [\depth_F(f), F]$ such that either $[f,F']\subseteq \mathcal {X}$ or $[f,F']\cap \mathcal {X}=\emptyset$. (Such a set is sometimes called \emph{completely Ramsey} in the literature.)

		\item We say that $\mathcal{X}$ is \emph{Ramsey null} if for every $[f,F]\neq \emptyset$ we can find $F'\in [\depth_F(f), F]$ such that  $[f,F'] \cap \mathcal {X}=\emptyset$.
	\end{enumerate}
\end{definition}
Referring to the Ellentuck topology, it is not hard to see that every Ramsey set has the Baire property  and every Ramsey null set is meager~\cite[Lemma 4.24]{todorcevic2010introduction}.
The following result establishes the equivalence.
\begin{theorem}[Ellentuck theorem for shape-preserving functions]
	\label{thm:shaperamseyspace}
	Let $\SNtree$ be an $(\S,\M)$-tree
	and consider $\M$ with the Ellentuck topology given by Definition~\ref{defn:EllentuckN}.
	Then every property of Baire subset of $\M$ is Ramsey and every meager subset is Ramsey null.
\end{theorem}

We end this subsection with a key application of~\ref{item:Ndecomposition}, which will be important in the next subsection.
\begin{prop}
	\label{prop:canonical}
	Let $\SNtree$ be a $(\S,\M)$-tree, $n\in \omega$ and $f\in \AM$ be a shape-preserving function with domain $T({\leq}n)$.
	Then there exists a unique $f^+\in \M$ such that $f^+\restriction_{T({\leq} n)}=f$ and
	for every $\ell\geq \tilde{f}(n)$ it holds that $\ell\in \tilde{f}^+[\omega]$. We call this $f^+$ the \emph{canonical extension} of $f$.
\end{prop}
\begin{proof}
	Fix $f\in \AM$ and $n\in \omega$ such that domain of $f$ is $T({\leq}n)$. Let $F\in \M$ to be a function extending $f$.
	We will prove that for every $m \geq n$ such that $\tilde{F}(m)+1 < \tilde{F}(m+1)$ there exists $F'\in \M$ such that $F'\restriction_{T({\leq}m)} = F\restriction_{T({\leq}m)}$ and $\tilde{F'}(m+1) < \tilde{F}(m+1)$:
	By~\ref{item:Ndecomposition} there are $F_1, F_2\in \M$ such that $F_2\circ F_1\restriction_{T({\leq} m)}=F\restriction_{T({\leq} m)}$ and $F_2$ only skips level $\tilde{F}(m+1)-1$. Notice that $F_1\restriction_{T({\leq} m)} = F\restriction_{T({\leq} m)}$ and $\tilde{F}_1(m+1)<\tilde{F}(m+1)$, hence we can put $F' = F_1$. By doing this repeatedly, we can even ensure that $\tilde{F'}(m+1) = \tilde{F}(m)+1$.

	Now we can use this to construct a sequence $F = F_0, F_1, \ldots\in \M$ such that $F_{i}\restriction_{T({\leq}n+i)} = F_{i+1}\restriction_{T({\leq}n+i)}$ and $\tilde{F}_{i+1}(n+i+1) = \tilde{F}_{i+1}(n+i)+1$. Put $f^+$ to be the limit of this sequence. By~\ref{item:Nmonoid} we have that $f^+\in \M$. Clearly, it is the unique canonical extension of $f$.
\end{proof}

\subsection{Fat subtrees}
We can further strengthen Theorem~\ref{thm:shaperamseyspace} to finer objects defined on an $(\S,\M)$-tree.
This strengthening is essential for applying the Todorcevic axiomatization of Ramsey spaces~\cite{todorcevic2010introduction}.
\begin{figure}
	\centering
	\includegraphics[scale=0.833]{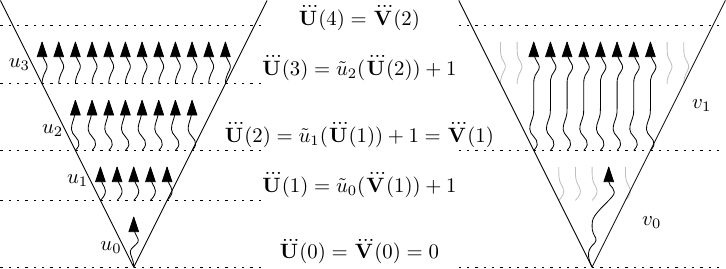}
	\caption{Sketch of fat subtrees $\fat{U}$ and $\fat{V}$ such that $\fat{V}<\fat{U}$.}
	\label{fig:levels}
\end{figure}

\begin{definition}[Fat subtree]
	Let $\SNtree$ be an $(\S,\M)$-tree and $n\in \omega+1$.
	A \emph{fat subtree} of height $n$ is a sequence $\fat{U}=(u_i)_{i\in n}$ of functions in $\AM$ for which there exists a \emph{cut function} $\fatc{U}\colon\min(n+1,\omega)\to\omega$ with $\fatc{U}(0)=0$ satisfying
	$u_\ell\in  \AM^{\fatc{U}(\ell)}_1(\fatc{U}(\ell+1)-1)$
	for every $\ell< n$:
\end{definition}
Notice that a fat subtree uniquely determines the function $\fatc{U}$.

\begin{definition}[Lift of a set]
	\label{def:lift}
	Given an $(\S,\M)$-tree $\SNtree$, a fat subtree $\fat{U}=(u_i)_{i\in n}$, levels $0\leq \ell \leq k\leq n$, and a non-empty set $X\subseteq T(\fatc{U}(\ell))$, we denote by $\lift{U}{X}{k}$ the set defined recursively as follows:
	\begin{enumerate}
		\item If $\ell = k$, we set $\lift{U}{X}{k}=X$,
		\item If $\ell< k$, we set $\lift{U}{X}{k}=\lift{U}{u_\ell^+[N]}{k}$, where $N$ is the set of all immediate successors of nodes in $X$; note that $u_\ell^+[N]\subseteq T(\fatc{U}(\ell+1))$.
	\end{enumerate}
	We will also write $\lift{U}{\ell}{k}$ for $\lift{U}{T(\ell)}{k}$.
\end{definition}
\begin{definition}[Fat subtree of a fat subtree]
	\label{def:subfatsubtrees}
	Let $\SNtree$ be an $(\S,\M)$-tree, $n\geq 0$ an integer, and $\fat{V}=(v_i)_{i\in m}$ and $\fat{U}=(u_i)_{i\in m'}$ fat subtrees.
	We say that \emph{$\fat{V}$ is a subtree of $\fat{U}$} and write $\fat{V}\leq \fat{U}$ if there exists a function $\varphi\colon \min(m+1,\omega)\to\min(m'+1,\omega)$ such that for every $\ell< m$
	it holds that
	\begin{enumerate}
		\item $\fatc{V}(\ell)=\fatc{U}(\varphi(\ell))$,
		\item $\lift{V}{\ell}{\ell+1}\subseteq \lift{U}{\varphi(\ell)}{\varphi(\ell+1)}.$
	\end{enumerate}
\end{definition}
\subsection{Ramsey theorem for fat subtrees}

Analogously as for $\M$, we can define the Ellentuck topology on the space of fat subtrees:
\begin{definition}[Ellentuck topological space of fat subtrees]
	\label{defn:R}
	Given an $(\S,\M)$-tree $(T,\allowbreak {\preceq},\allowbreak \Sigma,\S,\M)$ we denote
	by $\mathcal{R}$ the set all infinite fat subtrees of $(T,\allowbreak {\preceq},\allowbreak \Sigma,\S,\M)$
	and by $\mathcal{AR}$ the set of all finite fat subtrees of $(T,\allowbreak {\preceq},\allowbreak \Sigma,\S,\M)$. We will use upper-case bold letters to denote elements of $\mathcal R$ and lower-case bold letters to denote elements of $\mathcal{AR}$.
	We endow $\mathcal R$ with the \emph{Ellentuck topology} given by the following basic open sets:
	$$[\fat{x},\fat{U}]=\{\fat{V}\in \mathcal R:\fat{V}\leq \fat{U}\hbox{ and the sequence }\fat{x} \hbox { is an initial segment of }\fat{V}\}$$
	for every $\fat{x}\in \mathcal{AR}$ and $\fat{U}\in \mathcal {R}$.
\end{definition}
Given $n\in \omega$ and $\fat{U}\in \mathcal R$, we write $[n,\fat{U}]$ for $[\fat{U}\restriction_n,\fat{U}]$.
Given $\fat{x}\in\mathcal{AR}$ satisfying
$\fat{x}\leq \fat{U}$, we define $\depth_\fat{U}(\fat{x})$ to be the
unique integer $n$ satisfying $\fatc{U}(n)=\fatc{x}(|\fat{x}|)$.
Otherwise we put $\depth_\fat{U}(\fat{x})=\omega$.
With the notion of depth defined, Ramsey and Ramsey null sets on $\mathcal R$ are defined in a complete analogy to Definition~\ref{def:ramsey}.
Again we can establish an Ellentuck-type theorem.
\begin{theorem}[Ellentuck theorem for fat subtrees]
	\label{thm:fatramseyspace}
	Let $(T,\allowbreak {\preceq},\allowbreak \Sigma,\S,\M)$ be an $(\S,\M)$-tree
	and $\mathcal R$ a topological space given by Definition~\ref{defn:R}.
	Then every property of Baire subset of $\mathcal{R}$ is Ramsey and every meager subset is Ramsey null.
\end{theorem}
In other words, $(\mathcal R,\leq,\restriction)$ forms a topological Ramsey space
(see~\cite[Section 5]{todorcevic2010introduction} for details).

\medskip

The paper is organized as follows.  Section~\ref{sec:basic} introduces some notation
and basic properties of $(\S,\M)$-trees which are useful both for proof of the main results and for applications.
Section~\ref{sec:mainproof} proves Theorem~\ref{thm:fatramseyspace}, Section~\ref{sec:shapeprof} proves Theorem~\ref{thm:shaperamseyspace}. Finally Section~\ref{sec:applications} gives applications
of the main results.  Section~\ref{sec:applications} does not require understanding of Sections~\ref{sec:mainproof} and \ref{sec:shapeprof}.

\section{Shape-preserving functions and fat subtrees}\label{sec:basic}
In this section we give some basic observations about shape-preserving functions and fat subtrees which will be useful in the main proofs as well as for applications. We start by introducing some additional notation:

Let $(T,\preceq)$ be tree.
Given a node $a\in T$ and $n\leq \ell(a)$, we denote by $a|_n$ the (unique) predecessor of $a$ at level $n$.
Given $a,b\in T$ such that $a|_0=b|_0$, we denote by $a\meet b$ the \emph{meet} of $a$ and $b$ which is the $\preceq$-maximal member of $T$ with $a\wedge b\preceq a$ and $a\wedge b\preceq b$.
For $S\subseteq T$, we put $\ell(S)=\{\ell(a):a\in S\}$ to be the \emph{level set} of $S$.

Given an  $\S$-tree $(T,{\preceq},\Sigma,\S)$
and $a=\S(b,\bar{p},c)$ we put $\Dp(a)=\bar{p}$ and call it the \emph{parameters} of $a$ and  $\Dc(a)=c$ and call it the \emph{character} of $a$.
Since $\S$ is injective, $\Dp(a)$ and $\Dc(a)$ are well-defined.

\subsection{Properties of shape-preserving functions}
\label{sec:shapepres}
The following result summarizes basic properties of shape-preserving functions.
\begin{prop}[Basic properties of a shape-preserving function]
	\label{prop:shape-pres}
	Let $(T,{\preceq},\Sigma,\S)$ be an $\S$-tree, $S\subseteq T$ and $f\colon S\to T$ be a shape-preserving function.  Then
	\begin{enumerate}[label=(\roman*)]
		\item\label{item:sucessor} \emph{$f$ preserves successors}: For every $a,b\in S$ with $a\prec b$, we have $f(a)\prec f(b)$,
		\item\label{item:passingnumber} \emph{$f$ preserves decompositions}:
		$\S(f(a|_{n}),\allowbreak f(\Dp(a|_{n+1})),\allowbreak \Dc(a|_{n+1}))$ is defined whenever $a\in S$ and $n<\ell(a)$. Moreover, \[\S(f(a|_{n}),\allowbreak f(\Dp(a|_{n+1})),\allowbreak \Dc(a|_{n+1}))\preceq f(a).\]
		\item\label{item:orderpres} \emph{$f$ preserves the relative order of levels}: For every $a,b\in S$ with $\ell(a) < \ell(b)$, we have $\ell(f(a)) < \ell(f(b))$.
		\item\label{item:meets} \emph{$f$ preserves meets}: For every $a,b\in S$, we have $f(a\meet b)=f(a)\meet f(b)$,
		\item\label{item:closed-monoid} \emph{Shape-preserving functions form a closed monoid}:
		\begin{enumerate}
			\item The identity mapping $\mathrm{Id}\colon S\to S$ is shape-preserving.
			\item If $S'\subseteq S$ and $g\colon S'\to S$ is shape-preserving then $f\circ g$ is also shape-preserving.
			\item For every sequence $(F_i)_{i\in \omega}$ of shape-preserving functions $T\to T$ such that $(\forall_{i\in \omega})F_i\restriction_{T({\leq}i})=F_{i+1}\restriction_{T({\leq}i)}$ there exists a (unique) shape-pre\-serving function $F_\infty$ satisfying $(\forall_{i\in \omega})F_\infty\restriction_{T({\leq}i}) = F_i\restriction_{T({\leq}i)}$.
		\end{enumerate}
	\end{enumerate}
\end{prop}
\begin{proof}
	We may assume without loss of generality that $S=T$. Let $f\colon T\to T$ be shape-preserving.

	For Statement \ref{item:sucessor}, it is enough to consider the case $b = \S(a, \Dp(b), \Dc(b))$, and $f(a)\preceq \S(f(a), f\circ \Dp(b), \Dc(b))\preceq f(b)$, the second $\preceq$ following from Definition~\ref{def:shape-pres}~\ref{def:shape-pres:item2}.

	For Statement~\ref{item:passingnumber} we have $a|_{n+1} = \S(a|_n, \Dp(a|_{n+1}), \Dc(a|_{n+1})$, hence
	\[\S(f(a|_n), f\circ \Dp(a|_{n+1}), \Dc(a|_{n+1}))\preceq f(a|_{n+1})\preceq f(a).\]
	The first $\preceq$ follows from Definition~\ref{def:shape-pres}~\ref{def:shape-pres:item2}, and the second follows from Statement~\ref{item:sucessor}.

	Statement \ref{item:orderpres} is a direct consequence of the fact that $f$ is level-preserving (Definition~\ref{def:shape-pres} \ref{def:shape-pres:item1}) and
	Statement~\ref{item:sucessor}.

	To see Statement~\ref{item:meets}, consider $a,b\in S$ and let $c=a\meet b$ and $n=\ell(c)$. If $c=a$ or $c=b$, the conclusion follows from Statement~\ref{item:sucessor}.
	If $c\neq a,b$, it follows from Statement~\ref{item:passingnumber} that \[\S(f(c),f(\Dp(a|_{n+1})),\Dc(a|_{n+1}))\preceq f(a)\] and \[\S(f(c),f(\Dp(b|_{n+1})),\Dc(b|_{n+1}))\preceq f(b).\]
	Since $a|_{n+1}\neq b|_{n+1}$, it follows from two applications of \ref{S2} and the injectivity of $f$ that \[\S(f(c),f(\Dp(a|_{n+1})),\Dc(a|_{n+1}))\neq \S(f(c),f(\Dp(b|_{n+1})),\Dc(b|_{n+1})).\]

	Finally, Statement~\ref{item:closed-monoid} follows directly from Definition~\ref{def:shape-pres} and the fact that composition of level-preserving functions is still level-preserving.
\end{proof}

We also state the following generalized form of \ref{item:Ndecomposition}.
\begin{prop}
	\label{prop:shape-split}
	Let $\SNtree$ be an $(\S,\M)$-tree, $n\in \omega$ be an integer, $f\in \AM_n^0$ be a shape-preserving function and $m\in \omega$ be an integer satisfying one of the following two conditions
	\begin{enumerate}
		\item $n=0$ and $0\leq m\leq \tilde{f}(n)$, or
		\item $n>0$ and $\tilde{f}(n-1)<m\leq \tilde{f}(n)$.
	\end{enumerate}
	Denote by $f|_m$ the function $f|_m\colon T({\leq}n)\to T$ satisfying:
	$$
		f|_m(a)=
		\begin{cases}
			a|_m & \hbox{if $a\in T(n)$,} \\
			a    & \hbox{otherwise.}
		\end{cases}
	$$
	Then $f|_m\in \AM_n^0$ and there exists $g\in \AM^m_1$ such that  $g\circ f|_m=f$.
\end{prop}
\begin{proof}
	We proceed by induction on $\tilde{f}(n)-m$.  If $\tilde{f}(n)=m$, then the function $g$ is the identity and $f|_m = f\in \AM^0_n$.

	Now, assume that $\tilde{f}(n)-m>0$.
	We use \ref{item:Ndecomposition} to obtain $g_1\in \AM^0_n$ and $g_2\in \AM^{\tilde{f}(n)-1}_1$ such that
	$\tilde g_1(n)=\tilde{f}(n)-1$
	and $g_2\circ g_1=f$.
	Note that $g_1=f|_{\tilde{f}(n)-1}$.
	We apply the induction hypothesis for $g_1$ and $m$ to obtain $g'$ such that $g'\circ g_1|_m=g_1$.
	Since $g_1=f|_{\tilde{f}(n)-1}$ and $\tilde{f}(n)-m>0$, we get that $g_1|_m=(f|_{\tilde{f}(n)-1})|_m=f|_m \in \AM^n_0$
	and thus $f=g_2\circ g_1=g_2\circ g'\circ g_1|_m=g_2\circ g'\circ f|_m$.
	It remains to set $g= g_2\circ g'$.
\end{proof}
\subsection{Correspondence between fat subtrees and shape-preserving functions}
\label{sec:fatshape}
We now develop a correspondence between shape-preserving functions and fat subtrees.
\begin{definition}[level of a fat subtree]
	Given an $(\S,\M)$-tree $\SNtree$, a fat subtree $\fat{U}=(u_i)_{i\in n}$, and a level $\ell \in n$
	we put $$\fatnn{U}{\ell}=u_\ell(\lift{U}{0}{\ell}),$$
	and
	$$\fatn{U}=\bigcup_{\ell\in n} \fatnn{U}{\ell}.$$
\end{definition}
Notice that $(\fatn{U},\preceq)$ is a meet preserving subtree of $(T,\preceq)$ such that level $\ell$ of $\fatn{U}$ is a subset of level $\fatc{U}(\ell+1)-1$ of $T$.
In analogy to shape-preserving functions,  we define the \emph{level function} $\fatl{U}(\ell)\colon n\to\omega$ via $\fatl{U}(\ell)=\fatc{U}(\ell+1)-1$ for every $\ell\in n$.
\begin{definition}[Shape-preserving subtree of a fat subtree]
	Let $(T,\allowbreak {\preceq},\allowbreak \Sigma,\S,\M)$ be an $(\S,\M)$-tree, $n\geq 0$ an integer, $\fat{U}$ a fat subtree, and $F\in \M\cup \AM$ a shape-preserving function.
	We say that $F$ is a \emph{shape-preserving subtree of $\fat{U}$}, and write $F\oleq \fat{U}$, if   the range of $F$ is a subset of $\fatn{U}$.

\end{definition}
\begin{observation}
	Let $(T,\allowbreak {\preceq},\allowbreak \Sigma,\S,\M)$ be an $(\S,\M)$-tree, $n\geq 0$ an integer, and $\fat{U}\leq \fat{V}$ fat subtrees.
	Then $\fatn{U}\subseteq \fatn{V}$, and consequently, for every $F\in \M\cup\AM$ such that $F\oleq \fat{U}$, we also have $F\oleq \fat{V}$.
\end{observation}
Every fat subtree contains a (canonical) shape-preserving function:

\begin{definition}
	\label{defn:fatf}
	Let $\SNtree$ be an $(\S,\M)$-tree, $n\geq 0$ an integer, and $\fat{U}=(u_i)_{i\in \omega}$ a fat subtree.
	We put: $$\fatfn{U}{n}=\lim_{m\to\infty}u_{n+m}^+\circ u^+_{n+m-1}\circ \cdots \circ u^+_{n}.$$
	We also put $\fatf{U}=\fatfn{U}{0}$.
\end{definition}
\begin{observation}
	Let $\SNtree$ be an $(\S,\M)$-tree, $\fat{U}$ its fat subtree of height $\omega$, and $n$ an integer. Then
	\begin{enumerate}
		\item $\fatfn{U}{n}$ is well-defined and $\fatfn{U}{n}\in \M^{\fatc{U}(n)}$.
		\item $\fatf{U}\oleq \fat{U}$.
	\end{enumerate}
\end{observation}
\begin{proof}
	Let $\fat{U}=(u_i)_{i\in \omega}$. From the definition of a fat subtree it follows that $u_i$ is the identity on $T({<} \fatc{U}(i))$ for every $i\in \omega$. It follows that the limit defining $\fatfn{U}{n}$ exists, and by~\ref{item:Nmonoid}, this limit is in $\M$. The second point is immediate.
\end{proof}
In the opposite direction, every shape-preserving function can be widened to a fat subtree by a repeated application of Proposition~\ref{prop:shape-split} (note that this widening is not necessarily unique):
\begin{observation}
	\label{obs:shapetofat}
	For every $F\in \M$ there exists $\fat{U}\in \mathcal{R}$ such that $\fatf{U}=F$.
\end{observation}
\begin{proof}
	Fix $F\in \M$. We construct $\fat{U}=(u_i)_{i\in \omega}$ as follows.
	Put $u_0=F\restriction_{T(0)}$.
	For a given $i>0$ put $f_i=F\restriction_{T({\leq}i)}$  and use Proposition~\ref{prop:shape-split} to decompose $f_i$ on level $\tilde{F}(i-1)+1$, choosing $u_i$ to satisfy $f_i=u_i\circ f_i|_{\tilde{F}(i-1)+1}.$
\end{proof}

\section{Proof of the Ramsey theorem for fat subtrees}
\label{sec:mainproof}
\subsection{Pigeonhole for shape-preserving functions}
\label{sec:shapepigeon}
We start with a finitary form of a Ramsey statement for colouring shape-preserving functions with one moving level.
We fix an $(\S,\M)$-tree $\SNtree$
and put
\[
	\mathcal L_n=\{f\in \AM^n_1: \tilde{f}(n)\leq n+1\}
\]
($\mathcal L$ stands for \emph{line} since it generalizes a combinatorial line of the Hales--Jewett theorem.)
The main result of this subsection is the following lemma.
\begin{lemma}[1-dimensional pigeonhole]
	\label{lem:pigeonhole1}
	For every $n$ and every finite colouring $\chi$ of~$\AM^n_1$, there exists $h\in \mathcal \AM^n_2$ such that $\chi\restriction_{\{h\circ g:g\in \mathcal L_n\}}$ is constant.
\end{lemma}
This lemma is proved by an application of the $*$-version of the Hales--Jewett theorem which we now briefly recall.
Given a finite alphabet $\Sigma$ not containing symbols $*$ and $\lambda$, a \emph{combinatorial line} is a finite word (sequence) $L\in (\Sigma \cup \{\lambda\})^{<\omega}$ such that $\lambda$ occurs at least once ($\lambda$ is a \emph{parameter}).
For $c\in \Sigma$, we use $L(c)$ to denote the words created from $L$ by replacing all occurrences of $\lambda$ by $c$.
We also use $L(*)$ to denote the sequence created by truncating $L$ before the first occurrence of $\lambda$.
For words $u$ and $w$, we denote their \emph{concatenation} by $u\conc w$.
Given word $w$ and $i<|w|$, we denote by $w_i$ the character on index $i$ (i.e., $w(i)$ when $w$ is seen as a sequence). Indexes start from 0.

\begin{theorem}[The 1-dimensional $*$-version of the Hales--Jewett theorem~\cite{voigt1980partition,Hales1963}]
	\label{thm:HJ}
	For every finite alphabet $\Sigma$ and a positive integer $r$, there exists a positive integer $N=\mathrm{HJ}^*(|\Sigma|,r)$ such that for every $r$-colouring
	$\chi$ of $\Sigma^{\leq N}$ there exists a monochromatic line $L$ of length at most $N$. That is, $\chi$ is constant when restricted to $\{L(c):c\in \Sigma \cup \{*\}\}$.
\end{theorem}
\begin{proof}[Proof of Lemma~\ref{lem:pigeonhole1}]
	We consider the alphabet $\Sigma=\{f\in \AM^n_1:\tilde{f}(n)=n+1\}$ and establish a correspondence between the words in the alphabet $\Sigma$ and $\AM^n_1$ such that every combinatorial line will give a function in $\AM^n_2$.

	Given a finite word $w$ in the alphabet $\Sigma$, we construct $g_w\in \mathcal \AM^n_1$ by induction on $|w|$.
	If $|w|=0$, we let $g_w$ be the identity.
	Fix some $|w|>0$, and denote by $u$ the word created from $w$ by removing the last character and we let $e_w\in \Sigma$ be this last character.
	We construct $g_w$ by setting
	\[g_w(a)=g_u^+(e_w(a))\]
	for every $a \in T({\leq} n)$.
	Equivalently,
	\[g_w(a)=
		\begin{cases}
			a                                  & \mbox{for every $a\in T({<}n)$,} \\
			\S(g_u(a),\Dp(e_w(a)),\Dc(e_w(a))) & \mbox{for every $a\in T(n)$.}    \\
		\end{cases}
	\]
	Since $g_u\in \AM^n_1$ (and thus by Proposition~\ref{prop:canonical} also $g_u^+\in \M^n$) and $e_w\in \AM^n_1$, by~\ref{item:Nmonoid} it holds that $g_w\in \AM^n_1$.

	Let $s$ be an arbitrary finite word in the alphabet $\Sigma$ containing all characters from $\Sigma$ and let $\chi'$ be a colouring of finite words in $\Sigma$ defined by $\chi'(w)=\chi(g_{s\conc w})$. We apply Theorem~\ref{thm:HJ} to obtain a monochromatic combinatorial line $L=s\conc L'$.
	Recall that $L(*)$ denotes the maximal initial segment of $L$ which does not contain the parameter $\lambda$.
	Put $m=|L(*)|$ to be the length of the initial segment of $L$.

	We define functions $h_i\colon T({\leq}(n+1))\to T$ for every $i$ with $m\leq i\leq |L|$ by first setting
	\[
		h_0(a)=\begin{cases}
			a           & \mbox{for every $a\in T({<}n)$,} \\
			g_{L(*)}(a) & \mbox{for every $a\in T(n)$,}    \\
			a|_n        & \mbox{for every $a\in T(n+1)$.}  \\
		\end{cases}
	\]
	Now, assume that $h_i$ is defined for some $i$ with $0\leq i<|L|$ and set
	\[
		h_{i+1}(a)=\begin{cases}
			a                                          & \mbox{for every $a\in T({<}n)$,} \\
			g_{L(*)}(a)                                & \mbox{for every $a\in T(n)$,}    \\
			\S(h_{i}(a),\Dp(L_i(a|_n)),\Dc(L_i(a|_n))) & \mbox{for every $a\in T(n+1)$}   \\&\hbox{if $L_i\in \Sigma$,}  \\
			\S(h_{i}(a),\Dp(a),\Dc(a))                 & \mbox{for every $a\in T(n+1)$}   \\&\hbox{if $L_i=\lambda$.} \\
		\end{cases}
	\]
	We will verify that all the successor operations used are actually defined.

	Observe that, for every $i$ with $0\leq i\leq |L|$, we have $h_i\restriction_{T({\leq}n)}=g_{L(*)}$ and consequently $h_i\restriction_{T({\leq}n)}\in \AM_1^{n}$.
	Let $m$ be the index of the first occurrence of the parameter $\lambda$ in $L$ (so $m$ is the length of $L(*)$).
	Note that the functions $h_0,h_1,\ldots, h_m$ are not shape-preserving since $\tilde{h}_i(n+1)\leq \tilde{h}_i(n)$ for every $0\leq i\leq m$.
	Also observe, by comparing the construction of $h_{i+1}$ and the alternative definition of $g_w$, that $h_m(a)=h_m(a|_n)$ for every $a\in T(n+1)$.

	We prove by induction that the function $h_i$, for every $i$ with $m+1\leq i\leq |L|$, is well-defined (that is, all the successor operations used in its construction are defined), shape-preserving and in $\AM^n_2$.

	The function $h_{m+1}$ lies in $\AM^n_2$ since $h_{m+1}(a)=\S(h_{m+1}(a|_n),\Dp(a),\Dc(a))$ for every $a\in T(n+1)$
	and thus $h_{m+1}=g^+_{L(*)}\restriction_{T({\leq}(n+1))}$.
	Consequently, we have $h_{m+1}\in \AM^n_2$.

	Now, assume that the function $h_i$ lies in $\AM^n_2$ for some $i$ with $m+1\leq i<|L|$.
	We have already verified that $h_{i+1}\restriction_{T({\leq}n)}$ is shape-preserving.
	We consider the following two cases:
	\begin{enumerate}
		\item If $L_{i}=\lambda$, then $h_{i+1} = F_{n+i+1}^{n+m+1}\circ h_i$
		      (recall that~\ref{item:Nduplication} ensures the existence of $F_{n+i+1}^{n+m+1}\in \M$).
		\item If $L_{i}\in \Sigma$, then there exists $j<i$ satisfying $L_j=s_j=L_{i}$.
		      (Recall that $L=s\conc L'$ and the choice of $s$ ensure the existence of such $j$.)
		      Then $h_{i+1} = F_{n+i+1}^{n+j+1}\circ h_i$.
	\end{enumerate}
	In both cases,~\ref{item:Nduplication} gives us
	$h_{i+1}\in \AM^n_2$.

	Finally, we set $h=h_{|L|}$. Observe that
	\[\{h\circ g:g\in \mathcal L_n\}=\{g_{L(c)}:c\in \Sigma\cup\{*\}\}\]
	by the construction of $h$.
	Consequently, $\chi\restriction_{\{h\circ g:g\in \mathcal L^n\}}$ is constant.
\end{proof}
\subsection{Fat subtrees and Todorcevic's axiomatization of Ramsey spaces}
\label{sec:ramseyspace}
We fix an $(\S,\M)$-tree $\SNtree$,
consider the topological space $\mathcal R$ and set $\mathcal{AR}$ introduced in Definition~\ref{defn:R} and the triple $(\mathcal R,{\leq},\allowbreak r)$
where $\leq$ is the order of fat subtrees introduced in
Definition~\ref{def:subfatsubtrees} restricted to $\mathcal R$, and $r$
is a function $r\colon \mathcal R\times \omega\to \mathcal {AR}$ defined
by $$r(\fat{U},n)=\fat{U}\restriction_n$$ for every $\fat{U}\in
	\mathcal R$ and $n\in \omega$.
As is usual in the area, we will write $r_n(\fat{U})$ for $r(\fat{U},n)$.
Given $n\in \omega$ we also put $\mathcal{AR}_n=\{\fat{x}\in \mathcal {AR}:|\fat{x}|=n\}$.

The following is an axiomatization of topological Ramsey spaces given by Todorcevic~\cite[Section 5]{todorcevic2010introduction}.
\begin{enumerate}[label=(A\arabic*)]
	\item \label{item:A1} Sequencing:
	      \begin{enumerate}[label=(\arabic*)]
		      \item $r_0(\fat{U})=\emptyset$ for all $\fat{U}\in \mathcal R$.
		      \item $\fat{U}\neq\fat{V}$ implies $r_n(\fat{U})\neq r_n(\fat{V})$ for some $n$.
		      \item $r_n(\fat{U})=r_m(\fat{V})$ implies that $n=m$ and $r_k(\fat{U})=r_k(\fat{V})$ for all $k<n$.
	      \end{enumerate}
	\item \label{item:A2} There is a quasi-ordering	$\finleq$ on $\mathcal{AR}$ such that:
	      \begin{enumerate}[label=(\arabic*)]
		      \item\label{A2.1} $\{\fat{x}\in \mathcal{AR}:\fat{x}\finleq\fat{y}\}$ is finite for all $\fat{y}\in\mathcal{AR}$.
		      \item\label{A2.2} $\fat{V}\leq \fat{U}$ iff $\forall_n\exists_m: r_n(\fat{V})\finleq r_m(\fat{U})$.
		      \item\label{A2.3} $\forall_{\fat{x},\fat{y}, \fat{z}\in\mathcal{AR}}[\fat{x}\sqsubseteq\fat{y}\hbox{ and }\fat{y}\finleq \fat{z}\hbox{ implies }\exists_{\fat{z}'\sqsubseteq \fat{z}}:\fat{x}\finleq \fat{z}']$. Here, $\fat{x}\sqsubseteq\fat{y}$ means that $\fat{x}$ is an initial segment of $\fat{y}$.
	      \end{enumerate}
	\item \label{item:A3} Amalgamation:
	      \begin{enumerate}[label=(\arabic*)]
		      \item\label{A3.1} If $\depth_{\fat{U}}(\fat{x})$ is finite then $[\fat{x},\fat{V}]\neq \emptyset$ for all $\fat{V}\in [\depth_\fat{U}(\fat{x}),\allowbreak \fat{U}]$.
		      \item\label{A3.2} $\fat{V}\leq \fat{U}$ and $[\fat{x},\fat{V}]\neq\emptyset$ imply that there is $\fat{U}'\in [\depth_\fat{U}(\fat{x}),\allowbreak \fat{U}]$ such that $\emptyset\neq [\fat{x},\fat{U}']\subseteq [\fat{x},\fat{V}]$.
	      \end{enumerate}
	\item \label{item:A4} If $\depth_{\fat{U}}(\fat{x})$ is finite and $\mathcal O\subseteq \mathcal{AR}_{|\fat{x}|+1}$, then there is $\fat{V}\in[\depth_\fat{U}(\fat{x}),\fat{U}]$ such that $r_{|\fat{x}|+1}[\fat{x},\fat{V}]$ is either a subset of $\mathcal O$ or a subset of the complement of $\mathcal O$.
\end{enumerate}
To prove Theorem~\ref{thm:fatramseyspace} we will apply the following theorem.
\begin{theorem}[Todorcevic's Abstract Ellentuck theorem~\cite{todorcevic2010introduction}]\label{thm:stevo}
	If $(\mathcal R,\leq,r)$ is closed and if it satisfies axioms \ref{item:A1}, \ref{item:A2}, \ref{item:A3} and \ref{item:A4}, then every property of Baire
	subset of $\mathcal R$ is Ramsey and every meager subset is Ramsey null, or in other words, the triple $(\mathcal R,\leq, r)$ forms a topological Ramsey space.
\end{theorem}

The majority of this section will be devoted to proving~\ref{item:A4} as the following proposition:

\begin{prop}[\ref{item:A4} pigeonhole]
	\label{prop:A4}
	For every $\fat{x}\in \mathcal {AR}$ and $\fat{U}\in \mathcal {R}$ such that $\depth_\fat{U}(\fat{x})< \omega$ and every $\mathcal O\subseteq \mathcal{AR}_{|\fat x|+1}$
	there is $\fat{V}\in [\depth_\fat{U}(\fat{x}),\fat{U}]$ such that $r_{|\fat{x}|+1}[[\fat{x},\fat{V}]]$ is either a subset of $\mathcal O$ or a subset of the complement of $\mathcal O$.
\end{prop}

Assuming it is true, proving Theorem~\ref{thm:fatramseyspace} is easy.

\begin{proof}[Proof of Theorem~\ref{thm:fatramseyspace}]
	In order to apply Theorem~\ref{thm:stevo} we only need to verify axioms \ref{item:A1}, \ref{item:A2}, \ref{item:A3} and \ref{item:A4}.

	Since fat subtrees trees are sequences, axiom \ref{item:A1} is obvious.
	To verify \ref{item:A2} we define order $\finleq$ on $\mathcal {AR}$ by putting $\fat{x}\finleq \fat{y}$ if and only if $\fat{x}\leq \fat{y}$ and $\fatc{x}(|\fat{x}|)=\fatc{y}(|\fat{y}|)$.
	Item \ref{A2.1} follows from the fact that $(T,\preceq)$ is finitely branching and thus there are only finitely many subtrees to consider ending at a given level.
	Items \ref{A2.2} and \ref{A2.3} follow directly from the definition.

	To see \ref{item:A3} \ref{A3.1}, consider $\fat{x}\in\mathcal{AR}$ and $\fat{V}\in \mathcal R$ such that $n=\depth_\fat{V}(\fat{x})$ is finite. Choose $\fat{U}=(u_i)_{i\in \omega}\in [n,\fat{V}]$. Consider $\fat{U}'$ constructed as a concatenation of $\fat{x}$ and the sequence $(u_{i+n})_{i\in \omega}$. It follows that $\fat{U}'\in[\fat{x},\fat{U}]$.

	To see \ref{item:A3} \ref{A3.2}, choose $\fat{U}=(u_i)_{i\in\omega}\leq \fat{V}\in \mathcal{R}$ and $\fat{x}\in \mathcal{AR}$ such that $[\fat{x},\fat{U}]\neq \emptyset$. Let $n=\depth_\fat{U}(\fat{x})$ and $m=\depth_\fat{V}(\fat{x})$.
	Construct $\fat{U}'$ as a concatenation of $\fat{V}\restriction_m$ and $(u_{i+n})_{i\in\omega}$.
	Then the set $[\fat{x},\fat{U}']$ consists of all fat subtrees $\fat{U}''$ with initial segment $\fat{x}$ which continue as subtrees of $\fat{U}$. As the same is true for $[\fat{x},\fat{U}]$, we thus have $[\fat{x}, \fat{U}'] = [\fat{x}, \fat{U}]$.

	Finally, axiom~\ref{item:A4} is the content of Proposition~\ref{prop:A4}. This finishes the proof.
\end{proof}

We devote the rest of this section to a proof of Proposition~\ref{prop:A4}.
The proof builds on Lemma~\ref{lem:pigeonhole1} and
applies combinatorial forcing. This is a quite general technique that is an adaptation of the combinatorial forcing proof from~\cite{carlson1984,Karagiannis2013}.

We first need to re-state Lemma~\ref{lem:pigeonhole1} in the context of fat subtrees.

\begin{figure}
	\centering
	\includegraphics[scale=0.833]{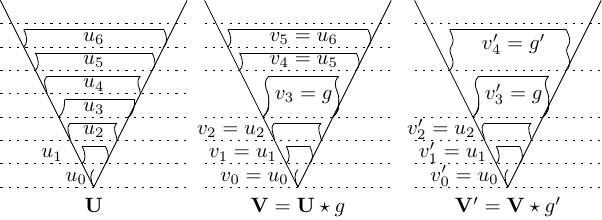}
	\caption{Operation $\fat{U}=\fat{V}\star g$ and line $(g,g')$.}
	\label{fig:line}
\end{figure}
\begin{definition}
	\label{def:star}
	Let $\fat{U}=(u_i)_{i\in \omega}$ be a fat subtree, $n\leq m\in \omega$ be integers and $g\in \AM^{\fatc{U}(n)}_1$ be a shape-preserving function such that $$g[T(\fatc{U}(n))]\subseteq u_m[\lift{U}{n}{m}].$$  By $\fat{U}\star g$ we denote a \emph{fat subtree of $\fat{U}$ produced by $g$} which is a fat subtree $\fat{V}=(v_i)_{i\in \omega}$ defined as follows:
	$$
		v_i=
		\begin{cases}
			u_i       & \hbox{if }i< n, \\
			g         & \hbox{if }i= n, \\
			u_{m+i-n} & \hbox{if }i>n.
		\end{cases}
	$$
	For functions $g \in \AM^{\fatc{U}(n)}_1$ such that $g(T(\fatc{U}(n)))\not \subseteq u_m[\lift{U}{n}{m}]$ we say that $\fat{U}\star g$ is undefined.
\end{definition}
See Figure~\ref{fig:line}.
\begin{definition}[Concatenation]
	Given $n< \omega$, a fat subtree $\fat{x}=(x_i)_{i\in n}\in \mathcal{AR}$,  and $g\in \AM^{\fatc{x}(n)}_1$, we denote by $\fat{x}\conc g$ the fat subtree $\fat{y}=(y_i)_{i\in n+1}$ where $y_i=x_i$ for every $i\in n$ and $y_n=g$.
\end{definition}

Given $\fat{x}\in \mathcal{AR}$ we put:
$$
	\mathcal {G}^\fat{x}=\{\fat{x}\conc g:g\in \AM^{\fatc{x}(|\fat{x}|)}_1\}.
$$
See Figure~\ref{fig:line}.
Notice that this is precisely the set of all fat subtrees trees of height $|\fat{x}|+1$ which extend $\fat{x}$ (\emph{one level extensions} of $\fat{x}$).
\begin{definition}[Line]
	For every $\fat{U}\in \mathcal R$ and $n< \omega$,
	we call the pair $(g,g')$ a \emph{line} if
	\begin{enumerate}
		\item $g\in \AM^{\fatc{U}(n)}_1$ and $\fat{V} =\fat{U}\star g$ is defined, and
		\item $g'\in \AM^{\fatc{V}(n+1)}_1$ and $\fat{V}'=\fat{V}\star g'$ is defined.
	\end{enumerate}
	Given $\fat{x}\in \mathcal {AR}$ with $n=\depth_\fat{U}(\fat{x})$ finite, we also put
	$$\mathcal{L}^{\fat{x},(g,g')}_\fat{U} = \{\fat{y}\in \mathcal{G}^\fat{x}:\fat{y}\leq \fat{V}'\hbox{ and }\fatl{y}(|\fat{y}-1|)\leq \fatl{V}'(n+1)\}$$
\end{definition}
Notice that $\fat{V}'$ restricts two levels just after the end of $\fat{U}$.
The name ``line'' is motivated by the following connection of $\mathcal{L}^{\fat{x},(g,g')}_\fat{U}$ to lines $\mathcal{L}^n$ introduced in previous section:
\begin{lemma}
	\label{lem:4andy}
	For every $\fat{U}\in \mathcal R$, $\fat{x}\in \mathcal {AR}$ such that $n=\depth_\fat{U}(\fat{x})$ is finite
	and every line $(g,g')$ it holds that:
	$$\mathcal{L}^{\fat{x},(g,g')}_\fat{U}  = \{\fat{x}\conc (g'\circ g^+\circ g''):g''\in \mathcal L^{\fatc{U}(n)}\}.$$
\end{lemma}
\begin{proof}
	$$
		\begin{aligned}
			\mathcal{L}^{\fat{x},(g,g')}_\fat{U} & = \{\fat{y}\in \mathcal{G}^\fat{x}:\fat{y}\leq \fat{V}'\hbox{ and }\fatl{y}(|\fat{y}-1|)\leq \fatl{V}'(n+1)\}                      \\
			                                     & = \{\fat{x}\conc g\}\cup \{\fat{y}\in \mathcal{G}^\fat{x}:\fat{y}\leq \fat{V}'\hbox{ and }\fatl{y}(|\fat{y}-1|) = \fatl{V}'(n+1)\} \\
			                                     & = \{\fat{x}\conc g\}\cup \{\fat{x}\conc (g'\circ g^+\circ g''):g''\in \AM^{\fatc{U}(n)}_1(\fatc{U}(n)+1)\}                         \\
			                                     & = \{\fat{x}\conc (g'\circ g^+\circ g''):g''\in \mathcal L^{\fatc{U}(n)}\}.
		\end{aligned}
	$$
	The third equality may deserve a justification.
	To see that $\fat{y}=\fat{x}\conc (g'\circ g^+\circ g'')<\fat{V}'$ for every $g''\in \AM^{\fatc{U}(n)}_1(\fatc{U}(n)+1)$ we want to check
	that $$\lift{y}{n}{n+1}\subseteq \liftp{V}{n}{n+2}.$$
	(All other inclusions required by Definition~\ref{def:subfatsubtrees} are trivial, since we already have $\fat{x}<\fat{U}$ and thus also $\fat{x}<\fat{V}'$ since $\fat{U}\restriction_n=\fat{V}'\restriction_n$.)
	Let $N$ be the set of all immediate successors of $\liftp{V}{n}{n+1}=g^+[T(\fatc{V}(n)+1)].$
	Notice that $$(g^+\circ g'')[T(\fatc{V}'(n)+1]\subseteq N\subseteq  T(\fatc{V}'(n+1)+1)$$
	By Definition~\ref{def:lift} have:
	$$\begin{aligned}\lift{y}{n}{n+1} & =(g'\circ g^+\circ g'')^+[T(\fatc{V}'(n)+1]             \\&=(g')^+[(g^+\circ g'')^+[T(\fatc{V}'(n)+1]]\\
                                & \subseteq(g')^+[N]=\liftp{V}{N}{n+2}=\liftp{V}{n}{n+2}.\end{aligned}$$
\end{proof}
The set $\mathcal{L}^{\fat{x},(g,g')}_\fat{U}$ thus contains all one-level extensions of $\fat{x}$ within two levels affected by $\fat{V}'$.
The following is the minimal finite pigeonhole statement for colouring fat subtrees.

\begin{lemma}[1-dimensional pigeonhole for fat subtrees]
	\label{lem:fatpigeonhole1}
	For every $\fat{U}\in \mathcal R$, $\fat{x}\in \mathcal {AR}$ such that $n=\depth_\fat{U}(\fat{x})$ is finite,
	and for every finite colouring $\chi$ of $\mathcal {G}^\fat{x}$ there exists a line $(g,g')$
	with the property that $\chi\restriction_{\mathcal{L}^{\fat{x},(g,g')}_\fat{U}}$ is constant.
\end{lemma}
\begin{proof}
	Put  $m=\fatc{U}(n)$  and $F=\fatfn{U}{n}$ (Definition~\ref{defn:fatf}).
	Define colouring $\chi'$ of $\AM^m_1$ by putting
	$$\chi'(g)=\chi(\fat{x}\conc(F\circ g)).$$
	Apply Lemma~\ref{lem:pigeonhole1} on $m$ and $\chi'$ to obtain $h\in\AM^m_2$.
	Put $$g=(F\circ h)\restriction_{T({\leq}m)}$$ and $$\fat{V}=\fat{U}\star g.$$
	Note that $\fat{U}\star g$ is defined. Now apply Proposition~\ref{prop:shape-split} to decompose $F\circ h$ as
	$$F\circ h= g'\circ (F\circ h)|_{\fatc{V}(n+1)}$$ where $g'\in \AM^{\fatc{V}(n+1)}_1$.
	It follows that $(g,g')$ is a a line. By Lemma~\ref{lem:4andy} we have $$\mathcal{L}^{\fat{x},(g,g')}_\fat{U}=\{\fat{x}\conc (F\circ h\circ g''):g''\in \mathcal L_m\},$$
	and thus $g$ and $g'$ have the desired properties.
\end{proof}

\begin{definition}[Large sets]
	For $\fat{U}\in \mathcal R$, $\fat{x}\in \mathcal {AR}$ such that $n=\depth_\fat{U}(\fat{x})$ is finite and $\mathcal O\subseteq \mathcal{G}^\fat{x}$ we say that $\mathcal O$ is
	$(\fat{U},\fat{x})$-large if for every $\fat{V} \in [n, \fat{U}]$ there exists $\fat{y}\in \mathcal O$ such that $\fat{y}\leq \fat{V}$.
\end{definition}
In other words, the set $\mathcal O$ is $(\fat{U},\fat{x})$-large if there is no  $\fat{V} \in [n,\fat{U}]$ ``avoiding'' $\mathcal O$.

\begin{observation}\label{obs:large}
	For every $\fat{U}\in \mathcal R$, $\fat{x}\in \mathcal {AR}$ such that $n=\depth_\fat{U}(\fat{x})$ is finite, and $(\fat{U},\fat{x})$-large set $\mathcal O$ there exist a fat subtree $\fat{U}'\in [n,\fat{U}]$ and $n'\geq n$ such that for every $\fat{V}\in [n',\fat{U}']$
	there exists some $\fat{y}\in \mathcal O$ satisfying $\fat{y}\leq \fat{V}$ and $\depth_{\fat{V}}(\fat{y})=n'+1$.
\end{observation}

\begin{proof}
	Suppose for a contradiction that for every $\fat{U}' \in [n,\fat{U}]$ and every $n'\geq n$ there is $\fat{V}\in [n',\fat{U}']$ such that
	there is no $\fat{y}\in \mathcal O$ satisfying $\fat{y}\leq \fat{V}$ and $\depth_{\fat{V}}(\fat{y})=n'+1$.
	For each choice of $\fat{U}'$ and $n'$, denote by $\fat{V}^{\fat{U}'}_{n'}$ some fat subtree with this property.
	Construct an infinite sequence $\fat{U} = \fat{U}_{0} \geq \fat{U}_1 \geq \ldots$ by putting $\fat{U}_{i+1}=\fat{V}^{\fat{U}_{i}}_{i+n}$ for every $i\in \omega$. Notice that for every $i\in \omega$ we have:
	\begin{enumerate}
		\item $\fat{U}_{i+1} \in [n,\fat{U}]$,
		\item There is no $\fat{y}\in \mathcal O$ with $\fat{y}\leq \fat{U}_{i+1}$ and $\depth_{\fat{U}_{i+1}}(\fat{y})\leq n+i+1$.
	\end{enumerate}
	Finally, put $\fat{V}$ to be the limit of $(\fat{U}_i)_{i\in \omega}$. It is easy to see that $\fat{V} \in [n,\fat{U}]$ and there is no $\fat{y}\in \mathcal O$ with $\fat{y}\leq \fat{V}$ which contradicts the fact that $\mathcal O$ is $(\fat{U},\fat{x})$-large.
\end{proof}

\begin{lemma}[1-dimensional pigeonhole for large sets]
	\label{lem:largepigeon}
	For every $\fat{U}\in \mathcal R$, $\fat{x}\in \mathcal {AR}$ such that $n=\depth_\fat{U}(\fat{x})$ is finite,
	and every $(\fat{U},\fat{x})$-large set $\mathcal O$, there exists a line $(g,g')$
	with the property that $\mathcal{L}^{\fat{x},(g,g')}_\fat{U}\subseteq \mathcal O$.
\end{lemma}
\begin{proof}
	Fix $\fat{U}$, $\fat{x}$, $n$, and $\mathcal O$ as in the statement. Use Observation~\ref{obs:large} to get $\fat{U}'=(u'_i)_{i \in \omega} \in [n,\fat{U}]$ and $n'\geq n$.
	Put $$
		\begin{aligned}
			\mathcal C & =\{h'\in \AM^{\fatc{U}(n)}_1(\fatc{U}'(n')): \fat{U}'\star (u'_{n'}\circ h')\hbox{ is defined}\}, \\
			\fat{x'}   & =\fat{x}\conc (u'_{n})\conc (u'_{n+1})\conc \cdots \conc (u'_{n'-1}).
		\end{aligned}
	$$
	Next we will apply Lemma~\ref{lem:fatpigeonhole1} on $\fat{U}'$, $\fat{x}'$, and a finite colouring $\chi$ of $\mathcal{G}^{\fat{x}'}$ which we now define.
	For a given $\fat{y}'={\fat{x}'}\conc h\in \mathcal{G}^{\fat{x}'}$ such that $\fat{y}'\leq \fat{U}'$, its colour $\chi(\fat{y}')$ is a function $\chi(\fat{y}')\colon \mathcal C\to \{0,1\}$ defined by putting $\chi(\fat{y}')(h')=1$ if and only if $\fat{x}\conc (h\circ h')\in \mathcal O$.

	Lemma~\ref{lem:fatpigeonhole1} gives us a line $(h,g')$ and a colour $c\colon \mathcal C\to \{0,1\}$ such that $\chi\restriction_{\mathcal{L}^{\fat{x}',(h,g')}_{\fat{U}'}} = c$. See Figure~\ref{fig:line2}.
	\begin{figure}
		\centering
		\includegraphics[scale=0.833]{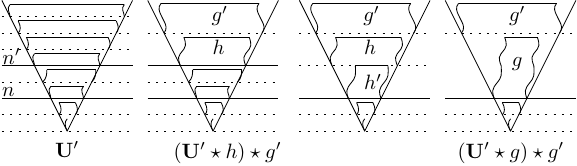}
		\caption{Construction of a line.}
		\label{fig:line2}
	\end{figure}
	We claim that there is $h'\in \mathcal C$ such that $c(h') = 1$. Once we see this, then setting $g = h\circ h'$ and applying Lemma~\ref{lem:4andy} we have $\mathcal{L}^{\fat{x},(g,g')}_\fat{U}\subseteq \mathcal O$ as desired:
	$$\begin{aligned}\mathcal{L}^{\fat{x},(g,g')}_\fat{U} & = \{\fat{x}\conc (g'\circ g^+\circ g''):g''\in \mathcal L^{\fatc{U}(n)}\}                                                  \\
                                                    & = \{\fat{x}\conc (g'\circ (h\circ h')^+\circ g''):g''\in \mathcal L^{\fatc{U}(n)}\}                                        \\
                                                    & \subseteq \{\fat{x}\conc (({h'})^+\circ (g'\circ h^+\circ g''')):g'''\in \mathcal L^{\fatc{U}'(n')}\}\subseteq \mathcal O.
		\end{aligned}
	$$

	If $c(h') = 0$ for every $h'\in \mathcal C$, put $\fat{V}=\fat{U}'\star h$ (this is defined because $h\in \AM^{\fatc{U}'(n')}_1$ and $\fat x'{}\conc h \leq \fat{U}'$). By Proposition~\ref{prop:shape-split}, every $\fat{y}\in \mathcal{G}^\fat{x}$ such that $\fat{y}\leq \fat{V}$ and
	$\depth_{\fat{V}}(\fat{y})=n'$
	has the property that $\fat{y}=\fat{x}\conc (h\circ h')$ for some $h'\in \mathcal C$.  Consequently, $\fat{y}\notin \mathcal O$, which means that we found $\fat{V}\in [n',\fat{U}']$ with no $\fat{y}\in \mathcal O$ such that $\fat y\leq \fat{V}$ and $\depth_{\fat{V}}(\fat{y})=n'+1$, a contradiction.

\end{proof}
\begin{lemma}
	\label{lem:fixlevel}
	For every $\fat{U}\in \mathcal R$, $\fat{x}\in \mathcal {AR}$ such that $n=\depth_\fat{U}(\fat{x})$ is finite,
	and $(\fat{U},\fat{x})$-large set $\mathcal O$ there exist $\fat{V}\in [n,\fat{U}]$ and $g\in \AM^{\fatc{U}(n)}_1$ such that the set
	$N$ of all words $g'$ such that $\mathcal{L}^{\fat{x},(g,g')}_\fat{U}\subseteq \mathcal O$ is $(\fat{V}\star g,\fat{x}\conc g)$-large.
\end{lemma}

\begin{proof}
	Enumerate $\AM^{\fatc{U}(n)}_1$ as $g_1,g_2,\ldots$ so that $\tilde{g}_i(\fatc{U}(n))\leq \tilde{g}_j(\fatc{U}(n))$ for all $i$ and $j$ with $0<i<j$.
	We will now try (and eventually fail) to inductively construct an infinite sequence $(\fat{U}_i)_{i \in \omega}$ of fat subtrees such that $\fat{U}=\fat{U}_0 \geq \fat{U}_1 \geq \fat{U}_2 \geq \cdots$ and the following two conditions are satisfied for every $i\geq 1$:
	\begin{enumerate}
		\item If $\fat{x}\conc g_{i}\not\leq \fat{U}_{i-1}$ then $\fat{U}_{i}=\fat{U}_{i-1}$.
		\item If $\fat{x}\conc g_{i}\leq \fat{U}_{i-1}$ then:
		      \begin{enumerate}[label=(\roman*)]
			      \item\label{item:TN1} There is no $g'$ such that $\mathcal L^{\fat{x},(g_i,g')}_{\fat{U}_{i-1}}\subseteq \mathcal O$.
			      \item\label{item:TN2} $\fat{U}_i\restriction_{m}=\fat{U}_{i-1}\restriction_m$ where $m=\depth_{\fat{U}_{i-1}}(\fat{x}\conc g_i)$.
		      \end{enumerate}
	\end{enumerate}
	We start by setting $\fat{U}_0=\fat{U}$.  Note that Conditions~\ref{item:TN1} and \ref{item:TN2} do not apply to $\fat{U}_0$.

	Now, assume that $\fat{U}_{i-1}$ has been constructed for some $i > 0$. If $\fat{x}\conc g_i\not\leq \fat{U}_{i-1}$ we can put $\fat{U}_{i}=\fat{U}_{i-1}$.
	If $\fat{x}\conc g_i\leq \fat{U}_{i-1}$ then put $N_i$ to be set of all shape-preserving functions $g'$ such that $(g_i,g')$ satisfies the conclusion
	of Lemma~\ref{lem:largepigeon} for $\fat{U}_{i-1}$ and $\mathcal O$. We consider two cases:
	\begin{enumerate}
		\item If $N_i$ is $(\fat{U}_{i-1}\star g_i,\fat{x}\conc g_i)$-large, then we can stop the induction and put $\fat{V} = \fat{U}_{i-1}$ and $g = g_i$.
		\item If $N_i$ is not $(\fat{U}_{i-1}\star g_i,\fat{x}\conc g_i)$-large let $\fat{V}_i\in [\fat{x}\conc g_i,\fat{U}_i \star{g}_i]$ be a fat subtree avoiding $N_i$.
			Construct $\fat{U}_i$ as the concatenation of $\fat{U}_{i-1}\restriction_{\depth_{\fat{U}_{i-1}}(\fat{x}}\conc g_i)$ and the sequence created form $\fat{V}_i$ by removing the first $n+1$ levels.
		      Clearly $\fat{U}_{i}$ satisfies \ref{item:TN2}.
		      Notice that $\fat{U}_i\star{g}_i=\fat{V}_i$, and since it avoids $N_i$, we have \ref{item:TN1}.
	\end{enumerate}

	Suppose that we finished constructing $\fat{U}_i$ for every $i\in \omega$
	without ever obtaining a large set.
	By~\ref{item:TN2}, the sequence $(\fat{U}_i)_{i \in \omega}$ has a limit $\fat{V}\leq \fat{U}$.
	Since $\mathcal O$ is $(\fat{U},\fat{x})$-large, it is also $(\fat{V},\fat{x})$-large, and thus Lemma~\ref{lem:largepigeon} gives a line $(g,g')$
	which contradicts condition~\ref{item:TN1} for $i$ such that $g = g_i$.
\end{proof}
\begin{proof}[Proof of Proposition~\ref{prop:A4}]
	Let $\fat{U}$, $\fat{x}$ and $\mathcal O$ be as in the statement of the proposition.
	Notice that for every subset $\mathcal O'$ of $\mathcal{AR}_{|\fat x|+1}$ it holds that either $\mathcal O'$ or its complement is $(\fat{U},\fat{x})$-large.
	Without loss of generality assume that $\mathcal O$ is $(\fat{U},\fat{x})$-large. Put $\mathcal O_0=\mathcal O$,
	and repeatedly apply Lemma~\ref{lem:fixlevel} to obtain sequences $\fat{U}=\fat{U}_0\geq \fat{U}_1\geq \ldots$ and $\fat{x} = \fat{x}_0 \sqsubseteq \fat{x}_1 \sqsubseteq \cdots$ and $\mathcal O_0, \mathcal O_1, \cdots$ such that $\mathcal O_i$ is $(\fat{U}_i,\fat{x}_i)$-large as follows:

	For every $i\in \omega$, apply Lemma~\ref{lem:fixlevel} for $\fat{U}_i$, $\fat{x}_i$ and $\mathcal O_i$ to obtain $\fat{V}$, $g$ and $N$. Put $\fat{U}_{i+1} = \fat{V}\star g$, $\fat{x}_{i+1} = \fat{x}_i\conc g$ and $\mathcal O_{i+1} = N$.

	Finally, putting $\fat{V}$ to be the limit of $(\fat{U}_i)_{i\in \omega}$ finishes the proof.
\end{proof}

\section{Proof of a Ramsey theorem for shape-preserving functions}

\label{sec:shapeprof}
Throughout this section, $\mathcal{R}$ and $\M$ are equipped with their respective Ellentuck topologies unless specified otherwise. Let us write $\pi\colon \mathcal{R}\to \M$ for the map given by $\pi(\fat{U}) = \fatf{U}$ (Definition~\ref{defn:fatf}).
\begin{lemma}
	\label{lem:open}
	The map $\pi$ is continuous and open.
\end{lemma}
\begin{proof}
	To see continuity, consider a basic open set $\mathcal{X}=[f,F]\subseteq \M$. Then $\pi^{-1}(X)$ is the union of all basic open sets $[\fat{x},\fat{U}]$ for all choices of $\fat{x}$ and $\fat{U}$ satisfying $\fatf{x}=f$  and $\fatf{U}=F$, hence it is open in $\mathcal R$.

	To see that $\pi$ is open, fix a basic clopen set $\mathcal{Y} = [\fat{x}, \fat{U}]\subseteq \mathcal{R}$. We claim that $\pi[[\fat{x}, \fat{U}]] = [\fatf{x}, \fatf{U}]$. The forward inclusion is clear. For the reverse, fix $F\in [\fatf{x}, \fatf{U}]$, and write $F = \fatf{U}\circ G$ for some $G\in \M$. We claim that there is $\fat{V}\in [\fat{x}, \fat{U}]$ with $\fatf{V} = F = \fatf{U}\circ G$, which will prove the reverse inclusion.

	To see this, we define a sequence $(\fat{U}_i)_{i\in m}$ of fat subtrees as follows:
	$$
		\fat{U}_0=\fat{U}\star (\fatf{U} \circ (G\restriction_{T({\leq}0)}))
	$$
	and
	$$
		\fat{U}_i=\fat{U}_{i-1}\star f_i
	$$
	for every $0<i<m$ where $f_i\in \AM^{\fatc{U}_{i-1}(i-1)}_1$ is obtained using Proposition~\ref{prop:shape-split} decomposing $\fatf{U}\circ G\restriction_{T({\leq}i)}$ on level $\fatc{U}_{i-1}(i-1)$.

	Notice that $\fat{U}_{i+1}\restriction_i = \fat{U}_i\restriction_i$ and $\fat{U}_{i+1}\leq \fat{U}_{i}$ for every $i$, and we put $\fat{V}'$ to be the limit of the sequence $(\fat{U}_i)_{i\in m}$. Since $\fat{U}_0\leq \fat{U}$, we get that $\fat{V}' \leq \fat{U}$. By simply replacing a suitable initial segment of $\fat{V}'$ with $\fat{x}$, we obtain $\fat{V}\leq \fat{U}$ with $\fat{x}\sqsubseteq \fat{V}$.

	From the construction it follows that $F_{\fat{U}_i} \restriction_{T({\leq} i)} = \fatf{U}\circ G\restriction_{T({\leq} i)}$ for every $i$. Hence, indeed, $\fatf{V} = \fatf{U}\circ G$.
\end{proof}

\begin{proof}[Proof of Theorem~\ref{thm:shaperamseyspace}]
	Let $\mathcal X$ be be a meager subset of $\M$. By Lemma~\ref{lem:open}, $\pi^{-1}(\mathcal X)$ is also meager.
	By Theorem~\ref{thm:fatramseyspace}, $\pi^{-1}(\mathcal X)$ is Ramsey null.  We need to show that $\mathcal X$ is also Ramsey null.
	Choose $[f,F]\neq \emptyset$. Let $f'\in \AM$ be such that $f=F\circ f'$ and
	$n=\depth_F(f)$.
	Let $\fat{U}$ be a fat subtree such that $\fatf{U}=F$ given by Observation~\ref{obs:shapetofat}.

	Arguing as in the proof of Lemma~\ref{lem:open}, there is $\fat{x}$ such that $\fatf{x} = \fatf{U} \circ f' = f$. Notice that $\depth_\fat{U}(\fat{x})=n$. Since $\pi^{-1}(\mathcal X)$ is Ramsey null, we know that there exists $\fat{V}\in [n,\fat{U}]$ such that $[\fat{x},\fat{V}]\cap \pi^{-1}(\mathcal {X})=\emptyset$.
	It follows that $[f,\fatf{V}]\subseteq [f,F]$ is non-empty and $[n,\fatf{V}]\cap \mathcal X=\emptyset$.

	Now let $\mathcal X$ be a set with the property of Baire. By Lemma~\ref{lem:open}, $\pi^{-1}(\mathcal{X})$ also has the property of Baire. By Theorem~\ref{thm:fatramseyspace}, we know that $\pi^{-1}(\mathcal X)$ is Ramsey.
	To see that $\mathcal X$ is also Ramsey, choose $[f,F]$. Just like in the previous case, let $f'\in \AM$ be such that $f=F\circ f'$ and $n=\depth_F(f)$.
	Obtain $\fat{U}$ with $\fatf{U}=F$ and $\fat{x}$ with $\fatf{x}=f$.
	Because $\pi^{-1}(\mathcal X)$ is Ramsey, we know that there exists $\fat{V}\in [n,\fat{U}]$ such that either $[\fat{x},\fat{V}]\subseteq \pi^{-1}(\mathcal X)$
	or $[\fat{x},\fat{V}]\cap \mathcal \pi^{-1}(\mathcal X) =\emptyset$. Put $F'=\fatf{V}$. It follows that $F'\in [n, F]$ and either $[f,F']\subseteq \mathcal {X}$ or $[f,F']\cap \mathcal {X}=\emptyset$.
\end{proof}
Theorem~\ref{thm:shaperamseyN} follows directly from Theorem~\ref{thm:shaperamseyspace} since every colouring of $\AM^n_k$ can be turned into a colouring of $\M^n$ by assigning the colour to the elements
of $\M^n$ according to their initial segment in $\AM^n_k$.  Corollary~\ref{cor:shaperamseyN1} follows by compactness.
Corollary~\ref{cor:shaperamseyN2} may deserve a short proof:
\begin{proof}[Proof of Corollary~\ref{cor:shaperamseyN2}]
	Fix $n,k,m,r$ as in the statement. Let $N$ be given by Corollary~\ref{cor:shaperamseyN1} and let $\chi$ be a colouring of $\AM^n_k(N)$.
	Given $f\in \AM^n_k$ put $$f'=\underbrace{F_0^{\tilde{f}(n+k)}\circ F_0^{\tilde{f}(n+k)}\circ \dots \circ F_0^{\tilde{f}(n+k)}}_{N-\tilde{f}(n+k)\hbox{ times}}\circ f,$$
	where $F_0^{\tilde{f}(n+k)}$ is the duplication function given by \ref{item:Nduplication} of Definition~\ref{def:sntree}.
	Define a colouring $\chi'$ of $\AM^n_k$ by putting $\chi'(f)=\chi(f')$.

	Let $F$ be the function given by Corollary~\ref{cor:shaperamseyN1}. Let
	$$F'=\underbrace{F_0^{\tilde{f}(n+m)}\circ F_0^{\tilde{f}(n+m)}\circ \dots \circ F_0^{\tilde{f}(n+m)}}_{N-\tilde{F}(n+m)\hbox{ times}}\circ F.$$
	$F'$ satisfies the conclusion of Corollary~\ref{cor:shaperamseyN2}.
\end{proof}
\section{Envelopes and embedding types}
\label{sec:envelope}
The following concept often plays a crucial role in applications of Ramsey-type theorems on trees~(see e.g. \cite[Section 6]{todorcevic2010introduction} or \cite{dodos2016}).
\begin{definition}[Envelope]
	Let $\SNtree$ be an $(\S,\M)$-tree and $X\subseteq T$ be a set.
	A function $f\in \M\cup \AM$ is called an \emph{envelope of $X$} if there exists set $Y\subseteq T$ such that $f[Y]=X$.
	In this setting, $Y$ is called the \emph{embedding type of $X$ in $f$}.
	For finite sets $X$, the \emph{height} of envelope $f$ is $\max\{\ell(a):a\in Y\}+1$.
\end{definition}
It is clear that the identity is an envelope of every set. However, it is often desirable to find an envelope minimizing
the height, a \emph{minimal envelope}. In this section we give a general algorithm and prove that under a mild additional assumption (which is always satisfied in the applications) it produces a minimal envelope of every set.

We start by determining some nodes that must be in the image of every envelope of set $X$.
\begin{definition}[Meet-closed set]
	Let $\SNtree$ be an $(\S,\M)$-tree and $X\subseteq T$.  We say that $X$ is  \emph{meet-closed}
	if for every $a,b\in X$ also $a\meet b\in X$.
\end{definition}
\begin{definition}[Parameter-closed set]
	Let $\SNtree$ be an $(\S,\M)$-tree, let $X\subseteq T$ be a set, and let $I\subseteq \omega$ be set of levels.  We say that $X$ is \emph{parameter-closed over $I$}
	if for every $i\in I$, $a\in X\cap T({>}i)$ and $b\in \Dp(a|_{i+1})$ it holds that $b\in X$.

	We say that $X$ is \emph{parameter-closed} if it is parameter-closed over $L(X)=\{\ell(a):a\in X\}$.
\end{definition}
Given $X\subseteq T$ and $I\subseteq \omega$, we denote by $\Cl_I(X)$ the inclusion minimal subset of $T$
containing $X$ that is both meet-closed and parameter-closed over $I$. It is easy to see that $\Cl_I(X)$ is uniquely defined.
\begin{algorithm}[Construction of an minimal envelope]
	\label{envelope}
	Let $\SNtree$ be an $(\S,\M)$-tree and let $X\subset T$ be a finite set. Put $\ell=\max \{\ell(a):a\in X\}$. Construct subsets $\{\ell\}=I_\ell\subseteq I_{\ell-1}\subseteq\cdots \subseteq I_{0}\subseteq \omega$
	of levels and functions $F_\ell,F_{\ell-1},\ldots, F_0\in \M$ by decreasing induction.

	\begin{enumerate}
		\item Put $I_\ell=\{\ell\}$ and $F_\ell$ to be the identity $\mathrm{Id}:T\to T$.

		\item Now assume that $I_{i+1}$ and $F_{i+1}$ are constructed for some $i< \ell$.
		      Level $i$ is \emph{interesting} if at least one of the following is satisfied:
		      \begin{enumerate}[label=(I\arabic*)]
			      \item\label{item:I1} \emph{Level $i$ contains a node}: There exists $a \in \Cl_{I_{i+1}}(X)$ such that $\ell(a)=i$,
			      \item\label{item:I2} \emph{Level $i$ contains a meet}: There exist $a,b \in \Cl_{I_{i+1}}(X)$ such that $\ell(a\meet b)=i$.
			      \item\label{item:I3} \emph{Level $i$ can not be skipped}: Assume that \ref{item:I1} and \ref{item:I2} are not satisfied. Put $S_i=\{a|_i:a\in \Cl_{I_{i+1}}(X)\cap T({>}i)\}$ and construct function $f_i\colon S_i\to T(i+1)$ sending $a|_i\mapsto a|_{i+1}$
			      (by our assumption such $f_i$ exists and is well-defined).
			      We say that level $i$ \emph{can not be skipped} if $f_i$ does not extend to a function in $\M$ skipping only level $i$.
		      \end{enumerate}
		      If $i$ is not interesting, let $H_i\in \M$ be some function extending $f_i$ constructed in the discussion of \ref{item:I3}. Put
		      $$
			      \begin{aligned}
				      F_i & =\begin{cases}
					             F_{i+1}          & \hbox{ if $i$ is interesting,}     \\
					             F_{i+1}\circ H_i & \hbox{ if $i$ is not interesting.}
				             \end{cases}  \\
				      I_i & =\begin{cases}
					             I_{i+1}\cup \{i\} & \hbox{ if $i$ is interesting,}     \\
					             I_{i+1}           & \hbox{ if $i$ is not interesting.}
				             \end{cases}
			      \end{aligned}
		      $$
	\end{enumerate}

	The output of this algorithm is the pair $(F=F_0, I=I_0)$. The set $I$ is called the set of  \emph{interesting levels} of $X$. Below we will prove that, under a minor assumption, $F$ is a minimal envelope of $X$ of height $|I|$.
\end{algorithm}

Notice that $F$ is not uniquely determined because the choice of $H_i$ in the algorithm is not uniquely specified.
The main result of this section is the following proposition.
\begin{prop}
	\label{prop:envelopes}
	Let $\SNtree$ be an $(\S,\M)$-tree. Assume that the following property is satisfied:
	\begin{enumerate}[label=(E\arabic*)]
		\item \label{item:E1}For every function $F\in \M$ skipping only one level and for every $a\in T$, $\bar{p}\in T^{\leq \omega}$ and $c\in \Sigma$
		      such that $\S(F(a),F(\bar{p}),c)$ is defined, $\S(a,\bar{p},c)$ is also defined.
	\end{enumerate}
	Given a finite set $X\subset T$, Algorithm~\ref{envelope} will produce a minimal envelope $F$ of $X$ with a set $I$ of interesting levels of $X$.
	Moreover:
	\begin{enumerate}
		\item The height of envelope $F$ is $|I|$ and $I=\tilde{F}[\omega]\cap \max \{\ell(a)+1:a\in X\}$.
		\item The set $I$ and the embedding type of $X$ in $F$ are uniquely determined, independent of the particular choice of functions $H_i$.
	\end{enumerate}
\end{prop}
\begin{example}
	To see that the additional assumption~\ref{item:E1} is necessary, consider the subset $T$ of $\omega^{<\omega}$ of all finite sequences
	$w=(w_i)_{i\in n}$ satisfying $w_i\leq i$ for every $i\in n$.  The tree $(T,\sqsubseteq)$ can be equipped with
	a natural successor operation of concatenation. Put $\Sigma=\omega$
	and define $\S(w,\bar{p},c)=w\conc c$ if and only if $w\conc c\in T$  and $\bar{p}=\emptyset$.

	Consider set $X_w=\{w,w\conc 7\}$ for some $w\in T$ of length at least $7$. And observe that Algorithm~\ref{envelope} will produce $I = \{\lvert w\rvert, \lvert w\rvert + 1\}$, while every minimal envelope of $X_w$ has height $9$
	and the additional levels are not uniquely determined.
\end{example}
We devote the rest of this section to a proof of Proposition~\ref{prop:envelopes}.
The following result is a consequence of
Proposition~\ref{prop:shape-pres} \ref{item:passingnumber} and \ref{item:meets}.
\begin{observation}
	\label{obs:closure}
	If $F\in N$ is an envelope of $X$, then it is also an envelope of $\Cl_I(X)$ for every $I\subseteq \tilde{F}[\omega]$.
\end{observation}

Correctness of Algorithm~\ref{envelope} follows from the following invariant:
\begin{lemma}
	\label{lem:invariant}
	In Algorithm~\ref{envelope},
	for every $i\leq \ell=\max \{\ell(a):a\in X\}$, it holds that $F_i$ is an envelope of $\Cl_{I_i}(X)$ and it is skipping exactly levels $\{i,i+1,\ldots,\ell\} \setminus I_i$.
	Consequently, $I_i=\{j\in \tilde{F}_i[\omega]:i\leq j\leq \ell\}$.
\end{lemma}
\begin{proof}
	We proceed by decreasing induction on $i$.  For $i=\ell$ the statement follows trivially, since $F_\ell$ is the identity which is an envelope of every set.

	Now assume that the induction hypothesis is satisfied for $i+1$. If level $i$ is interesting then both statements follow trivially. We now consider the case that level $i$ is not interesting. Let $H_i$ be the function used by the algorithm, that is, $F_i=F_{i+1}\circ H_i$.

	Recall that $H_i$ is a function skipping only level $i$. Using the induction hypothesis we get that, indeed, $F_i$ is skipping exactly levels $\{i,i+1,\ldots,\ell\} \setminus I_i$, and consequently $I_i=\{j\in \tilde{F}_i[\omega]:i\leq j\leq \ell\}$.

	It remains to prove that $F_i$ is an envelope of $\Cl_{I_i}(X)$.	Put $X_{i+1}=F_{i+1}^{-1}[\Cl_{I_{i+1}}(X)]$.

	\begin{claim}\label{cl:envelopes}
		For every $a\in X_{i+1}$ and $i<j\leq \ell(a)$ it holds $a\in H_i[T]$ and $a|_j\in H_i[T]$.
	\end{claim}
	Assuming that Claim~\ref{cl:envelopes} holds, the proof of Lemma~\ref{lem:invariant} is immediate: We have that $I_{i+1} = I_i$, hence $\Cl_{I_{i+1}}(X) = \Cl_{I_{i}}(X)$. By the claim we have, in particular, $a\in H_i[T]$ for every $a\in X_{i+1}$, and so indeed $\Cl_{I_i}(X) \subseteq F_i[T]$.

	It remains to verify Claim~\ref{cl:envelopes}. Note that $H_i$ skips only level $i$. Hence, if $\ell(a) < i$ we have that $a\in H_i[T]$. Note also that there is no $a\in X_{i+1}$ with $\ell(a) = i$, as $F_{i+1}$ is the identity on $T({\leq}i)$, hence we would have $a\in X$ which would mean that level $i$ is interesting by~\ref{item:I1}. So we can assume that $\ell(a) \geq i+1$ and we will proceed by induction on $j$ (note that $a = a|_{\ell(a)}$).

	Let $j=i+1$ and pick an arbitrary $a\in X_{i+1}$ with $\ell(a) \geq j$. Since $F_{i+1}$ is shape-preserving and identity on $T({\leq}i)$, it follows that $a|_{j} \preceq F(a)$, and so $a|_j = F(a)|_j$. The claim then follows from the choice of functions $f_i$ and $H_i$: We have $a|_j = F(a)|_j=H_i(F(a)|_{i})$.

	Now assume that the claim is satisfied for some $j\geq i$. For a given $a\in X_{i+1}\cap T({>}j)$
	we want to verify that $a|_{j+1}\in H_i[T]$.
	Since $\ell = \max(\ell(a) : a\in X)$, it follows that $\Cl_{I_{i+1}}(X) = \Cl_{I_{i+1} \cup (\omega \setminus (\ell+1))}(X)$, and so
	$X_{i+1}=F_{i+1}^{-1}[\Cl_{I_{i+1} \cup (\omega \setminus (\ell+1))}(X)]$.
	By the (outer) induction hypothesis we know that $F_{i+1}$ is skipping exactly levels $\{i+1,\ldots,\ell\}\setminus I_{i+1}$, hence $X_{i+1}=\Cl_{\omega\setminus (i+1)}(X_{i+1})$. This implies that $\Dp(a|_{j+1})$ consists of members of $X_{i+1}$

	By the (inner) induction hypothesis we thus know that $\bar{p}=H_i^{-1}(\Dp(a|_{j+1}))$ is defined, and we have some $b=H_i^{-1}(a|_{j})$.
	By~\ref{item:E1} there is $a'=\S(b,\bar{p},\Dc(a|_{j+1}))$.
	Since $H_i$ is shape-preserving, we have $$H_i(a')\succeq \S(H_i(b),H_i(\bar{p}), \Dc(a|_{j+1}))=\S(a|_j,\Dp(a|_{j+1}),\Dc(a|_{j+1}))=a|_{j+1}.$$
	As $H_i$ skips only level $i$, there is in fact equality, and thus indeed $a|_{j+1}\in H_i[T]$.
\end{proof}
\begin{proof}[Proof of Proposition~\ref{prop:envelopes}]
	Let $X$, $F$ and $I$ be as in the statement.  By Lemma~\ref{lem:invariant}, we know that $F$ is an envelope of height
	$|I|$ and $I=\tilde{F}[\omega]\cap \ell$.
	It is also easy to verify from the algorithm itself that the decision whether a level is interesting does not does not depend on the functions $F_\ell,F_{\ell-1},\ldots, F_0$ constructed.

	\medskip

	It remains to discuss minimality and uniqueness of $I$.  Suppose for a contradiction that there is a set $X$ for which the algorithm produces an
	envelope $F$ of height $n$ while there exists an envelope $F'$ of smaller height $m<n$.

	Let $i$ be the maximal level in $\tilde{F}[n]\setminus \tilde{F'}[m]$. The algorithm considered it interesting for some reason.
	Notice that the set $I_{i+1}$ has the property that $I_{i+1}\subseteq \tilde{F}[n]$ and $I_{i+1}\subseteq \tilde{F'}[m]$.
	If $i$ was deemed interesting by \ref{item:I1} or \ref{item:I2} we know that $\Cl_{I_{i+1}}(X)$
	contains a node of level $i$ and, by Observation~\ref{obs:closure}, we have $i\in \tilde{F}'[m]$ because $F'$
	is also an envelope.

	We thus conclude that level $i$ can not be skipped.
	Let $f_i$ be the function constructed in \ref{item:I3}.
	Let $m'\in \omega$ be the minimal level such that $\tilde{F}'(m')>i$ and put $f=F'\restriction_{T({{\leq} m})}$.
	By an application of Proposition~\ref{prop:shape-split} on $f$ and $i+1$, we have $f'=f|_{i+1}\in \AM$.
	By another application of Proposition~\ref{prop:shape-split}, this time on $f'$ and $i$, we obtain a decomposition $f'=g\circ f'|_i$.
	Notice that $g^+\in \M$ skips only level $i$ and extends $f_i$. A contradiction.
\end{proof}
As a consequence of Algorithm~\ref{envelope} we thus obtain a method useful for characterising minimal envelopes
in a given $(\S,\M)$-tree. Its key ingredient is an understanding of which
functions $f$ constructed in \ref{item:I3} can be extended to functions
in $\M$ skipping only one level.

\section{Applications}
\label{sec:applications}
\subsection{Regularly branching trees and combinatorial cubes}
Most combinatorial applications of Ramsey-type theorems on trees use regularly branching trees $(\Sigma^{\leq n},\sqsubseteq)$ or combinatorial cubes $\Sigma^n$ (for some $n\in \omega$ and finite alphabet $\Sigma$), see e.g. \cite{devlin1979,Nevsetvril1989,Nevsetvril1987,Laflamme2006,PromelBook,Hubicka2016,bhat2016ramsey,Hubicka2020CS,braunfeld2023big}.
The general strategy applied by such proofs is to invent a representation of objects of interest using words in finite alphabet $\Sigma$
in a way so that the subtrees (or embeddings) used by the Ramsey-type theorem preserve key features of this representation.
This transfers colouring of subobjects of interest to colouring of subtrees.

An advantage of our new Ramsey-type theorem  is the flexibility of choice of the monoid $\M$
which can be tailored for a given application.
In this section we formulate a version of our main results which makes such applications quite straightforward.
We discuss how this relates to classical theorems (Milliken tree theorem~\cite{Milliken1979}, Carlson--Simpson theorem~\cite{carlson1984} and Graham--Rothschild theorem~\cite{Graham1971})
and show how a specific choice of monoid $\M$ simplifies the proof of the Abramson--Harrington (or unrestricted Ne\v set\v ril--R\"odl)
theorem~\cite{Abramson1978,Nevsetvril1977}.

\begin{definition}[Boring extensions]
	\label{def:boringext}
	Given finite alphabet $\Sigma$, a \emph{family of boring extensions} is a countable sequence $\mathcal E=(\mathcal E_n)_{n\in \omega}$ of sets $$\mathcal E_n\subseteq \{e \hbox{ is a function }e\colon \Sigma^n\to \Sigma\}$$ satisfying the following two conditions:
	\begin{enumerate}[label=(B\arabic*)]
		\item \emph{Duplication}:
		      \label{item:boring:duplication}
		      For every $m<n$ the set $\mathcal E_n$ contains a function $e^n_m\colon \Sigma^n\to \Sigma$ defined by:
		      $$e^n_m(a)=a_m \hbox{ for every $a\in \Sigma^n$}.$$
		\item \emph{Insertion:}
		      \label{item:boring:insertion}
		      For every $m\leq n$, $e_1\in \mathcal E_m$, $e_2\in \mathcal E_n$ there exists $e_3\in \mathcal E_{n+1}$ such that
		      for every $a\in \Sigma^m$ and $b\in \Sigma^{n-m}$ the following is satisfied:
		      $$e_3(a\conc e_1(a)\conc b)= e_2(a\conc b).$$
	\end{enumerate}
\end{definition}
Every extension $e\in \mathcal E_n$ can be seen as an embedding of $\Sigma^n$ to $\Sigma^{n+1}$ by mapping every $w\in \Sigma^n$ to an element of $\Sigma^{n+1}$ via concatenation: $$w\mapsto w\conc e(w).$$
Each family $\mathcal E$ thus gives us notions of interesting levels and embedding types as follows:
\begin{definition}[Interesting levels]
	Given a finite alphabet $\Sigma$, a family of boring extensions $\mathcal E$ and a set $X\subseteq \Sigma^{{<}\omega}$ we denote by $I_{\mathcal E}(X)$ the set of \emph{interesting levels of $X$}. This is a set of all $\ell\in \omega$ such that there is no $e_\ell\in \mathcal E_\ell$ satisfying for every $a\in X$, $|a| \geq \ell$
	$$(a|_\ell)\conc e_\ell(a|_\ell)\sqsubseteq a.$$
\end{definition}
Notice that it always holds that $\{|w|:w\in X\}\subseteq I_\mathcal E(X)$ and $\{|w\meet w'|:w,w'\in X\}\subseteq  I_\mathcal E(X)$.
\begin{definition}[Embedding type]
	Given a finite alphabet $\Sigma$, a family of boring extensions $\mathcal E$ and a set $X\subseteq \Sigma^{{<}\omega}$ we define the \emph{embedding type of $X$}, denoted $\tau_{\mathcal E}(X)$, to be the set of all words created from words in $X$ by removing all letters with indices not in $I_\mathcal{E}(X)$.
\end{definition}

Given a finite alphabet $\Sigma$, a family of boring extensions $\mathcal E$ and sets $X,Y\subseteq \Sigma^{{<}\omega}$ we put $$\ebinom{Y}{X}=\left\{X'\subseteq Y:\tau_{\mathcal E}\left(X'\right)=\tau_{\mathcal E}\left(X\right)\right\}.$$
(This is the set of all subsets of $Y$ of a given embedding type.)
The Ramsey theorem for colouring sets of a given embedding type can be formulated either for trees or combinatorial cubes:
\begin{theorem}[Colouring sets of given embedding type in a finite tree]
	\label{thm:boring1}
	For every finite alphabet $\Sigma$, family of boring extensions $\mathcal E$, positive integer $r$ and finite sets $X,Y\subseteq \Sigma^{{<}\omega}$
	there exists $N\in \omega$ such that for every $r$-colouring $\chi\colon \ebinom{\Sigma^{{<}N}}{X}\to r$
	there exists $Y'\in \ebinom{\Sigma^{{<}N}}Y$ such that $\chi$ is constant when restricted to $\ebinom{Y'}{X}$.
\end{theorem}
\begin{theorem}[Colouring sets of given embedding type in a combinatorial cube]
	\label{thm:boring2}
	For every finite alphabet $\Sigma$, family of boring extensions $\mathcal E$, positive integers $m,n,r$ and (finite) sets $X\subseteq \Sigma^m$, $Y\subseteq \Sigma^n$
	there exists $N\in \omega$ such that for every $r$-colouring $\chi\colon\ebinom{\Sigma^N}{X} \to r$
	there exists $Y'\in \ebinom{\Sigma^N}Y$ such that $\chi$ is constant when restricted to $\ebinom{Y'}{X}$.
\end{theorem}

To relate Theorems~\ref{thm:boring1} and~\ref{thm:boring2} to Corollaries \ref{cor:shaperamseyN1} and \ref{cor:shaperamseyN2}
we only need to build the corresponding $(\S,\M)$-trees.
Given an alphabet $\Sigma$ we consider tree $(\Sigma^{{<}\omega},\sqsubseteq)$
and the corresponding
$\S$-tree $(\Sigma^{{<}\omega},\allowbreak {\sqsubseteq},\allowbreak \Sigma,\allowbreak \S)$ with $\S$ being the concatenation: Given $a\in \Sigma^{{<}\omega}$, $c\in \Sigma$ and $\vec{p}$ we define
$\S(a,\vec{p},c)$ if and only if $\vec{p}$ is empty sequence and put $$\S(a,\vec{p},c)=a\conc c.$$
In this context Definition~\ref{def:shape-pres} can be simplified as follows.
\begin{definition}[Shape-preserving functions on regularly branching trees]
	Function $F\colon \Sigma^{{<}\omega}\to \Sigma^{{<}\omega}$ is \emph{shape-preserving} if
	\begin{enumerate}[label=(\roman*)]
		\item\label{def:shape-pres:sitem1} $F$ is level preserving, and
		\item\label{def:shape-pres:sitem2} $F$ preserves passing numbers:
		for every $a\in \Sigma^{{<}\omega}$ and $\ell<\ell(a)$ it holds
		that $F(a)_{\tilde{F}(\ell)}=a_\ell$.
	\end{enumerate}
\end{definition}
Every set of boring extensions $\mathcal E$ can be associated with a corresponding monoid of shape-preserving functions:
\begin{definition}
	Given a finite alphabet $\Sigma$ and a family of boring extensions $\mathcal E$ we denote by $\mathcal M_\mathcal E$ the set of all shape-preserving functions $F\colon\Sigma^{{<}\omega}\to \Sigma^{{<}\omega}$
	such that $$I_\mathcal E(F[\Sigma^{{<}\omega}])=\tilde F[\omega].$$
\end{definition}
More technically this condition can be stated as follows:
\begin{observation}
	\label{obs:shapeboring}
	Given a finite alphabet $\Sigma$, family of boring extensions $\mathcal E$
	and a shape-preserving function $F\colon \Sigma^{{<}\omega}\to\Sigma^{{<}\omega}$ it holds that $F\in \M_\mathcal E$ if and only if
	for every level $\ell$ skipped by $F$ (i.e., $\ell\in \omega\setminus \tilde{F}(\omega)$) there exists a (not necessarily unique) extension $e^F_\ell\in \mathcal E_\ell$ such that for every $a\in \Sigma^{{<}\omega}$ with $\ell(F(a))>\ell$ it holds that
	$$(F(a)|_\ell)\conc e^F_\ell(F(a)|_\ell) \sqsubseteq F(a).$$
\end{observation}
Observation~\ref{obs:shapeboring} is useful both for verifying that a given function $F$ is in $\M_\mathcal E$ and for constructing a function in $F'\in \M_\mathcal E$ by giving extensions for every level skipped by $F'$.
We are now ready to show:
\begin{prop}
	\label{prop:boringmonoid}
	For every finite alphabet $\Sigma$ and family of boring extensions $\mathcal E$ it holds that $(\Sigma^{{<}\omega},\allowbreak \sqsubseteq,\allowbreak \Sigma,\M_\mathcal E)$ is an
	$(\mathcal S,\mathcal M)$-tree.
\end{prop}
\begin{proof}
	We verify that $\mathcal M_\mathcal E$ satisfies Definition~\ref{def:sntree}.

	First consider the identity $\mathrm{Id}\colon \Sigma^{{<}\omega}\to \Sigma^{{<}\omega}$. It is clearly shape-preserving and
	since $\omega\setminus \tilde{\mathrm{Id}}[\omega]=\emptyset$,
	by Observation~\ref{obs:shapeboring} we have $\mathrm{Id}\in \mathcal M_\mathcal E$.

	Next we verify $\mathcal M_\mathcal E$ is closed for composition.
	Fix $F, F'\in \mathcal M_\mathcal E$. By induction on level $\ell$, we verify that function $H=F\circ F'$ satisfies Observation~\ref{obs:shapeboring}.
	Towards that, for every $\ell\notin \tilde{H}[\omega]$ we find corresponding extension $e^H_\ell\in \mathcal E_\ell$.
	Given $\ell\notin \tilde{H}[\omega]$ we consider two cases:
	\begin{enumerate}
		\item $\ell\notin \tilde{F}[\omega]$. In this case let $e^F_\ell\in \mathcal E_\ell$ be given by Observation~\ref{obs:shapeboring}.
		      It is easy to check that we can put $e^H_\ell=e^F_\ell$.
		\item $\ell\in \tilde{F}[\omega]$. Let $\ell'=\tilde{F}^{-1}(\ell)$. Notice that $\ell'\notin \tilde{F}'[\omega]$. Let $e^{F'}_{\ell'}\in \mathcal E_{\ell'}$ be given by Observation~\ref{obs:shapeboring}.
		      If $\ell'=\ell$ we can put $e^H_\ell=e^{F'}_{\ell'}$.   Otherwise we proceed by induction and produce extensions $e^{F'}_{\ell'}=e'_{\ell'},e'_{\ell'+1}$,\ldots, $e'_\ell$.
		      Assume that extension $e'_n\in \mathcal E_n$ was built for some $n$ satisfying $\ell'\leq n<\ell$. Let $m$ be the $(n+1-\ell')$-th element of $\omega\setminus \tilde{F}[\omega]$ and $e^F_{m}\in \mathcal E_m$ the extension
		      given by Observation~\ref{obs:shapeboring}. Construct $e'_{n+1}$ by an application of \ref{item:boring:insertion} for $m$, $n$,
		      $e_1=e^F_m$ and $e_2=e'_{\ell'}$.
		      Finally, put $e^H_\ell=e'_\ell$.
	\end{enumerate}
	We have thus proved that $\mathcal M_\mathcal E$ is a monoid.
	Notice that by Observation~\ref{obs:shapeboring} the monoid is closed which completes verification of condition \ref{item:Nmonoid} of Definition~\ref{def:sntree}.

	Condition \ref{item:Nduplication} of Definition~\ref{def:sntree} follows by \ref{item:boring:duplication}.

	Finally we verify \ref{item:Ndecomposition}. Given $n\in \omega$ and $F\in \M_\mathcal E$ skipping level $\tilde{F}(n)-1$ such that $\tilde{F}(n)>0$
	we can construct $F_2$ skipping only level $m=\tilde{F}(n)-1$ using Observation~\ref{obs:shapeboring} and extension $e^F_m$.
	Function $F_1$ can be also constructed by means of Observation~\ref{obs:shapeboring}: for every $\ell\in \omega\setminus \tilde{F}[\omega]$ satisfying $\ell\neq \tilde{F}(n)-1$ put
	$e^{F_1}_\ell=e^F_\ell$.
\end{proof}

\begin{proof}[Proof of Theorem~\ref{thm:boring1}]
	Fix a finite alphabet $\Sigma$, a family of boring extensions $\mathcal E$, a positive integer $r$ and sets $X,Y\subseteq \Sigma^{{<}\omega}$.
	Without loss of generality assume that $\tau_\mathcal E(X)=X$ and $\tau_\mathcal E(Y)=Y$.
	Let $k$ be the maximal length of a word in $X$ and $m$ be the maximal length of a word in $Y$.
	By Proposition~\ref{prop:boringmonoid} we know that  $(\Sigma^{{<}\omega},\allowbreak \sqsubseteq,\allowbreak \Sigma,\M_\mathcal E)$ is an $(\mathcal S,\mathcal M)$-tree.
	Let $N$ be given by Corollary~\ref{cor:shaperamseyN1} on this $(\mathcal S,\mathcal M)$-tree for $n=0$, $k$, $m$.
	Notice that for every $Z\subseteq \Sigma^{{<}\omega}$ Algorithm~\ref{envelope} will give precisely $\tau_\mathcal E(Z)$. Consequently, a colouring $\chi\colon\ebinom{\Sigma^N}Y\to r$
	yields a colouring $\chi'\colon (\AM_\mathcal E)^0_k({\leq}N)\to r$ by putting $\chi'(f)=\chi(f[\tau_\mathcal E(X)]$ and this correspondence is one to one.
	Let $f$ be given by Corollary~\ref{cor:shaperamseyN1}. It follows that $f[Y]$ has the desired property.
\end{proof}
Theorem~\ref{thm:boring2} follows in a complete analogy.

We now relate our result to classical Ramsey-type theorems on trees.

\subsubsection{Milliken tree theorem}
In order to formulate the Milliken tree theorem we need to state some standard definitions, see e.g.~\cite{todorcevic2010introduction}.
\begin{definition}[Strong subtree]
	Let $(T,\preceq)$ be a tree with a single root vertex.  Recall that $S\subseteq T$ is a \emph{strong subtree} of height $n\in \omega+1$ if and only if:
	\begin{enumerate}
		\item $S$ is closed for meets: for every $a,b\in S$ it holds that $a\meet b\in S$.
		\item $S$ is level preserving: there exists a function $\tilde{S}:n\to\omega$ such that $S(\ell)\subseteq T(\tilde S(\ell))$ for every $\ell<n$.
		\item For every $a\in S$ of level $\ell < n-1$ in tree $S$ and every immediate successor $b$ of $a$ in $T$ there is a unique $c\in S(\ell+1)$ such that $a\sqsubseteq b\sqsubseteq c$.
		\item If $n$ is finite then $S({\geq} n)=\emptyset$.
	\end{enumerate}
\end{definition}
Fix a finite alphabet $\Sigma$ and put $T=\Sigma^{{<}\omega}$ to be an infinite $|\Sigma|$-ary branching tree.
In this setup, every infinite strong subtree $S$ can be also seen as a \emph{strong embedding} $F_S\colon T\to S\subseteq T$
constructed via induction over levels.
\begin{enumerate}
	\item Given the root $r$ of $T$ put $F_S(r)$ to be the root of $S$.
	\item Let $a\in S$ be such that $F_S(a)$ is has been already constructed. Given $c\in \Sigma$, put $F_S(a\conc c)$ to be the unique successor of $a\conc c$ in $T(\tilde{S}(\ell(a)+1))$.
\end{enumerate}

It follows directly from the definition that every strong embedding is also a shape-preserving function (in the $\S$-tree  $(\Sigma^{{<}\omega},\sqsubseteq,\Sigma,\S)$), and vice versa, every shape-preserving function $F$ is a strong embedding and defines a strong subtree $F[\Sigma^{{<}\omega}]$.

Denote by $\Str^k_n(T)$ the set of all strong subtrees $S$ of $T$ of height $n$ such that $\tilde{S}(\ell)=\ell$ for every $\ell<k$.
Let $\Emil=(\Emil_n)_{n\in \omega}$ be the maximal sequence of sets satisfying Definition~\ref{def:boringext}, that is, for every $n\in \omega$, the set $\Emil_n$ simply consists of all functions $\Sigma^n\to \Sigma$. Let $\Mmil$ be the corresponding monoid of shape-preserving functions.
By Proposition~\ref{prop:boringmonoid} we obtain an $(\mathcal S,\mathcal M)$-tree $(\Sigma^{{<}\omega},\allowbreak \sqsubseteq,\allowbreak \Sigma,\Mmil)$.
An immediate consequence of Theorem~\ref{thm:shaperamseyN} is:
\begin{theorem}[Milliken tree theorem~\cite{Milliken1979}]
	Let $\Sigma$ be a finite alphabet and $(\Sigma^{{<}\omega},\subseteq)$ be the infinite $|\Sigma|$-ary tree. Then for every $k\leq n$ and every finite colouring of $\Str^k_n(T)$ there exists $S\in \Str^k_\omega(T)$
	such that the set $\{Q\in \Str^k_n(T):Q\subseteq S\}$ is monochromatic.
\end{theorem}
This is a special from of the Milliken tree theorem for regularly branching trees.
As a consequence of Theorem~\ref{thm:shaperamseyspace} we can also build the associated Ramsey space~\cite[Theorem 6.3]{todorcevic2010introduction}.
\begin{remark}
	It is usual to formulate Milliken tree theorem for colouring
	product strong subtrees of finite product of trees. However, the formulation above
	is not weaker. Colouring $\Str^k_n(T)$ is equivalent to colouring product strong
	subtrees of $|\Sigma|^k$ trees. We chose this formalism because it is easier
	to generalize to irregularly branching trees as done in recent papers on coding trees~\cite{dobrinen2017universal,zucker2020}.

	On the other hand, the full strength of Milliken tree theorem for trees with finite but unbounded branching can not be
	derived directly from Theorem~\ref{thm:shaperamseyN}. There is a deeper reason for that: while Milliken tree theorem
	for regularly branching tree can use Hales--Jewett theorem as a pigeonhole, for unbounded branching trees, the Halpern--L\"auchli
	theorem is necessary.  While the Halpern--L\"auchli theorem has multiple proofs~\cite{Halpern1966,Milliken1979,todorcevic2010introduction,dodos2016}, they are all very different from proofs of the
	Hales--Jewett theorem.
\end{remark}

\subsubsection{Dual Ramsey theorem and Carlson--Simpson theorem}
The dual Ramsey theorem and its infinitary generalization, the Carlson--Simpson theorem, are usually not formulated in a language of trees and subtrees, they rather speak
about colouring equivalences (or partitions)~\cite{todorcevic2010introduction}. (A common equivalent formalism is also using the rigid surjections~\cite{Solecki2012,PromelBook}.)
Interpreting these theorems as theorems about trees and subtrees was instrumental in some recent progress
on big Ramsey degrees~\cite{Hubicka2020CS,Hubickabigramsey2,Balko2023}.
\begin{theorem}[Finite Dual Ramsey Theorem]
	\label{thm:dualRamsey}
	For every $k, r, m\in \omega\setminus\{0\}$, there is $N\in \omega$ such that for every $r$-colouring
	of the family $\Eqn_k (N)$
	with exactly $k$ classes there is an equivalence relation $E$ on $N$ elements with exactly
	$m$ classes such that $\{F \in \Eqn_k (N): F\hbox{ is coarser than }E\}$ is monochromatic.
\end{theorem}
Let $\Eqn_\infty$ be the set of all equivalence relations on $\omega$ with infinitely many equivalence classes.
Given $E_1,E_2\in \Eqn_\infty$ put $E_1<E_2$ if $E_1$ is coarser than $E_2$.
Given $n\in \omega$ and a finite equivalence relation $E$ of $n$, we say that $E'\in \Eqn_\infty$
\emph{extends} $E$ if and only if
\begin{enumerate}
	\item $E'\restriction n\times n=E$, and,
	\item $n$ is a minimal element of some equivalence class of $E'$.
\end{enumerate}
The \emph{Ellentuck topology} on $\Eqn_\infty$ is given by the following basic open sets:
$$[E,E']=\{E''\in \Eqn_\infty:E'' \hbox { extends $E$ and } E''<E'\}$$
for every finite partition $E$ of some $n\in \omega$ and $E'\in \Eqn_\infty$.
Finally, to define Ramsey and Ramsey null sets in analogy to Definition~\ref{def:ramsey}, we define $\depth_{E'}(E)$ to be finite if and only if
\begin{enumerate}
	\item $E$ extends to a partition $E''\in \Eqn_\infty$ such that $E''<E$, and
	\item $n$ is a minimal element of the $m$-th equivalence class of $E$ for some $m\in \omega$.
\end{enumerate}
In this case we put $\depth_{E'}(E)=m$
\begin{theorem}[Carlson--Simpson Theorem~\cite{carlson1984}]
	\label{thm:CS}

	$\Eqn_\infty$ with the Ellentuck topology is a topological Ramsey space: Every property of Baire subset of $\Eqn_\infty$ is Ramsey and every meager subset is Ramsey null.
\end{theorem}
Consider a finite alphabet $\Sigma$ with at least two elements and let $\Ecs=(\Ecs_n)_{n\in \omega}$ be the minimal sequence of sets satisfying Definition~\ref{def:boringext}.
In this case, for every $n\in \omega$, the set $\Ecs_n$ consists precisely of functions $e_n^m$, $m<n$, required by condition~\ref{item:boring:duplication} of Definition~\ref{def:boringext}.
It is easy to check that condition~\ref{item:boring:insertion} is also satisfied and one can make Observation~\ref{obs:shapeboring} more precise:
\begin{observation}
	\label{obs:shapeboringCS}
	Given a finite alphabet $\Sigma$ with at least 2 elements and a
	shape-preserving function $F\colon \Sigma^{{<}\omega}\to\Sigma^{{<}\omega}$ it holds that $F\in \Mcs$ if and
	only if for every level $\ell$ skipped by $F$ there exists unique level
	$n\in \tilde{F}[\omega]$, $n<\ell$, such that
	for every $a\in \Sigma^{{<}\omega}$ with $\ell(F(a))>\ell$ it holds that
	$$(F(a)|_\ell)\conc F(a)_n \sqsubseteq F(a).$$
\end{observation}
Every function $F\in \Mcs$ thus corresponds to an equivalence relation on $\omega$ as follows: Every level
in $F[\omega]$ is the minimal element of its equivalence class, and levels
$\ell\in \omega\setminus F[\omega]$ are split into equivalence classes according to
levels $n$ given by Observation~\ref{obs:shapeboringCS}.

Notice that this correspondence is one-to-one: for every equivalence relation $E$ on $\omega$ there exists a
unique $F_E\in \Mcs$. Moreover, if relation $E'$ is coarser than $E$ then there is $F\in \Mcs$ such that
$F_{E}\circ F=F_{E'}$.  It follows that Theorem~\ref{thm:dualRamsey} is a direct consequence of Corollary~\ref{cor:shaperamseyN2}
and Theorem~\ref{thm:CS} a direct consequence of Theorem~\ref{thm:shaperamseyspace}.

\subsubsection{Graham--Rothschild theorem}
Given a finite alphabet $\Sigma$ not containing characters $\{\lambda_i:i\in \omega\}$ and $k\in \omega+1$, a \emph{$k$-parameter word in alphabet $\Sigma$} is a (possibly infinite) sequence (word) $W$ in
alphabet $\Sigma\cup \{\lambda_i\colon 0\leq i<k\}$ such that for every $0\leq i  k$, $\lambda_i$ occurs in $W$, and for every $0< i < k$, the first
occurrence of $\lambda_i$ appears after the first occurrence of $\lambda_{i-1}$. A \emph{parameter word} is a $k$-parameter word for some $k\in \omega+1$.

Let $W$ be an $n$-parameter word and let $U$ be a parameter word of length $k\leq n$ (where $k,n\in \omega+1$). Then we denote by
$W(U)$ the parameter word created by \emph{substituting} $U$ to $W$. More precisely, this is a parameter word created from $W$ by replacing each occurrence of $\lambda_i$, $0\leq i < k$, by $U_i$ and truncating it just
before the first occurrence of $\lambda_k$ (in $W$).
Given an $n$-parameter word $W$ and set $S$ of parameter words of length at most $n$, we denote by $W(S)$ the set $\{W(U)\colon U\in S\}$.

We will prove:
\begin{theorem}[Graham--Rothschild Theorem~\cite{Graham1971}]
	\label{thm:GR1}
	For every pair of finite alphabets $\Pi\subseteq \Sigma$,
	for every $k,r,m\in \omega$ such that $0<k<m$ and $r>0$, there is a positive integer $N$ such that for every
	$r$-colouring of $k$-parameter words of length $N$ in the alphabet $\Sigma$ there exists an $m$-parameter
	word $W$ of length $N$ in alphabet $\Pi$ such that the set $$\{W(U):U\hbox{ is a $k$-parameter word in alphabet $\Sigma$ of length $m$}\}$$
	is monochromatic.
\end{theorem}
\begin{theorem}[$*$-version of the Graham--Rothschild Theorem~\cite{Graham1971}]
	\label{thm:GR2}
	For every pair of finite alphabets $\Pi\subseteq \Sigma$,
	for every $k,r,m\in \omega$ such that $0<k<m$ and $r>0$ there is a positive integer $N$ such that for every
	$r$-colouring of $k$-parameter words of length at most $n$ in the alphabet $\Sigma$ there exists an $m$-parameter
	word $W$ in the alphabet $\Pi$ of length at most $N$ such that the set $$\{W(U):U\hbox{ is a $k$-parameter word in alphabet $\Sigma$  of length at most $m$}\}$$
	is monochromatic.
\end{theorem}
Given finite alphabets $\Pi\subseteq \Sigma$ where $\Sigma$ has at least two elements ($\Pi$ can be empty) we define a family of boring extensions $\Egr=(\Egr_n)_{n\in \omega}$
such that $\Egr_n$ consists precisely of
\begin{enumerate}
	\item all functions $e_n^m$, $m<n$, required by condition~\ref{item:boring:duplication} of Definition~\ref{def:boringext}, and
	\item all constant functions $e:\Sigma^n\to \Pi$.
\end{enumerate}
It is easy to verify that $\Egr_n$ satisfies condition~\ref{item:boring:insertion} and that, similarly to Observation~\ref{obs:shapeboringCS}, every
function $F\in \Mgr$ uniquely corresponds to an infinite-parameter word in alphabet $\Pi$.

Theorem~\ref{thm:GR1} now follows by Corollary~\ref{cor:shaperamseyN2} and Theorem~\ref{thm:GR2} follows by Corollary~\ref{cor:shaperamseyN1}.
We remark that the original formulation of the Graham--Rothschild Theorem also considers parameter words over finite groups.
These also naturally yield a corresponding family of boring extensions, we however leave the details out for brevity.

Notice also that $\Ecs\subseteq \Egr \subset \Emil$. While the three theorems are formulated and proven using different formalisms,
they all fit our generalized Ramsey theorem.  There are many other meaningful choices of the families of boring extensions which
yield Ramsey-type theorems. This is useful for applications, see~\cite{Balko2023}. In the following section we show an application that uses such a custom family.
\subsection{Unrestricted Ne\v set\v ril--R\"odl (or Abramson--Harrington) theorem}
\label{sec:AH}
Let $L$ be a relational language.  We consider standard model-theoretic
$L$-structures which we denote by bold letters $\str{A},\str{B},\ldots$ and
their vertex-sets (domains or underlying sets)  by corresponding normal letters $A,B,\ldots$
We consider finite and countably infinite structures only.
Given an $L$-structure $\str{A}$ and a set $S\subseteq A$, we write $\str{A}\restriction_S$ for the $L$-structure induced by $\str{A}$ on $S$.
Substructures are always induced: $\str{A}$ is a substructure of $\str{B}$ if and only if $A\subseteq B$ and $\str{A}=\str{B}\restriction_{A}$.

If $L$ contains a binary symbol $\leq$ we call an $L$-structure $\str{A}$ \emph{ordered} if $(A,\leq_\str{A})$ is
a linear order.
Given $L$-structures $\str{A}, \str{B}$ and $\str{C}$ we write $\Emb(\str{A},\str{B})$ for the set of all embeddings
$\str{A}\to\str{B}$.
\begin{definition}[Erd\H os--Rado partition arrow]
	We write $\str{C}\longrightarrow(\str{B})^\str{A}_2$ for the following statement:
	\begin{quote}
		For every $2$-colouring $\chi\colon\Emb(\str{A},\str{C})\to\{0,1\}$ there exists an embedding $f\in \Emb(\str{B},\str{C})$ such
		that $\chi$ restricted to $\{f\circ g:g\in \Emb(\str{A},\str{B})\}$ is constant.
	\end{quote}
\end{definition}

We give a short proof of the following theorem:
\begin{theorem}[Ne\v set\v ril--R\"odl~\cite{Nevsetvril1977}, Abramson--Harrington~\cite{Abramson1978}]
	\label{thm:ah}
	Let $L$ be a relational language and $\str{A},\str{B}$ finite ordered $L$-structures.
	Then there exists a finite ordered $L$-structure $\str{C}$ satisfying $\str{C}\longrightarrow(\str{B})^\str{A}_2$.
\end{theorem}
The proof applies Theorem~\ref{thm:boring2} but is otherwise fully self-contained.
\begin{proof}
	Fix a relational language $L$ containing a binary symbol $\leq$ and an ordered $L$-struc\-ture $\str{B}$. Without loss of generality assume that $B=n$  and
	$\leq_\str{B}$ is the natural ordering of $n$.

	Given two substructures $\str{B}'$, $\str{B}''$ of $\str{B}$ we put $\str{B}'\prec \str{B}''$ if either $|\str{B}'|<|\str{B}''|$ or $|\str{B}'|=|\str{B}''|$ and
	$B'$ is lexicographically before $B''$ (in the order of vertices of $\str{B}$).
	Put $N=2^n-1$ and
	enumerate all non-empty substructures of $\str{B}$ as $\str{B}^0, \str{B}^1,\ldots,\str{B}^{N-1}$ in the increasing order (given by $\preceq$).
	For each $i<N$, let $\str{D}^i$ be the lexicographically first substructure of $\str B$ isomorphic to $\str{B}^i$ and denote by $f^i$ the unique isomorphism $\str{B}^i\to\str{D}^i$. Put $\Sigma=n+1$ and define a function $\varphi\colon B\to \Sigma^{N}$ by:
	$$\varphi(v)_i=\begin{cases}
			n      & \hbox{if $v\notin B^i$} \\
			f^i(v) & \hbox{if $v\in B^i$}
		\end{cases}
		\hbox{ for every $v\in B$ and $i<N$.}$$

	See Figure~\ref{fig:abramson1} for an example. The function $\varphi$ will serve as a representation of $\str{B}$ in the regularly branching tree $(\Sigma^{{<}\omega},\sqsubseteq)$ which will allow us to later apply Theorem~\ref{thm:boring2}.
	\begin{figure}
		\centering
		\includegraphics[scale=0.833]{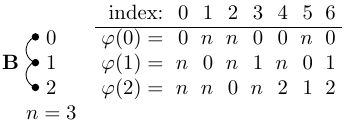}
		\caption{A representation, produced in the proof of Theorem~\ref{thm:ah}, of the ordered path.}
		\label{fig:abramson1}
	\end{figure}

	\medskip

	Given $k,\ell\in \omega$,  and a tuple  $\bar{w}=(w^0, w^1, \ldots, w^{k-1})$ of elements of $\Sigma^\ell$:
	\begin{enumerate}
		\item We say that $\bar w$ \emph{decides a structure on level $i<\ell$} if $w^0_i< w^1_i<\cdots<w^{k-1}_i<n$ and $i$ is the minimal index with this property.
		\item We say that  $\bar w$  \emph{become incompatible on level $i'<\ell$} if either
		      \begin{enumerate}
			      \item $k=2$ and $n > w^0_{i'}\geq w^1_{i'}$,
			      \item $w^0_{i'}< w^1_{i'}<\cdots<w^{k-1}_{i'} < n$ but there exists $i<i'$ such that $\bar w$ decides structure on level $i$
			            and $B\restriction_{\{w^0_i,w^1_i,\ldots,w^{k-1}_i\}}$ is not isomorphic to $B\restriction_{\{w^0_{i'},w^1_{i'},\ldots,w^{k-1}_{i'}\}}.$
		      \end{enumerate}
		\item We call $\bar w$ \emph{compatible}  for every subsequence $\bar w'$ of $\bar w$ there is no $i'$ such that $\bar w'$ becomes incompatible on level $i'$.
		      (In particular, every compatible sequence $\bar w$ is lexicographically increasing: $w^0\lexleq w^1\lexleq\cdots\lexleq w^{k-1}$.)
	\end{enumerate}
	The following summarises the main properties of the construction:
	\begin{claim}
		\label{claim:c1}
		For every $k\leq |B|$ and $v_0<v_1<\cdots<v_{k-1}\in B$ it holds that tuple $\bar{v}=(\varphi(v_0),\allowbreak \varphi(v_1),\allowbreak \ldots,\allowbreak \varphi(v_{k-1}))$ is strictly lexicographically increasing, compatible, decides structure on some level $i$, and moreover $\str{B}\restriction_{\{\varphi(v_0)_i,\allowbreak \varphi(v_1)_i,\allowbreak\ldots,\allowbreak\varphi (v_{k-1})_i\}}$ is isomorphic to $\str{B}\restriction_{\{v_0,v_1,\ldots v_{k-1}\}}$.
	\end{claim}
	Fix a tuple $\bar{v}$.
	To see the first part of Claim~\ref{claim:c1} (about lexicographic ordering) notice that for every $v\in B$ it holds that $\str{D}^v=\str{B}\restriction_{\{v\}}$ and this is the first
	substructure (in the order $\preceq$) containing $v$. Therefore $\varphi(v)_v$ is the fist letter of $\varphi(v)$ that is not $n$.
	Now let $i$ be such that $\str{B}^i=\str{B}\restriction_{\{v_0,v_1,\dots, v_{k-1}\}}$. By the construction, $\bar{v}$ is compatible and agrees on a substructure on level $i$. This finishes the proof of the claim.

	\medskip

	Now we equip every level $\ell$ of the tree $(\Sigma^{{<}\omega},\sqsubseteq)$ with an ordered $L$-structure $\str{C}_\ell$:
	\begin{enumerate}
		\item The vertex set of  $\str{C}_\ell$ is $C_\ell=\Sigma^\ell$,
		\item $\leq_{\str{C}_\ell}$ is the lexicographic ordering of $\Sigma^\ell$,
		\item Whenever $(w^0,w^1,\ldots,w^{k-1})$ is a tuple of vertices of $\str{C}_\ell$ which is compatible and decides structure on some level $i$ then $B\restriction_{\{w^0,w^1,\ldots,w^{k-1}\}}$ is isomorphic to $B\restriction_{\{w^0_i,w^1_i,\ldots,w^{k-1}_i\}}$.
		\item There are no other tuples in relations of $\str{C}_\ell$ then ones required by the above two conditions.
	\end{enumerate}
	Notice that the definition of compatibility and Claim~\ref{claim:c1} ensures that such a structure exists.
	Moreover, the following claim follows directly from the construction:
	\begin{claim}
		\label{claim:embedding}
		Function $\varphi$ is an embedding $\str{B}\to \str{C}_{2^{|B|}-1}$.
	\end{claim}
	Next, we define a sequence $\mathcal E=(\mathcal E_n)_{n\in\omega}$ of sets where for every $n$ the set $\mathcal E_n$ consists precisely of those functions $e\colon \Sigma^n\to \Sigma$ such that
	for every strictly lexicographically increasing sequence $\bar{w}=(w^0,w^1,\ldots,w^{k-1})$ of elements of $\Sigma^{n}$ the sequence $\bar{w}'=((w^0)\conc e(w^0),(w^1)\conc e(w^1),\ldots,(w^k)\conc e(w^k))$ satisfies:
	\begin{enumerate}[label=(\Roman*)]
		\item\label{item:ah1} $\bar{w}'$ decides a structure on some level $i$ if and only if $\bar{w}$ decides a substructure on same level,
		\item\label{item:ah2} $\bar{w}'$ is compatible if and only if $\bar{w}$ is compatible.
	\end{enumerate}
	\begin{claim}
		\label{claim:boring}
		$\mathcal E$ is a family of boring extensions.
	\end{claim}
	Property \ref{item:boring:duplication} of Definition~\ref{def:boringext} follows from the fact that duplicating an earlier level will never affect properties \ref{item:ah1} and \ref{item:ah2} since they are concerned about first occurrences of certain combinations.

	Let $m\leq n$, $e_1$ and $e_2$ be as in property \ref{item:boring:duplication} of Definition~\ref{def:boringext}. We define $e_3\colon \Sigma^{n+1}\to \Sigma$ by putting
	$$
		e_3(w)=\begin{cases}
			e_2((w|_m)\conc b) & \hbox{if $w=(w|_m)\conc e_1(w|_m)\conc b$ where $b$ is a suffix of $w$,} \\
			n                  & \hbox{otherwise.}
		\end{cases}
	$$
	It is easy to verify that $e_3\in \mathcal E_{n+1}$:
	observe that if a tuple $\bar{w}$ agrees on a substructure or becomes incompatible on level $\ell+1$
	then for every $w\in \bar{w}$ it holds that $w_\ell\neq n$.  For this reason we can use $n$ as a safe choice when a value of $e_3$ is not implied by $e_2$.

	\medskip

	Next we show:
	\begin{claim}
		\label{claim:types}
		For every $m,n\in \omega$ and sets $X\subseteq \Sigma^m$ and $Y\subseteq \Sigma^n$ satisfying $\tau_\mathcal E(X)=\tau_\mathcal E(Y)$ it holds that $\str{C}_m\restriction_X$ is isomorphic to $\str{C}_n\restriction_Y$
	\end{claim}
	The construction of family $\mathcal E$ is chosen in a way so that $\tau_\mathcal E(X)$ removes precisely those levels that does not affect the construction of structure $\str{C}_m\restriction_X$. The same is true about $\str{C}_n\restriction_Y$. Consequently, sets of the same embedding type induce isomorphic substructures.

	\begin{claim}
		\label{clm:envelope}
		Let $\str{A}$ be a substructure of $\str{B}$.
		Then for every pair of embeddings $f_1,f_2\colon\str{A}\to \str{B}$ it holds that $\tau_\mathcal E(\varphi\circ f_1[A])=\tau_\mathcal E(\varphi\circ f_2[A])$. Moreover $\tau_\mathcal E(\varphi\circ f_1[A])$ consists of words of length $2^{|A|-1}$.
	\end{claim}
	To see this, claim fix $\str{A}$, $i$, and $f_1,f_2$ as in the statement. Put $S=\varphi\circ f_1[A]$.  We show that $\ell< 2^{|B|}-1$ interesting if and only if
	there exists a lexicographically increasing tuple $\bar{w}\subseteq S$ that agrees on a substructure on level $\ell$.  Clearly if such a tuple exists, the level is interesting. To see the contrary,
	assume there is no such tuple.
	In this case we can define $e\colon \Sigma^\ell\to \Sigma$ by putting $e(u)=n$ if $u$ is not an initial segment of some word in $S$ and $e(u)=w_{\ell+1}$
	if $u$ is an initial segment of some $w\in S$. Notice that this is well defined, since if we have $w,w'\in S$ such that $w\meet w'=u$ and $w$ lexicographically before $w'$ then, by the construction of $\varphi$,
	$w'_{\ell+1}=n$ and $w_{\ell+1}$ is the first element of $w$ different from $n$. In this case level $\ell$ is interesting because the 1-tuple $(w)$ agrees on a substructure.

	Similarly as in the verification of Claim~\ref{claim:boring} we can verify that $e\in \mathcal E_\ell$:
	We know that no lexicographically strictly increasing tuple of vertices of $S$ decides substructure on level $\ell$.  By Claim~\ref{claim:c1} all such tuples of vertices in $S$ are compatible.
	Any tuple containing a word $w\in \Sigma^{\ell+1}\setminus S$ has the property that $w_\ell=n$ and thus it neither decides a substructure nor does it become incompatible.

	The interesting levels of $S$ are thus those levels where some strictly lexicographically increasing tuple of vertices in $S$ decides a structure.  It follows from the construction of $\varphi$, the choice of the ordering $\preceq$, and the encoding using the lexicographically first isomorphic substructures, that for both $S$ and $\varphi\circ f_2[A]$, substructures to agree on appear in the same order and are always represented by their lexicographically first isomorphic copies.

	\medskip

	Let $N$ be given by Theorem~\ref{thm:boring2} for $n=0, k=2^{|A|}-1, m=2^{|B|}-1, r=2$, $X=\tau_\mathcal E(\varphi\circ f[A])$ (for some embedding $f\colon\str{A}\to\str{B}$) and $Y=\varphi[B]$.
	By Claims~\ref{claim:embedding} and \ref{claim:types} we know that every subset of $\Sigma^N$ of the same embedding type as $Y$ corresponds to a copy of $\str{B}$ in $\str{C}_N$.
	By Claim~\ref{clm:envelope}, every subset of $\Sigma^N$ of embedding type $X$ corresponds to a copy of $\str{A}$ in $\str{C}_N$.
	Let $\chi$ be a 2-colouring of embeddings $\str{A}\to\str{C}_{N}$. By this correspondence, we obtain a colouring $\chi'$ of $\ebinom{\Sigma^N}{X}$.
	Let $Y'\subseteq \Sigma^N$ be a monochromatic subset given by Theorem~\ref{thm:boring2}, and let $g\colon \str{B}\to\str{C}_N$ be the corresponding embedding.  By Claim~\ref{claim:types},
	for every $g'\in \Emb(\str{A},\str{B})$ we know that $\tau_\mathcal E(g\circ g'[A])=X$. It follows that $g$ is the desired embedding, and hence $\str{C}_N\longrightarrow(\str{B})^\str{A}_2$
\end{proof}

\subsection{Big Ramsey degrees of free amalgamation classes}
We consider same model-theoretic $L$-structures as in Section~\ref{sec:AH}. However instead of ordered structures we will work with enumerated structures to signify the fact that the infinite orders of vertices considered will always be of the order-type $\omega$.
We say that an $L$-structure $\str{A}$ is \emph{enumerated} if its vertex-set is the cardinal $|\str{A}|$.
An embedding $f\colon \str{A}\to\str{B}$ of two enumerated structures is \emph{ordered} if for every $i<j\in A$ we also have $f(i)<f(j)$.
Recall that we denote by $\Emb(\str{A},\str{B})$ the set of all embeddings from $\str{A}$ to $\str{B}$. Denote by $\OEmb(\str{A},\str{B})$ the set of all ordered embeddings from $\str{A}$ to $\str{B}$.

We call an $L$-structure $\str{A}$ \emph{irreducible} if for every $u,v\in A$ there exists a relational symbol $\rel{}{}\in L$ and a tuple $\bar{x}\in \rel{A}{}$ containing both $u$ and $v$. Given a family $\mathcal F$ of finite enumerated irreducible structures, we denote by $\KF$ the class of all finite enumerated structures
$\str{A}$ such that there is no $\str{F}\in \mathcal F$ with an ordered embedding $\str{F}\to\str{A}$.
Given an enumerated $L$-structure $\str{K}$, we denote by $\OAge(\str{K})$ the set of all finite enumerated $L$-structures $\str{A}$ such that there exists an ordered embedding $\str{A}\to\str{K}$.
A countable enumerated $L$-structure $\str{K}$ is \emph{$\KF$-universal} if $\OAge(\str{K})=\KF$ and every countable enumerated $L$-structure $\str{K}'$ with $\OAge(\str{K}')\subseteq \KF$ has an ordered embedding to $\str{K}$.

\begin{definition}
	We say that an $L$-structure $\str{K}$ has \emph{finite big Ramsey degrees} if for every finite substructure $\str{A}$ of $\str{K}$ there exists $n\in \omega$ such that
	for every finite colouring $\chi$ of $\Emb(\str{A},\str{K})$ there exists $F\in \Emb(\str{K},\str{K})$ such that $\chi$ restricted to ${\{F\circ f:f\in \Emb(\str{A},\str{K})\}}$ takes
	at most $n$ colours.
\end{definition}

We prove the following mild strengthening of a theorem by Zucker~\cite{zucker2020}, where we are allowed to forbid specific enumerations of irreducible $L$-structures.
\begin{theorem}
	\label{thm:zucker}
	Let $L$ be a finite relational language with binary and unary symbols only, and let
	$\mathcal F$ be a finite family of finite enumerated irreducible $L$-structures. Then every $\KF$-universal $L$-structure has finite big Ramsey degrees.
\end{theorem}

The main purpose of this section is to demonstrate an application of Theorem~\ref{thm:shaperamseyN} for trees with unbounded branching and with a complicated structure, and also to give the first direct proof of the result
above, i.e.,\ with no use of set-theoretic forcing.

We now fix a finite language $L$ with binary and unary symbols only, a finite family $\mathcal F$ of finite enumerated irreducible $L$-structures
and some symbol $t\notin \omega$ to serve as a special vertex called \emph{type vertex}.
We assume without loss of generality that $\KF$ contains an $L$-structure with a single vertex and all relations empty,
and also at least one irreducible binary structure.

Let $\str{A}$ be an $L$-structure.
Given $v\in A$ and some $u\notin A$, we denote by $\str{A}(v\mapsto u)$ the structure created from $\str{A}$ by renaming vertex $v$ to $u$.
Given $L$-structure $\str{B}$ we write $\str{A}\subseteq \str{B}$ if $A\subseteq B$ and $\str{A} = \str{B}\restriction_A$.

Bounds on big Ramsey degrees are based on studying Ramsey properties of trees of types which can be constructed as follows.
Given an $L$-structure $\str{A}$, a vertex $a\in A$ and a set $S\subseteq (A\setminus{a})$, the \emph{type of $a$ in $\str{A}$ over socle (or parameter) $S$} is the structure $(\str{A}\restriction_{(S\cup\{a\})})(a\mapsto \{t\})$.  If $\str{A}$ is enumerated, it is natural to consider socles that are initial segments in the enumeration: Given $u\leq v\in A$, we put $$\typev^u_\str{A}(v)=(\str{A}\restriction_{(u\cup\{v\})})(v\mapsto t).$$

\begin{figure}
	\centering
	\includegraphics[scale=0.833]{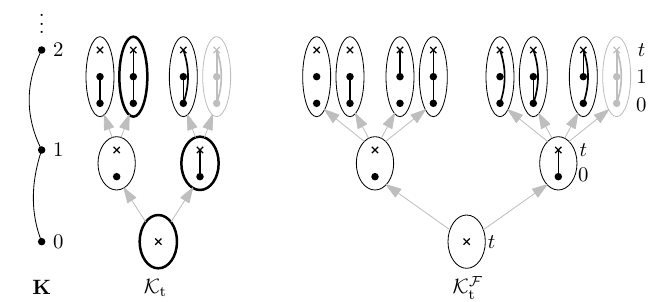}
	\caption{First three levels of trees  $(\mathcal{K}_{\mathrm t},\subseteq)$ (left, with coding nodes in bold) and $(\KFt,\subseteq)$ (left) with $\mathcal F$ consisting of non-bidirectional edges and a triangle.
		Enumerated vertices are depicted by dots while type vertices by crosses}
	\label{fig:typetree}
\end{figure}
Given an enumerated structure $\str{K}$, put $\mathcal{K}_{\mathrm t}=\{\typev^u_\str{K}(v):u\leq v\in \str{K}\}$. It is easy to see that $(\mathcal{K}_{\mathrm t},\subseteq)$ is a finitely branching tree which we call the \emph{tree of types of~$\str{K}$}, and it is isomorphic to the tree $\mathrm{CT}^\str{K}$ considered in \cite{zucker2020} (see Figure~\ref{fig:typetree} left).
Each level $\ell$ of this tree  contains a special node $\typev^\ell_\str{K}(\ell)$ called a \emph{coding node}. Dobrinen~\cite{dobrinen2017universal,dobrinen2019ramsey} and Zucker~\cite{zucker2020} developed a Ramsey theorem for trees with coding nodes which is the main tool used in~\cite{zucker2020} (see also recent ICM survey~\cite{dobrinen2021ramsey}).  These trees are highly irregular and do not directly fit the framework of $(\S,\M)$-trees (even though a suitable $(\S,\M)$-tree can be constructed).  We present an alternative approach based on a new, more regular tree, the \emph{tree of types of all enumerations} of $L$-structures with ordered age $\KF$  (see Figure~\ref{fig:typetree} right).

We put $$\KFt=\{\str{A}(\max A\mapsto t):\str{A}\in \KF\hbox{ and }|A|>0\}$$ to be the class of all \emph{types} and consider the tree $(\KFt,\subseteq)$.
This step avoids explicit coding nodes and the does not depend on a specific choice of a $\KF$-universal structure $\str{K}$.

Many proofs of statements about big Ramsey degrees implicitly use regular trees of types of all enumerations (see~\cite{devlin1979,Laflamme2006,dobrinen2016rainbow,Hubicka2020CS,BalkoForbiddenCycles}), usually
in a combination with Milliken's tree theorem, a technique introduced by Laver~\cite{todorcevic2010introduction}. More recently, the Carlson--Simpson theorem has also been used as an underlying Ramsey theorem for trees~\cite{Hubicka2020CS}.
For the first time, we make the enumeration an explicit part of the type which makes it possible to work with
bigger forbidden substructures than those consisting of at most 3 vertices, but it yields trees with unbounded branching even for structures in finite binary languages.

\begin{remark}
	Trees with unbounded branching are also used to give an upper bound on big Ramsey degrees of $L$-structures in languages containing relations of arity larger than two~\cite{Hubicka2020uniform,braunfeld2023big}.
	The notion of shape-preserving functions originates from a precise characterisation of big Ramsey degrees of the universal partial order~\cite{Balko2023}.
	Techniques presented here generalize to both higher arity and to non-free amalgamation classes (see~\cite{typeamalg}), this is however a more complicated result and it is out of the scope of this paper.
\end{remark}

\subsubsection{Tree of partial types}

To define a successor operation satisfying axioms \ref{S1} and \ref{S2} we need to further refine the tree $(\KFt,\subseteq)$ as follows.
Let $L^+$ be a relational language extending $L$ by a single binary relation $E\notin L$. We will denote $L^+$-structures by bold letters with superscript $+$ (such as $\pstr{A},\pstr{B},\ldots$), and the corresponding bold letters with no superscript will denote their \emph{$L$-reducts}, that is, $L$-structures created by forgetting the relation $E$.
\begin{definition}[Partial structure]
	A \emph{partial structure} is an enumerated $L^+$-structure $\pstr{A}$ such that
	\begin{enumerate}
		\item for every $(u,v)\in E_\pstr{A}$ it holds that $u<v-1$,
		\item if $(u,v)\in  E_\pstr{A}$ then then also $(u',v)\in E_\pstr{A}$ for every $u'<u$,
		\item for every $u<v$ such that the $L$-reduct $\str{A}\restriction_{\{u,v\}}$ is irreducible we have $(u,v)\in E_\pstr{A}$.
	\end{enumerate}
	For a partial structure $\pstr{A}$ we define the \emph{free level} function $\freef{A}$ by putting $$\freef{A}(v)=\min\{u\in A:(u,v)\notin E_\pstr{A}\}.$$
\end{definition}
\begin{remark}
	A similar, less restrictive notion of partial structure is used in constructions of Ramsey and EPPA classes~\cite{Hubicka2016,Hubicka2018EPPA}
	to amalgamate structures in a complicated way which can not be done by iterating normal amalgamation.

	Partial structures are more important when working with non-free amalgamation classes. In our proof, the main role of the relation $E$ is to obtain
	\ref{S2}.
\end{remark}
Given a partial structure $\pstr{A}$ and a vertex $v\in A$ we put $$\typev_\pstr{A}(v)=\typev_\pstr{A}^{\freef{A}(v)}(v)$$
and call it a \emph{partial type}.
We put
$$\KFpt=\{\typev_\pstr{A}(v):\pstr{A}\hbox{ is a partial structure, }\str{A}\in \KF\hbox { and } v\in A\}.$$
Notice that for every partial type $\type{T}\in \KFpt$ and every $v\in (T\setminus\{t\})$ it holds that $(v,t)\in E_\type{T}$.
Moreover, $\type{T}\setminus \{t\}$ is a partial structure and thus we can define the free level function $\freef{T}$ on it.

\begin{figure}
	\centering
	\includegraphics[scale=0.833]{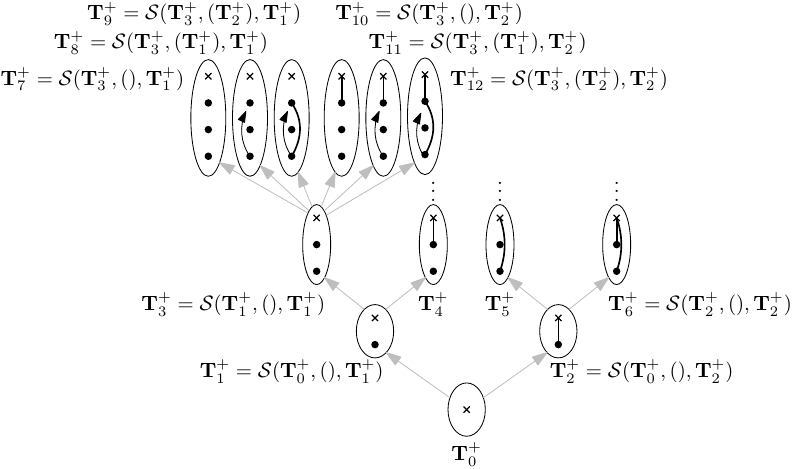}
	\caption{Initial part of tree  $(\KFpt,\subseteq)$. Pairs of vertices in relation $E$ are depicted by black arrows. Grey arrows denote immediate successors in $(\KFpt,\subseteq)$.}
	\label{fig:partialtypetree}
\end{figure}
The tree $(\KFpt,\subseteq)$ (see Figure~\ref{fig:partialtypetree}) is called the \emph{tree of partial types} of members of $\KF$.
Notice that the level of $\type{T}$ is $T\setminus \{t\} < \omega$.
Put $\Sigma=\left\{\type{A}\in \KFpt:A=\left\{0,t\right\}\right\}$.
\begin{figure}
	\centering
	\includegraphics[scale=0.833]{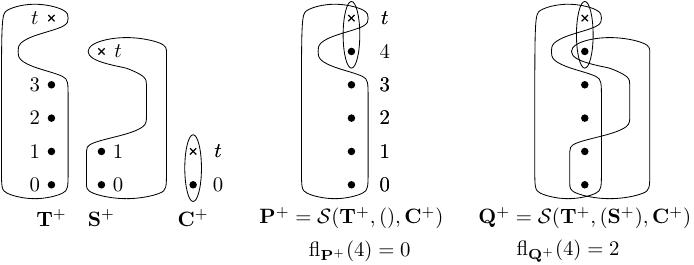}
	\caption{Successor operation on partial types. Enumerated vertices are depicted by dots and type vertices by crosses.}
	\label{fig:typesuccessor}
\end{figure}
\begin{definition}[Successor operation on partial types]
	\label{defn:levelledS}
	Given a partial type $\type{T}$ of level $\ell$, a tuple $\bar{p}$ such that:
	\begin{enumerate}
		\item$\bar{p}=()$ is empty, or
		\item $\bar{p}=(\type{S})$ for some partial type $\type{S}$ of level $\ell'<\ell$ such that $\type{S}\setminus \{t\}\subseteq \type{T}$,
	\end{enumerate}
	and a partial type $\type{C}\in \Sigma$, we define (if possible) a partial type $\type{Q}$ as follows:
	\begin{enumerate}
		\item $Q=(\ell+1)\cup\{t\}=T\cup \{\ell\}$,
		\item both $\type{T}$ and $\type{C}(0\mapsto \ell)$ are substructures of $\type{Q}$,
		\item if $\bar{p}=(\type{S})$ then $\type{S}(t\mapsto \ell)$ is a substructure of $\type{Q}$, and,
		\item there are no tuples in any relations of $\type{Q}$ other than the ones required by the conditions above.
	\end{enumerate}

	If there is $\type{Q} \in \KFpt$ as above, we put $\S(\type{T},\bar{p},\type{C})=\type{Q}$. We leave $\S(\type{T},\bar{p},\type{C})$ undefined in all other cases.
	See Figure~\ref{fig:typesuccessor}.
\end{definition}
\begin{remark}
	$\S(\type{T},(\type{S}),\type{C})$ is an amalgamation of $\type{T}$ and $\type{S}$ over $\type{S}\restriction_\omega$.
\end{remark}
\begin{prop}
	$(\KFpt,{\subseteq}, \Sigma, \S)$ is an $\S$-tree.
\end{prop}
\begin{proof}
	Because the language $L$ is finite, and thus for every $\ell$ the set $\{\type{T}\in \KFpt:|T|=\ell+1\}$ is finite, we know that $(\KFpt,{\subseteq})$ is finitely branching and has finitely many nodes of level 0.
	Condition \ref{S1} of Definition~\ref{def:sucessor} follows directly from Definition~\ref{defn:levelledS}.
	It is easy to see that whenever $\type{Q}=\S(\type{T},\bar{p},\type{C})$ is defined, then
	$\type{Q}\in \KFpt$, and since $\type{Q}$ extends $\type{T}$ by a single vertex, it is an immediate successor of $\type{T}$.
	If $\bar{p}$ is non-empty we also know that $\bar{p}=(\type{S})$ for some $\type{S}\in \KFpt$ satisfying $\ell(\type{S})< \ell(\type{T})$.

	To see \ref{S3} one can verify that every $\type{Q}\in \KFpt$ of level $\ell>0$ can be decomposed into $\type{T}$, $\bar{p}$ and $\type{C}$
	such that $\type{Q}=\S(\type{T},\bar{p},\type{C})$:
	\begin{enumerate}
		\item $\type{T}$ is the immediate predecessor of $\type{Q}$:
		      $$\type{T}=\type{Q}|_{\ell-1}=\type{Q}\restriction_{(\ell-1\cup \{t\})}.$$
		\item If $\freef{Q}(\ell-1)=0$ then $\bar{p}$ is the empty sequence.  If $\freef{Q}(\ell-1)=\ell'>0$ then $$\bar{p}=\left(\typev_\type{Q}(\ell-1)\right)=\left(\type{Q}\restriction_{\left(\ell'\cup \{\ell-1\}\right)}\left(\ell-1\mapsto t\right)\right),$$
		\item $\type{C}=(\type{Q}\restriction_{\{\ell-1,t\}})(\ell-1\mapsto 0)$.
	\end{enumerate}
	Finally, \ref{S2} follows from the fact that the decomposition above is unique. Here, the relation $E$ plays an important role, since the free-level function determines the level of the parameter used to build the given element of $\KFpt$.
\end{proof}
\subsubsection{Properties of shape-preserving functions on partial types}
We define $\M$ to be the set of all shape-preserving functions $F\colon \KFpt\to\KFpt$ of $(\KFpt,{\subseteq},\allowbreak \Sigma, \S)$
satisfying $\tilde{F}(0)=0$.
To establish that $(\KFpt,\allowbreak {\subseteq},\allowbreak  \Sigma,\allowbreak  \S,\allowbreak \M)$ is an $(\S,\M)$-tree, we first prove three lemmas about properties of shape-preserving functions.

The condition $\tilde{F}(0)=0$ is not necessary to establish an $(\S,\M)$-tree. However it is needed to obtain:
\begin{observation}
	\label{obs:has_e1}
	$(\KFpt,\subseteq, \Sigma, \S,\M)$ satisfies assumption~\ref{item:E1} of Proposition~\ref{prop:envelopes}.
\end{observation}
This is because Definition~\ref{defn:levelledS} treats specially partial types of level 0 which can not appear as parameters.
We will eventually apply Proposition~\ref{prop:envelopes} to determine an upper bound on the big Ramsey degrees. Towards that we however need
to first understand what levels can be skipped by functions in $\M$.

\iftrue
	Given $\type{T}\in \KFpt$ and vertex $v\in T\setminus \{t\}$  we denote by $\type{T}\ominus v$  the $L^+$-structure created from $\type{T}$ by removing vertex $v$ and renaming all vertices $i\in T\setminus\{t\}$ with $i>v$ as $i\mapsto i-1$.
	\begin{lemma}[Level removal lemma]
		\label{lem:vertexremoval}
		For every $n>0$ and $F\in \M$ satisfying $\tilde{F}(n)>\tilde{F}(n-1)+1$ it holds that
		the function $F'\colon \KFpt\to \KFpt$ defined by $$F'(\type{T})=\begin{cases}
				F(\type{T})                         & \hbox{if }\ell(\type{T})<n,    \\
				F(\type{T})\ominus (\tilde{F}(n)-1) & \hbox{if }\ell(\type{T})\geq n
			\end{cases}$$
		is shape-preserving.
	\end{lemma}
	\begin{proof}
		Let $F$ and $n$ be as in the statement of the lemma.
		Put $m=\tilde{F}(n)-1$.
		Fix $\type{T}$ with $\ell(\type{T})>n$ and put
		$\type{Q}=F(\type{T})\ominus (\tilde{F}(n)-1)$. It is clear that removing a vertex from the $L$-reduct $\str{T}$ did not create a forbidden substructure. Renumbering vertices then ensures that $\str{Q}\in \KFt$.

		To see that $\type{Q}\in \KFpt$ we also need to verify that whenever $(u,v)\in E_\str{Q}$ and $v\neq t$ then $u<v-1$.
		Put $\type{S}=F(\type{T})$.  We removed vertex $m$ from $\type{S}$, hence we need to verify that $\freef{S}(m+1)<m$. This would then imply $\freef{Q}(m)<m$.
		Since $F$ is shape-preserving, we know that $\type{S}|_{m+2}$ is built using a parameter that is in the image of $F$, and all such parameters
		are strictly bellow $m$, hence indeed $\freef{S}(m+1)<m$.

		Function $F'$ is clearly level-preserving. It is easy to see that whenever $\S(\type{T},\bar{p},\type{C})$ is defined (and thus also $\S(F(\type{T}),(F(\bar{p})),\type{C})$ is defined) then $\S(F'(\type{T}),\allowbreak F'(\bar{p}),\type{C})$
		is defined and equal to $\S(F(\type{T}),(F(\bar{p})),\type{C})\ominus (\tilde{F}(n)-1)$, which proves that $F'\in \M$.
	\end{proof}
\fi

\begin{definition}[Type compatibility]
	We call types $\type{T},\type{S}\in \KFpt$ of level $\ell$ \emph{compatible}, and write $\type{T}\comp \type{S}$, if and only if $\type{T}\restriction_{\ell}=\type{S}\restriction_{\ell}$.
\end{definition}

\begin{lemma}[Boring extension lemma]
	\label{lem:boring}
	Let $\ell> 0$ be a level, $S\subseteq \KFpt(\ell)$ a set of types of level $\ell$, and $f\colon S\to \KFpt(\ell+1)$ a function satisfying:
	\begin{enumerate}[label=(B\arabic*)]
		\item\label{item:boringB1} $\type{T}\subseteq f(\type{T})$ for every $\type{T}\in S$.
		\item\label{item:boringB2} \emph{No compatibility change:} If $\type{T}\comp\type{S}$ then also $f(\type{S})\comp f(\type{T})$.
		\item\label{item:boringage} \emph{No age change:} For every partial structure $\pstr{A}$
		such that for every $i\geq \ell+1$ we have $\typev^\ell_\pstr{A}(i)\in f[S]$ it holds
		that whenever there is $\str{F}\in \mathcal F$ and an ordered embedding $e\colon \str{F}\to \str{A}$ with $\ell\in e[F]$ there is also $\str{F}'\in \mathcal F$ and an
		ordered embedding $e'\colon \str{F}'\to \str{A}$ with $\ell\notin e[F']$.
	\end{enumerate}
	Then there exists $F\in \M$ skipping only level $\ell$ extending f.
\end{lemma}
\begin{remark}
	Condition~\ref{item:boringage} corresponds to the absence of age changes in Zucker's paper~\cite[Definition 5.1]{zucker2020}.
	Age changes are a key property of trees of types of structures with forbidden substructures.
\end{remark}
\begin{proof}
	Fix $\ell$, $S$ and $f$ as in the statement. We construct $F$ as follows.

	First for every $\type{T}\in \KFpt({<}\ell)$ we we put $F(\type{T})=\type{T}$.

	Now for given $\type{T}\in \KFpt$ of level $m\geq \ell$ we put $F(\type{T})=\type{Q}$ defined in several steps:

	\begin{enumerate}[label=\arabic*]
		\item\label{item:boringstep1} {\bf Inserting new isolated vertex $\ell$:}
		Start constructing $\type{Q}$ from $\type{T}$ by renaming every vertex $i\geq \ell$ to $i\mapsto i+1$ and inserting a new vertex $\ell$.
		Structure $\type{Q}\setminus \{\ell\}$ is already fully specified, it remains to specify tuples in relations of $\type{Q}$ containing the new vertex $\ell$.

		\item\label{item:boringstep2} {\bf Defining the type of vertex $\ell$:}
		If there is $\type{S}\in S$ such that $\type{S}\comp \type{T}|_\ell$,
		notice that $f(\type{S})\restriction_{\ell}=\type{S}\restriction_{\ell}=\type{Q}\restriction_{\ell}$, and
		add precisely those tuples (containing vertex $\ell$) to relations of $\type{Q}$ so that $$\type{Q}\restriction_{\ell+1}=f(\type{S})\restriction_{\ell+1}.$$
		Observe that by \ref{item:boringB2} the outcome of this step does not depend on the choice of $\type{S}$ as all possible choices are mutually compatible.

		\item\label{item:boringstep4} {\bf Completing the type of $t$:}
		Add tuple $(\ell,t)$ to $E_\type{Q}$.
		If $\type{T}|_\ell\in S$ then add tuples to relations of $\str{Q}$ (containing both $\ell$ and $\{t\}$) so that $$\type{Q}\restriction_{((\ell+1)\cup\{t\})}=f(\type{T}|_\ell).$$ If $\type{T}|_\ell\notin S$, we do not add any such relations.
		This completes the construction of $\type{Q}\restriction_{((\ell+1)\cup\{t\})}$, and so $F(\type{S})$ is fully defined for every $\type{S}\in \KFpt({\leq} \ell)$. This will be useful in the next step.

		\item\label{item:boringstep3} {\bf Completing types of vertices above $\ell$:}
		Now consider every $i\geq \ell\in T$ satisfying $\freef{T}(i)\geq \ell$ and add $(\ell, i+1)$ to $E_\type{Q}$.
		Put $\type{S}=\typev_\type{T}(i)|_\ell$ and
		add precisely those tuples (containing $\ell$ and $i+1$) to relations of $\type{Q}$ so that it holds that

		$$\typev_\type{Q}^\ell(i+1)=F(\type{S}).$$
		Note that if $\type{S}\in S$, after step~\ref{item:boringstep2} we have by~\ref{item:boringB2} that $f(\type{S})\restriction_{(\ell+1)}=F(\type{S})\restriction_{(\ell+1)}$.
		Notice that we now have $\typev_\type{Q}(i+1)=F(\type{S}).$

	\end{enumerate}

	This finishes the construction of $\type{Q}$.
	Notice that in steps \ref{item:boringstep4} and \ref{item:boringstep3} we extended the relation $E_\type{Q}$ so that $\type{Q}\setminus \{t\}$ is a partial structure.
	In step \ref{item:boringstep2} we extended $E_\type{Q}$ so that $\type{Q}$ is a partial type.

	Observe that for every $\type{T}\in S$ we have $F(\type{T})=f(\type{T})$. $F(\type{T})$ extends $\type{T}$ by a single new vertex $\ell$ which is added in step~\ref{item:boringstep1}. Steps~\ref{item:boringstep2} and \ref{item:boringstep4} reconstruct all tuples from $f(\type{T})$ containing $\ell$.

	By \ref{item:boringage} we have $\type{Q}\in \KFpt$: Indeed, it suffices to prove that there is no $\str F\in \mathcal F$ with an embedding $g\colon \str F\to \str{Q}$. Suppose for a contradiction that there is such a $g$. As $\type{Q} \ominus m = \type{T} \in \KFpt$, we know that $m$ is in the range of $g$. Let $\str A^+$ be the substructure induced by $\type{Q}$ on the range of $g$ together with $m\cup \{t\}$ after the natural renumbering. Note that as $\str F$ is irreducible, all vertices of $\str A^+$ above $m$ are connected to $m$ by a relation from $L$. From the construction it follows that for each such vertex $x$ we have that $\typev_\pstr{A}(x)\in f[S]$. \ref{item:boringage} then gives the desired contradiction.

	Similarly, for every $\type{T}\in \KFpt(\ell)$ we have $F(\type{T})\supseteq
		\type{T}$ as desired. To see that $F\in \M$ it remains to verify that
	for every partial type $\type{T}$ of level
	$\ell\geq m$ and for every $\bar{p}$ and $\type{C}$ such that
	$\S(\type{T},\bar{p},\type{C})$ is defined, we have
	$\S(F(\type{T}),F(\bar{p}),\type{C})=F(\S(\type{T},\bar{p},\type{C}))$.

	If $\bar{p} = ()$, observe that $F(\S(\type{T},\type{C}))$ is obtained by inserting vertex $\ell$ to $\S(\type{T},\type{C})$ using the procedure above while $\S(F(\type{T}),\type{C})$ is obtained by first inserting vertex $\ell$ to $\type{T}$ and then building a successor which gives the same outcome.
	Similarly if $\bar{p}=(\type{S})$.
\end{proof}
\iftrue
	The following lemma shows that the assumptions of Lemma~\ref{lem:boring} can not be relaxed.
	\begin{lemma}
		\label{lem:comp}
		For every $F\in \M$ it holds that:
		\begin{enumerate}
			\item $\type{T}$ can be created from $\type{Q} = F(\type{T})$ by removing all vertices in $Q\cap (\omega\setminus \tilde{F}[\omega])$ and renumbering the remaining vertices accordingly,
			\item whenever $\type{T}\comp \type{S}$ then also $F(\type{T})\comp F(\type{S})$,
			\item for every $n\in \omega$ and every partial structure $\pstr{A}$ such that for every $i\geq \tilde{F}(n)$ we have $\typev_\pstr{A}(i)\in F[\KFpt(n)]$, it holds
			      that whenever there are $\str{F}\in \mathcal F$ and an ordered embedding $e\colon \str{F}\to \str{A}$ with $n\in e[F]$ then there are also
			      $\str{F}'\in \mathcal F$ and an ordered embedding $e'$ from $\str{F}$ to a partial structure created from $\str{A}$ by removing all vertices $\{v<\tilde{F}(n):v\notin \tilde{F}[\omega]\}$.
		\end{enumerate}
	\end{lemma}
	\begin{proof}
		The first statement follows directly from the construction of the successor operation and definition of shape-preserving functions.

		To see the second statement,
		notice that for every pair $\type{T},\type{S}\in \KFpt$ of the same level we have that
		$\type{T}\comp \type{S}$ if and only if there exists some $\type{C},\type{C}'\in \Sigma$ such that $\S(\S(\type{T},\type{C}),(\type{S}),\type{C}')$ is defined; for example, both $\str C$ and $\str C'$ may be chosen to have no binary relations and with unary relations determined by $\type{S}$ and $\type{T}$. Because $F$ is shape-preserving we also have $\S(\S(F(\type{T}),\allowbreak \type{C}),\allowbreak (F(\type{S})),\allowbreak \type{C}')$ defined and this is possible if and only if $F(\type{T})\comp F(\type{S})$.

		To see the last statement, chose $n\in\omega$ and a partial structure $\pstr{A}$ such that for every $i\geq \tilde{F}(n)$ we have $\typev_\pstr{A}(i)\in F[\KFpt(n)]$.
		Now create $\pstr{B}$ by removing all vertices $\{v<\tilde{F}(n):v\notin \tilde{F}[\omega]\}$. Observe that this is a partial structure by the same argument as in
		Lemma~\ref{lem:vertexremoval}.
		Let $\type{T}$ be a type extending $\pstr{B}$ by an isolated vertex $t$. Assume, towards the contrary, that no $\str F'\in \mathcal F$ embeds into $\pstr{B}$, which implies that $\type{T}\in \KFpt$. Then $F(\type{T})$ is defined, but by the construction it contains an ordered copy of $\pstr{A}$, a contradiction.
	\end{proof}
\fi
Now we are ready to apply Theorem~\ref{thm:shaperamseyN}.
\begin{prop}
	\label{prop:typeramsey}
	$(\KFpt,{\subseteq},\Sigma,\S,\M)$ is an $(\S,\M)$-tree.
	Consequently, for every pair $n,k\in \omega$ and every finite colouring
	$\chi$ of $\AM^n_k$, there exists $F\in \M^n$ such that $\chi$ is constant when restricted to
	$\{F\circ g: g\in \AM^n_k\}$.
\end{prop}
\begin{proof}
	By Proposition~\ref{prop:shape-pres} we know that $\M$ satisfies condition \ref{item:Nmonoid} of Definition~\ref{def:sntree}.

	We first verify \ref{item:Nduplication}.
	Given $n<m$, define a function $f\colon \KFpt(m)\to \KFpt(m+1)$ by putting  $f(\type{T})=\S(\type{T},\bar{p},\type{C})$ where $\bar{p}$ and $\type{C}$ satisfy $\type{T}|_{n+1}=\S(\type{T}|_n,\bar{p},\type{C})$.
	Observe that $\S(\type{T},\bar{p},\type{C})$ extends $\type{T}$ by vertex $m$ which is a duplicate of vertex $n$, and is thus always defined (this follows from the fact that $\mathcal F$ consists of irreducible structures).
	An application of Lemma~\ref{lem:boring} yields the desired function $F_m^n$.

	Finally, we verify \ref{item:Ndecomposition}.
	Let $F$ and $n$ be as in~\ref{item:Ndecomposition}.
	Obtain $F_1$ by an application of Lemma~\ref{lem:vertexremoval}.
	Let $f\colon F_1[\KFpt(n)]\to \KFpt(\tilde{F}(n))$ be defined by putting $f(\type{T})=F(F^{-1}_1(\type{T}))$.
	It is easy to see that $f$ satisfies condition~\ref{item:boringB1} of Lemma~\ref{lem:boring}. Since conditions~\ref{item:boringB2} and \ref{item:boringage} follow by Lemma~\ref{lem:comp}, we can apply Lemma~\ref{lem:boring} to obtain the desired function $F_2$.
\end{proof}
\subsubsection{Universal structure}
We have established that the tree of partial types is an $(\S,\M)$-tree. In order to finish the proof of Theorem~\ref{thm:zucker}, we will now construct three structures:
\begin{enumerate}
	\item The $L^+$-structure $\pstr{G}$ will be a ``tree of amalgamations'' of all countable partial structures whose $L$-reducts have ordered age $\KF$.
	      This will serve as the object on which we apply our Ramsey theorem, and we will show that functions in $\M$ naturally induce (special) embeddings $\pstr{G}\to\pstr{G}$.
	\item $\pstr{G}$ is not enumerated (and hence it is not a partial structure), whereas Theorem~\ref{thm:zucker} talks about enumerated structures. We will thus fix a suitable enumeration of $\pstr{G}$ and get an enumerated $L$-structure $\Gp$ with ordered age $\KF$, and an isomorphism $\embGG \colon \str G \to \Gp $.
	\item The final step of the proof is to show that envelopes are bounded (that is, there exists a function $b\colon \omega\to\omega$ such that the height of an envelope of set $X$ is at most $b(|X|)$).  This is not true for arbitrary subsets of $G$, but similarly as in~\cite{zucker2020} or \cite{Hubicka2020uniform} this can be established for special subsets. We thus define an enumerated structure $\str{H}$ created from any given $\KF$-universal structure $\str{K}$ by adding extra vertices and relations.  We will show that if $\str{H}$ is embedded to $\str{G}$ in the natural way then the sizes of envelopes of subsets of vertices of this copy are bounded by a function which depends only on the size of the subset. Having such embeddings $\str{K}\to\str{H}\to\str{G}$ and $\str{G}\to\str{K}$ will let us apply Proposition~\ref{prop:typeramsey} and obtain the desired upper bound on big Ramsey degrees of $\str{K}$.
\end{enumerate}
We start with an explicit definition of $\pstr{G}$:
\begin{definition}
	Construct $L^+$-structure $\pstr{G}$ as follows:
	\begin{enumerate}
		\item Its vertex set is $$G=\{\pstr{A}:\pstr{A}\hbox { is a partial structure, } \str{A}\in \KF\hbox{ and }|A|>0\}.$$
		\item Given $\pstr{A}\subseteq \pstr{B}\in G$ put $a=|A|-1$ and $b=|B|-1$ and add tuples to relations of $\pstr{G}$
		      so that $$\pstr{G}\restriction_{\{\pstr{A}, \pstr{B}\}}=(\pstr{B}\restriction_{\{a,b\}})(a\mapsto\pstr{A},b\mapsto\pstr{B}).$$
		\item There are no other tuples in relations of $\type{G}$ than what is required above.

	\end{enumerate}
\end{definition}
One can check that this construction gives precisely the $L^+$-structure created by the following steps:
\begin{enumerate}
	\item Create the disjoint union of all partial structures in $G$.
	\item Identify vertex $u$ of partial structure $\pstr{M}$ with vertex of the same index $u$ of partial structure $\pstr{N}$ if an only if $\pstr{M}\restriction_{u+1} = \pstr{N}\restriction_{u+1}$.
\end{enumerate}

Note that $(G,\subseteq)$ is a tree. The fact that we have a tree order on $\type{G}$ is important. However, Theorem~\ref{thm:zucker} talks about enumerated structures, that is, structures whose vertex set is an ordinal. Fix an injective function $\embGG\colon G\to\omega$ such that for every $\pstr{A},\pstr{B}$ with $|\pstr{A}|<|\pstr{B}|$ we also have $\embGG(\pstr{A})<\embGG(\pstr{B})$. Let $\Gp$ be the enumerated $L$-structure such that $\embGG\colon \str{G}\to \Gp$ is an isomorphism.
Since both trees $(G,\subseteq)$ and $(\KFpt,\subseteq)$ are used in following discussion, we will always use $\pstr{A}$, $\pstr{B}$, $\pstr{C}$ for nodes of  $(G,\subseteq)$ since these are partial structures, and $\type{T}$, $\type{S}$, $\type{Q}$, $\type{O}$ for nodes of $(\KFpt,\subseteq)$ since these are partial types.
\begin{lemma}
	\label{lem:KFage}
	$\OAge(\Gp)\subseteq \KF$.
\end{lemma}
\begin{proof}
	Assume for a contradiction that there exists $\str{F}\in \mathcal F$ with an ordered embedding $f\colon \str{F}\to\Gp$.
	Put $\pstr{A}=\embGG^{-1}[\max f[F]]$. By the construction of $\pstr{G}$, and because $\str{F}$ is irreducible, for every $\pstr{B}\in \embGG^{-1}[f[F]]$ we have
	$\pstr{B}\subseteq \pstr{A}$.
	From the way we constructed tuples of relations of $\pstr{G}$ it follows that there is also an ordered embedding $\str{F}\to\str{A}$. A contradiction.
\end{proof}

We associate every pair $(\type{S},\type{T})$ satisfying $\type{S}\subsetneq \type{T}\in \KFpt$ with a vertex of $\pstr{G}$ by putting $v(\type{S},\type{T})=\type{T}\restriction_{\ell(\type{S})+1}$.
\begin{observation}
	\label{obs:embedding}
	For every $F\in \M$ it holds that $$v(\type{S},\type{T})=v(\type{Q},\type{O}) \iff v(F(\type{S}),F(\type{T}))=v(F(\type{Q}),F(\type{O})).$$
	Every $F\in \M$ thus yields an embedding $\embF\colon \pstr G\to \pstr G$ defined by putting $\embF(\pstr{A})=v(F(\type{T}),F(\type{Q}))$ for every $\pstr{A}\in G$ and some $\type{T}\subset \type{Q}\in \KFpt$ satisfying $v(\type{T},\type{Q})=\pstr{A}$.
\end{observation}
Now we associate substructures of $\str{G}$ with envelopes in the $(\S,\M)$-tree $(\KFpt,\allowbreak {\subseteq},\allowbreak \Sigma,\allowbreak \S,\allowbreak \M)$.
\begin{definition}
	\label{def:Genvelope}
	Let $X\subset G$ be a finite set which forms a chain in $\subseteq$, and let $\pstr{C}\in G$ be a partial structure   satisfying $\max X\subseteq \pstr{C}$ and $\lvert\max X\rvert<\freef{C}(|C|-1)$.
	Put $\type{O}=\typev_\pstr{C}(|C|-1)$.
	Let $$Y=\{\type{O}\}\cup\{\type{O}|_{|B|-1}:\pstr{B}\in X\}.$$
	Let $(F, I)$ be the pair produced by Algorithm~\ref{envelope} for $Y$. We call $F$ an \emph{envelope of $X$ in $\pstr{C}$} and denote it by $F(X,\pstr{C})$. We call $I$ the \emph{interesting levels of $X$ in $\pstr{C}$} and denote this set by $I(X,\pstr{C})$. $|I(X,\pstr{C})|$ is thus the height of the envelope.
\end{definition}
Intuitively, one can think of $\pstr{C}$ as an initial segment of a countable enumerated structure in which members of $X$ lie.
Notice that for every $\pstr{B}\in X$ we have $v(\type{O}|_{|B|-1},\type{O})=\type{B}$. It follows that the envelope $F$ of $Y$
constructed in Definition~\ref{def:Genvelope} has the property that $X\subseteq \embF(G)$
and thus can be also considered to be an envelope of $X$ in the tree $(G,\subseteq)$.

We will devote the rest of this section to a proof of the following proposition.
\begin{prop}
	\label{prob:upperbound}
	For every enumerated structure $\str{K}$ such that $\OAge(\str{K})\subseteq \KF$ there exists an embedding $\embKG\colon\str{K}\to\str{G}$ such that $\embGG\circ \embKG$ is an ordered embedding and for every finite $\str{A}\subseteq \str{K}$ there exists $b(\str{A})\in \omega$ such that $|I(\embKG[A'])|\leq b(\str{A})$  for every $\str{A}'\subseteq \str K$ isomorphic to $\str A$.
\end{prop}

For the rest of section we fix $\str{K}$ a let $k$ be the maximal size of a structure in $\mathcal F$, or, if $\mathcal F$ is empty, we put $k=1$.
\begin{definition}
	Let $\iota\colon \omega\times (k+1)\to\omega$ be the function defined by $$\iota(i,n)=2(k+1)i+2n$$ for every $i\in \omega$ and $n\leq k$.
	By $\pstr{H}$ we denote the countable enumerated partial structure constructed from $\str{K}$ as follows:
	\begin{enumerate}
		\item For every $i<j\in \omega$ and every $m\leq n\leq k$ such that either $m=n=k$ or $m<n$, add precisely those tuples to relations of $\pstr{H}$ such that $$(\str{H}\restriction_{\{\ind{i}{m},\ind{j}{n}\}})(\ind{i}{m}\mapsto i, \ind{j}{n}\mapsto j)\subseteq \str{K}.$$
		      Also add $(\ell,\ind{j}{n})$ for every $\ell\leq \ind{i}{m}$ to $E_\pstr{H}$.
		\item There are no tuples in relations of $\pstr{H}$ other than the ones required by the conditions above.
		      In particular, every odd vertex is fully isolated.
	\end{enumerate}
\end{definition}
Every vertex with an odd index is called a \emph{fake} vertex.
Given $m\leq k$ we call every vertex of the form $\ind{i}{m}$, $i\in \omega$, a \emph{vertex of generation $m$};
we will also consider all fake vertices to be of generation 0.
Generation $k$ is special and holds a copy of $\str{K}$. Lower generation vertices will be later used
to limit the height of envelopes.

Let $\embKH \colon \omega\to\omega$ be the function defined by $$\embKH(i)=\ind{i}{k}$$ for every $i\in \omega$.
The following can be checked directly from the construction of $\pstr{H}$:
\begin{observation}\label{obs:H}
	$\pstr{H}$ is a well defined partial structure and $\embKH$ is an ordered embedding $\str{K}\to\str{H}$.
	Moreover, for every $v\in \pstr{H}$ satisfying at least one of:
	\begin{enumerate}
		\item $v<2(k+1)$, or
		\item $v$ is vertex of generation 0 (including fake vertices),
	\end{enumerate}
	it holds that $\freef{H}(v)=0$.
	For every other $v\in \pstr{H}$ it holds that $\freef{H}(v)$ is odd.
\end{observation}
Let $\embHG\colon \pstr{H} \to \pstr{G}$ be the embedding defined by $$\embHG(v)=\pstr{H}|_{v+1}.$$
Notice that $\embGG\circ \embHG\circ \embKH$ is an ordered embedding $\str K\to \Gp$.

It remains to analyze the behaviour of Algorithm~\ref{envelope} on finite subsets of $\embHG\circ \embKH[K]$.
Fix a finite set $X'\subseteq K$ and put $X=\embHG\circ \embKH[X']$ and $\pstr{C}=\embHG\circ \embKH(\max X'+1)$.
Let $\str{O}$ and $Y$ be constructed as in Definition~\ref{def:Genvelope}.
\begin{observation}\label{obs:env1}
	The following is true for closures of $Y$:
	\begin{enumerate}
		\item For every $I\subseteq \omega$ it holds that $\Cl_I(Y)\subseteq \Cl_\omega(Y)$.
		\item For every $\type{T}\in \Cl_\omega(Y)$ there exists $v\in C$ and $\ell\leq v$ such that $\type{T}=(\typev_\pstr{C}(v))|_\ell$.
		      Consequently, every $\type{T}\in \Cl_\omega(Y)$ satisfies $\type{T}\setminus \{t\}\subseteq \pstr{C}$.
		\item If set $Z\subseteq G$ satisfies $\type{T}\setminus \{t\}\subseteq \pstr{C}$ for every $\type{T}\in Z$
		      then for every $I\subseteq \omega$ it holds that $Z$ is parameter-closed over $I$ if and only if
		      for every $i\in I$ with $\freef{C}(i)>0$ it holds that $\typev_\pstr{C}(i)\in Z$.
	\end{enumerate}
\end{observation}
The iterative nature of Algorithm~\ref{envelope} makes its analysis complicated, as every time some level is deemed interesting,
the parameter closure is extended by the corresponding parameters needed to construct the successors of that level, which in turn may cause
other levels to become interesting. In general, newly introduced parameters may introduce new meets and new parameters and this process
may lead to unbounded envelopes.  We will see that the construction of $\pstr{H}$ makes this analysis easier.

It is a consequence of Observation~\ref{obs:env1} that for every interesting
level $i$ the parameter-closure adds no new type if $\freef{C}(i)=0$, and only one type
$\type{T}=\typev_\pstr{C}(i)$ if $\freef{C}(i)\neq 0$. By Observation
\ref{obs:H} we know that $\ell(\type{T})=\freef{C}(i)$ is odd.  Because
$\type{T}$ becomes part of the parameter closure, level $\ell(\type{T})$
will be later considered interesting by \ref{item:I1}.  Since it corresponds to a
fake vertex, it will not further extend the parameter closure. Consequently, we only need to be concerned about levels considered interesting by \ref{item:I2} and \ref{item:I3}.
\begin{remark}
	The discussion above shows that, with the exception of level $0$, the envelopes after introduction of fake vertices behave identically to ones characterised by Zucker~\cite[Section 4]{zucker2020}.
	These envelopes allow for the characterization of exact big Ramsey degrees in~\cite{Balko2021exact} (refining upper bounds from \cite[Section 5]{zucker2020} which gives finite but not necessarily optimal big Ramsey degrees).  It is thus possible to also refer to arguments in \cite[Section 5]{zucker2020}, however
	for completeness we finish the analysis based on our construction of $\pstr{H}$. This approach is more similar to the argument used in~\cite{Hubicka2020uniform}
	to give a finite upper bound on big Ramsey degrees of 3-uniform hypergraphs.
\end{remark}

As a consequence of Observation~\ref{obs:env1}, we can assign every $\type{T}\in \Cl_\omega(Y)$ an ``original'' vertex $o(\type{T})\in \pstr{C}$ satisfying $\type{T}=\typev_\pstr{C}(o(\type{T}))|_{\ell(\type{T})}$.
If the choice of such a vertex is not unique, we choose one minimizing its generation.
\begin{lemma}
	\label{lem:meets}
	For every $\type{T},\type{S}\in \Cl_\omega(Y)$, at least one of the following is satisfied:
	\begin{enumerate}
		\item $\type{T}\meet\type{S}\in \{\type{T},\type{S}\}$
		\item $\ell(\type{T}\meet\type{S})$ is of generation 0
		\item $\ell(\type{T}\meet\type{S})$ is a vertex of generation strictly lower than the minimum of the generation of $o(\type{T})$ and the generation of $o(\type{S})$.
	\end{enumerate}
\end{lemma}
\begin{proof}
	Fix $\type{T},\type{S}\in \Cl_\omega(Y)$ such that $\type{T}\meet\type{S}\notin \{\type{T},\type{S}\}$.  Without loss of generality assume that $\ell(\type{T})\leq \ell(\type{S})$.
	Let $u$ be maximal such that $\iota(u,0)\leq\ell(\type{T})$.  By Observation~\ref{obs:env1} we have that $\type{T}\setminus \{t\}\subseteq \type{S}\setminus \{t\}$.

	First consider the case that $o(\type{T})$ and $o(\type{S})$ are not fake and
	let $i,j\in \omega$, $m,n<k$ satisfy
	$$\begin{aligned}o(\type{T})&=\iota (j,n)\\ o(\type{S})&=\iota (i,m).\end{aligned}$$

	If $n=m$ (i.e., the generations of original vertices are the same) then, as $\type{T}\meet\type{S}\notin \{\type{T},\type{S}\}$, we know that $n>0$ and there exists $v<u$
	such that  $\str{K}\restriction_{\{v,j\}}(j\mapsto i)\neq \str{K}\restriction_{\{v,i\}}$.
	If $v$ is the minimal vertex with this property then by the construction of $\pstr{H}$ it holds that $\ell(\type{T}\meet\type{S})=\iota(v,0)$ and thus is of generation 0.

	If $n\neq m$  then, since $\type{T}\meet\type{S}\notin \{\type{T},\type{S}\}$, there exists vertex
	$v\leq u$ such that at least one of the following holds:
	\begin{enumerate}
		\item $n>0$, $u\neq v$ and  $\str{H}\restriction_{\{v,j\}}$ is irreducible, or
		\item $m>0$ and $\str{H}\restriction_{\{v,i\}}$ is irreducible.
	\end{enumerate}
	If $v$ is the minimal vertex of this property then by construction of $\pstr{H}$ it holds that  $$\ell(\type{T}\meet\type{S})=\iota(v,\max(\min(m,n)-1,0))).$$

	Finally, observe that since $\type{T}\meet\type{S}\notin \{\type{T},\type{S}\}$, at most one of $\type{T}$ and $\type{S}$ can be fake.
	In the case one of them is fake, we can proceed similarly as in the previous case considering it to be of generation 0.
\end{proof}
Levels which will be interesting by condition \ref{item:I3} are only those failing the assumptions of the boring extension lemma (Lemma~\ref{lem:boring})
for function $f_i$ constructed in condition \ref{item:I3}.
For this function, assumptions \ref{item:boringB1} and \ref{item:boringB2} will always be satisfied
(\ref{item:boringB1} is clear from the choice of $f_i$ while \ref{item:boringB2} follows by Observation~\ref{obs:env1}).
Consequently, if $\mathcal{F}=\emptyset$ then assumption~\ref{item:boringage} becomes vacuously true and the condition \ref{item:I3} will trigger only for level 0 (because of
our restriction that functions in $\M$ fix level 0).

For empty $\mathcal F$ we can thus finish our analysis:  By \ref{item:I1}, we will get interesting levels for every $\ell(\type{T})$, $\type{T}\in Y$, namely $|\pstr{C}|-1$ and $|\pstr{B}|-1$ for every $B\in X$.
These levels will introduce new parameters, which correspond to types of vertices of $X$ in $\pstr{H}$ (at most $|X|$ new types in total).  Levels corresponding to these new types and meets will introduce
no new types in the type-closure and thus the set of interesting levels is bounded from above by $2|X|+2$:
\begin{enumerate}
	\item $|X|+1$ levels will be interesting because they contain $Y$.
	      Each such level is generation $k$ and will introduce at most one new parameter to the parameter-closure.
	      (These are types of vertices of $X$, similarly as in various other bounds on big Ramsey degrees.)
	\item Vertices of $Y$ form a chain and thus contribute no new interesting levels for meets.
	      However, types introduced to the parameter-closure will introduce at most $|X|$ levels containing meets.
	\item Level 0 can never be skipped and will thus be considered interesting.
\end{enumerate}

Finally, we analyze which levels may be considered interesting by \ref{item:I3}.
Except for level $0$ which is always interesting by our choice of $\M$, every other level $\ell$ considered interesting by \ref{item:I3} has an associated structure $\pstr{A}$ (given by the failure assumption of Lemma~\ref{lem:boring}) with an ordered embedding $e\colon\str{F}\to\str{A}$ for some $\str{F}\in \mathcal F$.
Let $f$ be a function with domain $F$ defined by
$$f(i)=\begin{cases}
		\emptyset                      & \hbox{if $e(i)\leq \ell$} \\
		o(\typev_\pstr{A}(e(i))|_\ell) & \hbox{if $e(i)> \ell$.}
	\end{cases}
$$
We call the pair $(\str{F},f)$ a \emph{signature of interesting level $\ell$}.
Signatures are not necessarily unique, however:
\begin{observation}
	If two levels $\ell,\ell'>0$ are considered interesting by \ref{item:I3},
	$(\str{F},f)$ is a signature of level $\ell$, and $(\str{F}',f')$ is a signature of level $\ell'$ then $(\str{F},f)\neq (\str{F}',f')$.
\end{observation}
\begin{lemma}
	\label{lem:signatures}
	Let $\ell>0$ be interesting by \ref{item:I3} and let $(\str{F},f)$ be its signature.
	Let level $\ell$ correspond to a generation $n$ vertex.  Then for every type $\type{T} \in f[F]$ we have that $o(\type{T})$ is of generation at least $n+1$.
\end{lemma}
\begin{proof}
	This follows from the construction of $\str{H}$ which ensures that the lexicographically first copy of every irreducible structure $\str{F}$
	of size at most $k$ uses vertices of generations $0,1,\ldots |F|-1$.
\end{proof}

\begin{proof}[Proof of Proposition~\ref{prob:upperbound}]
	Fix $\str{A}$.  We prove an upper bound for the number of vertices of every generation in $I(\embKG[A'])$ for every $\str{A}'\in \binom{\str{K}}{\str{A}}$.
	Every level in $I(\embKG[A'])$ was considered interesting by \ref{item:I1}, \ref{item:I2} or \ref{item:I3}.
	By Lemmas~\ref{lem:meets} and \ref{lem:signatures} we get that only vertices of generation $k$ are considered interesting by \ref{item:I1} which is at most $|A|+1$ vertices.
	Vertices of generation $k-1$ can be considered interesting by \ref{item:I2} or \ref{item:I3}. There are only boundedly many meets of vertices of generation $k$,
	and by Lemma~\ref{lem:signatures} also boundedly many signatures using vertices of generation $k$ yielding a bound on the number of vertices of generation $k-1$.
	We can proceed similarly for lower generations.
\end{proof}
\begin{proof}[Proof of Theorem~\ref{thm:zucker}]
	Let $\mathcal F$ be a finite family of finite enumerated irreducible $L$-structures and let $\str{K}$ be a $\KF$-universal structure.
	Fix a finite enumerated structure $\str{A}$,
	an ordered embedding $\embGpK\colon \Gp\to\str{K}$ (which exists by Lemma~\ref{lem:KFage} and by the $\KF$-universality of $\str{K}$), and a
	finite colouring $\chi$ of $\OEmb(\str{A},\str{K})$. Define a colouring $\chi'$ of $\OEmb(\str{A},\str{G})$ by putting $$\chi'(\str{A})=\chi(\embGpK\circ \embGG(\str{A})).$$
	Let $b=b(\str{A})$ be the bound given by Proposition~\ref{prob:upperbound}.  For every $$e\in \OEmb(\str{A},\str{G}\restriction_{\KFpt({\leq}b)}),$$ we obtain a colouring $\chi''_e$ of $\AM_0^b$
	defined by $$\chi''_e(f)=\chi'(f\circ e(\str{A})).$$
	Notice that by Observation~\ref{obs:embedding} this is well defined.
	Because $\OEmb(\str{A},\str{G}\restriction_{\KFpt({\leq}b)})$ is finite, by a repeated application of Theorem~\ref{thm:shaperamseyN} we obtain $F\in \M$ such that all the $\chi_e''$ colourings are monochromatic. Let $\embKG:\str{K}\to\str{G}$ be the embeddng given by Proposition~\ref{prob:upperbound}.
	By Observation~\ref{obs:embedding} we know that $F\circ \embKG$ is a copy of $\str{K}$ in $\str{G}$ where the number of colours of copies of $\str{A}$ is bounded by $|\OEmb(\str{A},\str{G}\restriction \KFpt({\leq}b)|$.
	Now, the embedding $\embGpK\circ \embGG \circ F\circ \embKG$ is the desired ordered embedding $\str{K}\to\str{K}$.
	The upper bound on the big Ramsey degree of $\str{A}$ in $\str{K}$ is thus  $|\OEmb(\str{A},\str{G}\restriction \KFpt({\leq}b)|$.
\end{proof}

\section*{Acknowledgement}
The authors would like to thank to Sam Braunfeld, Jamal Kawach, S{\l}awomir Solecki, Stevo Todorcevic, and Spencer Unger for fruitful discussions.
Part of the work on this paper was done during the Thematic Program on Set Theoretic Methods in Algebra, Dynamics and Geometry
held at Fields Institute, Toronto in 2023.

D.~Ch., J.~H., M.~K. and J.~N. were supported by the project 21--10775S of  the  Czech  Science Foundation (GA\v CR).
M.~B was supported by the Center for Foundations of Modern Computer Science (Charles University project UNCE/SCI/004).
This article is part of a project that has received funding from the European Research Council (ERC) under the European Union's Horizon 2020 research and innovation programme (grant agreement No 810115).
N.~D.\ is supported by National Science Foundation grants DMS--1901753 and DMS--2300896. A.~Z.\ is supported by NSERC grants RGPIN--2023--03269 and DGECR--2023--00412.
\bibliographystyle{alpha}
\bibliography{ramsey.bib}
\end{document}